\documentclass[a4paper,11pt,reqno,english]{smfart}

\usepackage{enumerate}
\usepackage{amssymb,amsmath,latexsym,amsthm}
\usepackage[T1]{fontenc}
\usepackage[a4paper, tmargin=1in, bmargin=1.1in, lmargin=1in,rmargin=1in]{geometry}
\usepackage{url}
\usepackage[main=english]{babel}
\usepackage[utf8]{inputenc}
\usepackage{mathrsfs}
\usepackage{xcolor}
\usepackage{bigints}
\usepackage{comment}
\definecolor{violet}{rgb}{0.0,0.2,0.7}
\definecolor{rouge2}{rgb}{0.8,0.0,0.2}
\usepackage{tikz}
\usepackage{empheq}
\usepackage{tikz-cd}
\usetikzlibrary{matrix,arrows,decorations.pathmorphing}
\usepackage{hyperref}
\usepackage{mathpazo} 
\hypersetup{
	bookmarks=true,         
	unicode=false,          
	pdftoolbar=true,        
	pdfmenubar=true,        
	pdffitwindow=false,     
	pdfstartview={FitH},    
	pdftitle={},    
	pdfauthor={},     
	colorlinks=true,       
	linkcolor=rouge2,          
	citecolor=violet,        
	filecolor=black,      
	urlcolor=cyan}           
\usepackage{enumitem}
\usepackage{appendix}

\usepackage{mathtools}
\usepackage{dsfont}

\theoremstyle{plain}
    \newtheorem{thm}{Theorem}[section]
	
	\newtheorem{lem}[thm]{Lemma}
	\newtheorem{prop}[thm]{Proposition}
	
	\newtheorem{cor}[thm]{Corollary}
	
\theoremstyle{plain}
	\newtheorem{bigthm}{Theorem}

	\newtheorem{taggedbigsetx}{Setting} 
	\newenvironment{taggedbigset}[1]
    {\renewcommand\thetaggedbigsetx{#1}\taggedbigsetx}
    {\endtaggedbigsetx}

    \newtheorem*{bigrmk*}{Remark}

\theoremstyle{definition}
	\newtheorem{defn}[thm]{Definition}
    \newtheorem*{defn*}{Definition}

    \newtheorem*{conj*}{Conjecture}
	\newtheorem*{claim*}{Claim}
        \newtheorem{claim}[thm]{Claim}
	\newtheorem*{ack*}{Acknowledgements}
\theoremstyle{remark}
	\newtheorem{rmk}[thm]{Remark}
	\newtheorem*{rmk*}{Remark}

	\newtheorem*{ques*}{Question}
	\newtheorem*{ans*}{Answer}

\numberwithin{equation}{section}
\renewcommand{\labelenumi}{(\roman{enumi})}
\newlist{steps}{enumerate}{1}
\setlist[steps, 1]{label = Step \arabic*:}

\mathcode`l="8000
\begingroup
\makeatletter
\lccode`\~=`\l
\DeclareMathSymbol{\lsb@l}{\mathalpha}{letters}{`l}
\lowercase{\gdef~{\ifnum\the\mathgroup=\m@ne \ell \else \lsb@l \fi}}%
\endgroup

\DeclareFontFamily{U}{MnSymbolC}{}
\DeclareSymbolFont{MnSyC}{U}{MnSymbolC}{m}{n}
\DeclareFontShape{U}{MnSymbolC}{m}{n}{
	<-6>  MnSymbolC5
	<6-7>  MnSymbolC6
	<7-8>  MnSymbolC7
	<8-9>  MnSymbolC8
	<9-10> MnSymbolC9
	<10-12> MnSymbolC10
	<12->   MnSymbolC12}{}
\DeclareMathSymbol{\intprod}{\mathbin}{MnSyC}{'270}
\DeclareMathOperator{\Id}{Id}

\DeclareMathOperator{\im}{im}
\DeclareMathOperator{\tr}{tr}
\DeclareMathOperator{\pr}{pr}
\DeclareMathOperator{\Vol}{Vol}
\DeclareMathOperator{\Ric}{Ric}

\DeclareMathOperator{\supp}{supp}

\DeclareMathOperator{\dist}{dist}

\DeclareMathOperator{\PH}{PH}

\DeclareMathOperator{\PSH}{PSH}

\DeclareMathOperator{\CAP}{Cap}

\DeclareMathOperator*{\osc}{osc}

\DeclareMathOperator{\Exc}{Exc}

\DeclareMathOperator{\Aut}{Aut}

\DeclareMathOperator{\Bisec}{Bisec}
\DeclareMathOperator{\Scal}{S}

\def\1{\mathds{1}}

\def\E{\mathbf{E}}
\def\H{\mathbf{H}}
\def\I{\mathbf{I}}

\def\L{\mathbf{L}}
\def\M{\mathbf{M}}

\newcommand{\ii}{\mathrm{i}}

\newcommand{\loc}{\mathrm{loc}}

\newcommand{\nmlz}{{\mathrm{norm}}}
\newcommand{\fibre}{{\mathrm{fibre}}}

\newcommand{\wF}{{\widetilde{F}}}

\newcommand{\wM}{{\widetilde{M}}}

\newcommand{\wom}{{\widetilde{\omega}}}

\newcommand{\wu}{{\widetilde{u}}}
\newcommand{\wv}{{\widetilde{v}}}
\newcommand{\ww}{{\widetilde{w}}}

\def\wE{\widetilde{\mathbf{E}}}
\def\wH{\widetilde{\mathbf{H}}}

\def\wM{\widetilde{\mathbf{M}}}

\newcommand{\bBD}{\overline{\mathbb{D}}}

\newcommand{\hg}{{\widehat{g}}}

\newcommand{\homg}{{\widehat{\omega}}}

\newcommand{\hR}{{\widehat{R}}}

\newcommand\sm{\sigma}
\newcommand\dt{\delta}

\newcommand\vep{\varepsilon}
\newcommand\vph{\varphi}

\newcommand\om{\omega}
\newcommand\ta{\theta}
\newcommand\gm{\gamma}

\newcommand\af{\alpha}
\newcommand\bt{\beta}
\newcommand\ld{\lambda}
\newcommand\zt{\zeta}

\newcommand\Dt{\Delta}
\newcommand\Om{\Omega}

\newcommand\Ta{\Theta}

\newcommand\BN{\mathbb{N}}
\newcommand\BZ{\mathbb{Z}}
\newcommand\BQ{\mathbb{Q}}
\newcommand\BR{\mathbb{R}}
\newcommand\BC{\mathbb{C}}

\newcommand\BD{\mathbb{D}}
\newcommand\BP{\mathbb{P}}

\newcommand\CC{\mathcal{C}}

\newcommand\CE{\mathcal{E}}

\newcommand\CI{\mathcal{I}}
\newcommand\CK{\mathcal{K}}

\newcommand\CO{\mathcal{O}}

\newcommand\CU{\mathcal{U}}

\newcommand\CW{\mathcal{W}}
\newcommand\CX{\mathcal{X}}
\newcommand\CY{\mathcal{Y}}
\newcommand\CZ{\mathcal{Z}}

\newcommand\lt{\left}
\newcommand\rt{\right}

\newcommand\pl{\partial}
\newcommand\db{\bar{\partial}}

\newcommand\dd{\mathrm{d}}
\newcommand\dc{\mathrm{d}^{\mathrm{c}}}
\newcommand\ddc{\mathrm{d}\mathrm{d}^{\mathrm{c}}}

\newcommand\norm[1]{\left\lVert {#1} \right\rVert}

\newcommand\abs[1]{\left\lvert {#1} \right\rvert}

\newcommand\w{\wedge}

\newcommand\reg{\mathrm{reg}}
\newcommand\sing{\mathrm{sing}}

\newcommand\set[2]{\left\{ {#1} \, \middle| \, {#2} \right\}}
\newcommand\iprod[2]{\left\langle {#1}, {#2} \right\rangle}

\newcommand{\RN}[1]{\textup{\uppercase\expandafter{\romannumeral#1}}}

\makeatletter
\newsavebox{\@brx}
\newcommand{\llangle}[1][]{\savebox{\@brx}{\(\m@th{#1\langle}\)}%
  \mathopen{\copy\@brx\kern-0.5\wd\@brx\usebox{\@brx}}}
\newcommand{\rrangle}[1][]{\savebox{\@brx}{\(\m@th{#1\rangle}\)}%
  \mathclose{\copy\@brx\kern-0.5\wd\@brx\usebox{\@brx}}}
\makeatother

\title{Singular cscK metrics on smoothable  varieties}

\author{Chung-Ming Pan}
\address[Chung-Ming Pan]{Centre interuniversitaire de recherches en g\'eom\'etrie et topologies (CIRGET); Universit\'e du Qu\'ebec \`a Montr\'eal; Case postale 8888, Succursale centre-ville, Montr\'eal, Qu\'ebec, H3C 3P8, Canada \qquad\qquad\qquad\qquad\qquad\qquad\qquad\qquad\qquad\qquad\qquad\qquad\qquad\qquad\qquad\qquad\qquad\qquad\qquad\qquad\qquad\qquad\qquad}
\email{\href{mailto:pan.chung_ming@uqam.ca}{pan.chung\_ming@uqam.ca} \qquad\qquad\qquad\qquad\qquad\qquad\qquad\qquad}
\urladdr{\href{https://chungmingpan.github.io/}{https://chungmingpan.github.io/}}

\author{Tat Dat T\^o}
\address[Tat Dat T\^o]{Institut de Math\'ematiques de Jussieu-Paris Rive Gauche; Sorbonne Universit\'e - Campus Pierre et Marie Curie, 4 place Jussieu, 75252 Paris Cedex 05, France \qquad\qquad\qquad\qquad\qquad\qquad\qquad\qquad\qquad\qquad\qquad}
\email{\href{mailto:tat-dat.to@imj-prg.fr}{tat-dat.to@imj-prg.fr} \qquad\qquad\qquad\qquad\qquad\qquad\qquad\qquad\qquad\qquad\qquad}
\urladdr{\href{https://sites.google.com/site/totatdatmath/}{https://sites.google.com/site/totatdatmath/}}

\author{Antonio Trusiani}
\address[Antonio Trusiani]{Chalmers University of Technology; Chalmers tv\"argata 3, 41296 G\"oteborg, Sweden \qquad\qquad\qquad\qquad\qquad\qquad\qquad\qquad\qquad\qquad\qquad}
\email{\href{mailto:trusiani@chalmers.se}{trusiani@chalmers.se} \qquad\qquad\qquad\qquad\qquad\qquad\qquad\qquad\qquad\qquad\qquad}
\urladdr{\href{https://sites.google.com/view/antonio-trusiani/}{https://sites.google.com/view/antonio-trusiani/}}

\date{\today}

\begin{document} 

\maketitle

\begin{abstract}
We prove the lower semi-continuity of the coercivity threshold of Mabuchi functional along a degenerate family of normal compact K\"ahler varieties with klt singularities. 
Moreover, we establish the existence of singular cscK metrics on $\BQ$-Gorenstein smoothable klt varieties when the Mabuchi functional is coercive, these arise as a limit of cscK metrics on close-by fibres.
The proof relies on developing a novel strong topology of pluripotential theory in families and establishing uniform estimates for cscK metrics. 
\end{abstract}

\tableofcontents

\section*{Introduction}
A central theme in complex geometry for decades has been the search for canonical metrics on K\"ahler manifolds.
In dimension one, Poincar\'e's uniformization theorem establishes the existence of metrics with constant Gaussian curvature on any compact Riemann surface.
Constant scalar curvature K\"ahler metrics (cscK) are natural generalizations in higher dimensions, garnering extensive attention in the literature (see the surveys \cite{Boucksom_2018, Donaldson_2018_2} and references therein for in-depth overviews, details, and numerous works on cscK metrics). 

A special instance of cscK metric is the K\"ahler--Einstein metrics, which has attracted the intensive focus on K\"ahler geometry. 
Landmark contributions in this area include Yau's resolution of Calabi's conjecture \cite{Yau_1978}, and the resolution of the Yau--Tian--Donaldson conjecture on Fano manifolds by Chen--Donaldson--Sun \cite{Chen_Donaldson_Sun_2015} (see also \cite{Tian_2015}).

The Yau--Tian--Donaldson (YTD) conjecture asserts that, on a polarized manifold $(X,L)$, the existence of cscK metrics in $c_1(L)$ is equivalent to an algebro-geometric notion called "K-stability". 
Recent significant advancements by Darvas--Rubinstein~\cite{Darvas_Rubinstein_2017}, Berman--Darvas--Lu~\cite{Berman_Darvas_Lu_2020} and Chen--Cheng~\cite{Chen_Cheng_2021_2} have established an analytic characterization on compact K\"ahler manifolds. 
Specifically, the existence of unique cscK metric in a K\"ahler class is equivalent to the coercivity of the Mabuchi functional. 
On a polarized manifold $(X,L)$, Boucksom--Hisamoto--Jonsson~\cite{Boucksom_Hisamoto_Jonsson_2019} demonstrated the implication of coercivity of Mabuchi functional to uniform K-stability (see \cite{Dervan_Ross_2017, Dervan_2018, SD_2018, SD_2020} for the related results in the transcendental setting). 
Conversely, C. Li~\cite{Chi_Li_2022_2} showed that the uniform K-stability for filtrations implies the coercivity of Mabuchi functional.
The remaining challenge in the (uniform) YTD conjecture lies in proving the uniform K-stability for filtrations derived from K-stability.

The central theme of this article focuses on the cscK problem on singular K\"ahler varieties. 
Singularities are prevalent in the classification theories of K\"ahler manifolds, such as the minimal model program (MMP) in birational geometry and moduli theory. 
In moduli theory, the consideration of singular varieties arises when compactifying moduli spaces of smooth manifolds. 
From the idea of moduli theory, it is interesting to study the behavior of canonical metrics moving in families. 

In the existing literature, singular K\"ahler--Einstein metrics have received comprehensive study on a fixed variety \cite{EGZ_2009, BBEGZ_2019, BBJ_2021, Li_Tian_Wang_2022, Chi_Li_2022} and their families (cf. \cite{Koiso_1983, Rong_Zhang_2011, Ruan_Zhang_2011, Spotti_Sun_Yao_2016, Li_Wang_Xu_2019, DGG2023, Pan_Trusiani_2023} and the references therein).
However, there are very few results concerning cscK metrics on singular varieties and their degenerate families. 

\smallskip
This article introduces a pluripotential theoretical approach to studying singular cscK metrics along degenerate families.  
Our results contribute to the following two aspects:
\begin{itemize}
    \item {\bf Stability of coercitivity for Mabuchi functionals:} 
    We establish the lower semi-continuity of coercitivity threshold along a degenerate family of normal compact K\"ahler varieties with klt singularities. 
    Precisely, we obtain a uniform coercivity with an almost optimal slope. 
    Our method also covers the situation in moving K\"ahler classes on a fixed variety.
    \item {\bf Existence of singular cscK metrics:} 
    Under the condition of coercivity of the Mabuchi functional, we prove the existence of singular cscK metrics on $\BQ$-Gorenstein smoothable normal compact K\"ahler varieties with klt singularities. 
\end{itemize}
Relating to the second point, we also provide a strong convergence of cscK potentials from the general fibres to the singular cscK metric on the central fibre.

\subsection*{Openness of coercivity for Mabuchi functional}
Before stating precisely our main results, we give here some basic definitions and introduce the context. 
In the sequel, by complex variety, we mean a irreducible reduced complex analytic space.  
Let $X$ be a normal compact K\"ahler variety and $\omega$ be a K\"ahler form on $X$. Denote by $\PSH(X, \omega)$ the set of all $\om$-psh functions which are not identically $-\infty$ (cf. Section \ref{sect_finite_energy}). 
Then it is natural to define {\sl singular} cscK metrics on $X$ as follows:
We say that $\om_\vph := \om + \ddc \vph$ is a singular cscK metric if $\om_\vph$ is a genuine cscK metric on $X^\reg$ and $\vph \in \PSH(X,\om) \cap L^\infty(X)$.

In the sequel, we shall consider a family $\pi: \CX \to \BD$ fulfills the following setting: 
Let $\CX$ be an $(n+1)$-dimensional, complex variety. 
Let $\pi: \CX \to \BD$ be a proper, surjective, holomorphic map such that each (schematic) fibre $X_t := \pi^{-1}(t)$ is a complex variety for any $t \in \BD$. 
Assume that $\om$ is a hermitian metric on $\CX$ such that $\om_t := \om_{|X_t}$ is K\"ahler for all $t \in \BD$.

\begin{bigthm}\label{bigthm:openness_coercivity}
Let $\pi: (\CX,\om) \to \BD$ be a family as above. 
Suppose that $\CX$ is $\BQ$-Gorenstein, and $X_0$ is normal with klt singularities. 
Then the coercivity threshold
\[
    \sm_t := \sup \set{A \in \BR}{\M_{\om_t} \geq A(-\E_{\om_t}) - B \text{ on $\CE^1_{\nmlz}(X_t,\om_t)$, for some $B \in \BR$}}
\]
is lower semi-continuous at $t=0$. 
In particular, if the Mabuchi functional $\M_{\omega_0}$ is coercive  on $X_0$, then there exists $r>0$ such that the Mabuchi functional $\M_{\omega_{t}}$ is coercive on $X_t$ for any $\lvert t\rvert\leq r$.
\end{bigthm}

In fact, we have obtained a uniform coercivity with an almost optimal slope, presenting a stronger result than Theorem~\ref{bigthm:openness_coercivity} (see Theorem~\ref{thm:Openness Coercivity} for more details).

Theorem~\ref{bigthm:openness_coercivity} strengthens evidence supporting the openness of (uniform) K-stability for general families of compact K\"ahler varieties with klt singularities. 
So far, such openness has been built only in the context of Fano varieties by \cite{Blum_Liu_2022, Blum_Liu_Xu_2022} where the authors proved the stability threshold, Fujita--Odaka's $\dt$-invariant \cite{Fujita_Odaka_2018}, is lower semi-continuous; however, such kind of notion does not readily extend to the general situation. 

On a smooth K\"ahler manifold, due to Berman--Darvas--Lu~\cite{Berman_Darvas_Lu_2020} and Chen--Cheng~\cite{Chen_Cheng_2021_2}, it is known that the existence of the unique cscK metric in a given K\"ahler class is equivalent to the coercivity of the corresponding Mabuchi functional of the class. 
Therefore, Theorem~\ref{bigthm:openness_coercivity} can be viewed as a generalization of the result of LeBrun and Simanca \cite[Thm.~5]{LeBrun_Simanca_1994} to singular families.
In complex dimension $2$, Biquard and Rollin \cite{Biquard_Rollin_2015} also showed openness condition on a $\BQ$-Gorenstein smoothing of normal surfaces which have an orbifold cscK class. 
The proof of Biquard--Rollin used in-depth the orbifold structure and a gluing technique from Arezzo--Pacard \cite{Arezzo_Pacard_2006, Arezzo_Pacard_2009}. 
Since higher-dimensional klt singularities are not always quotient singularities, we pursue the proof in a completely different fashion.

The main input of the article is the notion of strong convergence in families (see Section~\ref{sect_families}) and a relative version of pluripotential theory in families regarding this notion. 
This concept strengthens the $L^1$-convergence in families introduced in \cite{Pan_Trusiani_2023}. 
We establish crucial results along sequences that strongly converge in families. 
Notable outcomes include: 
\begin{itemize}
    \item strong compactness of potentials with uniformly bounded entropy (Theorem \ref{thm:strong_compact_fami}),
    \item lower semi-continuity of the entropy (Lemma \ref{lem:semi_conti_entropy}) with respect to the strong convergence in families, 
    \item lower semi-continuity of the Mabuchi functional (Proposition \ref{prop:Lower Semicontinuity Mabuchi families}) with respect to the strong convergence in families.
\end{itemize} 
These results are essential in proving Theorem~\ref{bigthm:openness_coercivity}. 
The main difficulty in the degenerate family context lies in the change of the underlying complex space (e.g. complex structures, topology of spaces), and the appearance of singularities. 
As the spaces of potentials change in the family setting, the proof on a fixed manifold or variety \cite{BBGZ_2013, BBEGZ_2019} cannot be applied directly. 
Particularly in establishing the strong compactness in families, a much more complicated approximation argument is needed. 

\smallskip
On the other hand, our method applies to prove the openness of coercivity for classes in the K\"ahler cone on a normal compact K\"ahler variety with klt singularities: 

\begin{bigthm}\label{bigthm:openness_classes}
Let $(X, \om)$ be a normal compact K\"ahler variety with klt singularities and let $\CK_X$ be the Kähler cone. 
Then the coercivity threshold {\small
\[
    \CK_X \ni [\om] \mapsto \sm_{\om} := \sup \set{A \in \BR}{\M_{\om} \geq A(-\E_{\om}) - B \text{ on $\CE^1_{\nmlz}(X,\om)$, for some $B \in \BR$}}
\] 
}%
is lower semi-continuous. 
In particular, if $\M_{\om}$ is coercive, then there is an open neighborhood $U\subset \CK_X$ of $[\om]$ such that for any K\"ahler form $\om'$ with $[\om'] \in U$, $\M_{\om'}$ is coercive.  
\end{bigthm}

Clearly, $\om$ in the definition of the coercivity threshold is an arbitrary Kähler form associated to the Kähler class considered. 
We refer to Section \ref{ssec:Proposition B} for the precise definition of the Kähler cone in the singular setup.

On a smooth K\"ahler manifold, combining again the correspondence of the existence of cscK metric and the coercivity of Mabuchi functional established in \cite{Berman_Darvas_Lu_2020, Chen_Cheng_2021_2}, Theorem~\ref{bigthm:openness_classes} recovers the famous result of LeBrun--Simanca \cite[Thm.~4]{LeBrun_Simanca_1994} on the openness of the existence of cscK metrics on the K\"ahler cone.
We also refer to the recent work of Boucksom--Jonsson \cite[Thm.~C]{Boucksom_Jonsson_2023} on the continuity of coercivity threshold with respect to the K\"ahler classes on smooth K\"ahler manifolds.

\subsection*{Singular cscK metrics on smoothable K\"ahler varieties}
Theorem \ref{bigthm:openness_coercivity} marks the first step toward the openness of singular cscK metrics on a family of singular varieties. 
The analytic characterization of cscK metrics and uniform YTD conjecture are still lacking in the singular setting. 
In the next result, we make progress on the analytic characterization when $(X_0,\om_0)$ admits a {\it $\BQ$-Gorenstein smoothing} $\pi: (\CX, \om) \to \BD$; namely, we have $\CX, \om, \pi$ as in Theorem~\ref{bigthm:openness_coercivity} and $X_t$'s are smooth K\"ahler manifolds for all $t \neq 0$.

\begin{bigthm}\label{bigthm:smoothable_variety}
Suppose that $(X_0,\om_0)$ is a compact K\"ahler variety with klt singularities. 
Assume that $(X_0,\om_0)$ admits a $\BQ$-Gorenstein smoothing $\pi: (\CX,\om) \to \BD$. 
If the Mabuchi functional  $\M_{\omega_0}$ is coercive on $X_0$, then up to shrinking $\BD$,
\begin{enumerate}
    \item for any $t \neq 0$, $\M_{\omega_t}$ is coercive on $X_t$, and $X_t$ admits a cscK metric in $[\om_t]$;
    \item $X_0$ admits a singular cscK metric in $[\om_0]$ constructed as a limit of cscK metric on nearby fibres. 
    Furthermore, its potential minimizes $\M_{\omega_0}$.
\end{enumerate}
\end{bigthm}

In Theorem \ref{bigthm:smoothable_variety}, we remark that in $(i)$, the coercivity of $\M_{\omega_t}$ comes from Theorem \ref{bigthm:openness_coercivity} and the existence of cscK metric in $ [\omega_t]$ by Chen--Cheng \cite{Chen_Cheng_2021_2}. 
In $(ii)$, we prove that the cscK potentials on smooth fibres converge strongly and smoothly in the family sense to the cscK potential on the central fibre.
By smooth and strong convergence of cscK potentials, we mean that the sequence converges smoothly outside the singular locus of $\pi$ and their energy also converges to the energy of the singular cscK potential on the central fibre. A direct application of our theorems (cf. Corollary \ref{cor:fano_cscK}) is the existence of singular cscK metrics that are not K\"ahler--Eisntein on certain smoothable K-stable Fano varieties. 

To achieve Theorem \ref{bigthm:smoothable_variety}, inspired by work of Chen and Cheng \cite{Chen_Cheng_2021_1, Chen_Cheng_2021_2}, we establish uniform estimates in families on the cscK potentials on nearby fibres.  
The first key point involves establishing a uniform bound on the entropy in families, referring to a family of metrics with canonical Monge--Amp\`ere densities. 
Then, we follow a new approach for Chen--Cheng's result \cite{Chen_Cheng_2021_1} provided by Guo-Phong \cite{Guo_Phong_2022} to obtain a uniform $L^\infty$-esitmate. 
For a higher-order estimate, in a degenerate situation, a non-trivial modification is required in the approach of Chen--Cheng \cite{Chen_Cheng_2021_1}, since the holomorphic bisectional curvatures of reference metrics are not uniformly bounded from below along the family. 

\subsection*{Organization of the article}
\begin{itemize}
    \item Section \ref{sect_Preliminaries} provide a recap of pluripotential theory on singular variety, finite energy spaces, and relevant functionals within the variational approach of the cscK problem;  
    
    \item In Section \ref{sect_families}, we introduce the notion of strong topology in families. 
    We also investigate conditions on the Monge--Amp\`ere densities to achieve strong convergence within this context; 

    \item Section \ref{sect_entropy_family} focuses on strong compactness in families for potentials with uniformly bounded entropy. 
    This proof also leads to establishing (semi-)continuity properties for entropy and twisted energy in the family contexts.
    
    \item Section \ref{sect_singular_cscK} initiates the concept of singular cscK metrics and the definition of the Mabuchi functional on a normal Kähler variety with klt singularities.
    We develop a variational approach for singular cscK metrics, exploring the strong lower semi-continuity property of the Mabuchi functional on fixed varieties and in families.
    Then we prove Theorem~\ref{bigthm:openness_coercivity} and Theorem~\ref{bigthm:openness_classes}.
    
    \item Section \ref{sect_smoothing} concentrates on proving Theorem \ref{bigthm:smoothable_variety}. 
    In this section, we establish uniform a priori estimates of cscK potentials on a family of $\BQ$-Gorenstein smoothing.
\end{itemize}

\begin{ack*}
The authors are grateful to Vincent Guedj and Henri Guenancia for their support and suggestions. 
The authors would like to thank Stefano Trapani for his helpful comments and careful reading of the first draft, and Cristiano Spotti, Zakarias Sj\"ostr\"om Dyrefelt and Theodoros Papazachariou for enlightening discussions. 
Special thanks to Abdellah Lahdili for the helpful conversation that led to identifying an imprecision in an earlier version. 
C.-M. Pan has partially benefited from research projects HERMETIC ANR-11-LABX-0040 and Karmapolis ANR-21-CE40-0010.
T. D. T\^o is partially supported by
ANR-21-CE40-0011-01 (research project MARGE).
A. Trusiani is partially supported by the Knut and Alice Wallenberg Foundation.
\end{ack*}

\section{Preliminaries}\label{sect_Preliminaries}
Let $X$ be an $n$-dimensional normal compact K\"ahler variety. 
By K\"ahler variety, we mean an irreducible reduced complex analytic space equipped with a K\"ahler form.
A {\sl K\"ahler form $\omega$} on $X$ is locally a restriction of a K\"ahler form defined near the image of a local embedding $j:X \underset{\loc.}{\hookrightarrow} \BC^N$. 

We use the notation $\dc := \frac{\ii}{4\pi} (\db - \pl)$ for the twisted exterior derivative.

\subsection{CscK metrics on smooth K\"ahler manifolds}\label{ssec:cscK smooth}

First of all, we consider $X$ to be smooth. 
A \emph{constant scalar curvature Kähler} (cscK) metric is a Kähler metric $\om$ whose scalar curvature
\[
    \Scal(\om) = \tr_{\om} \Ric(\om) = n\frac{\Ric(\om)\wedge \om^{n-1}}{\om^n}.
\]
is a constant $\bar{s}$ on $X$. 
Note that $\bar{s} = n\frac{c_1(X) \cdot [\om]^{n-1}}{[\om]^n}$ is a cohomological constant.

\subsubsection{The Mabuchi functional}
Given any K\"ahler class $\alpha$ with a reference K\"ahler metric $\omega\in \alpha$, the cscK problem is finding a cscK metric $\omega_u := \omega+\ddc u$ in the class $\alpha$. If such  a metric exists, it is a minimizer of the Mabuchi functional (cf. \cite{Mabuchi_85}) $\M$ characterized by 
\[
    \frac{\dd}{\dd t} \M(u_t) = -\int_X \dot{u}_t (\Scal(\omega_{u_t}) - \bar{s}) \frac{\omega^n_{u_t}}{V},
\]
for any path $(u_t)_t$ in $ \mathcal{H}_{\omega} := \set{u\in \CC^\infty(X)}{\omega_u = \omega + \ddc u > 0}$. 
The Chen--Tian formula gives a precise expression: 
\[
    \M(u)= \H(u) + \bar{s} \E(u) - n\E_{\Ric(\om)}(u), 
\]
for all $u \in \mathcal{H}_\omega$, where
{\small
\begin{align*}
    &\E(u) = \frac{1}{(n+1)V} \sum_{j=0}^n \int_X u \om_u^j \w \om^{n-j},\quad \E_{\eta}(u) = \frac{1}{n V} \sum_{j=0}^{n-1} \int_X u \eta \w \om_u^j \w \om^{n-1-j}, \\
     &\H(u):=\frac{1}{V} \int_X \log \lt(\frac{\om_u^n}{\om^n}\rt) \om_u^n, \quad \text{ where $\eta$ is a smooth $(1,1)$-form.}
\end{align*}}
Given  $u_0,u_1\in\mathcal{H}_{\omega}$, we define
$$
d_1(u_0, u_1):=\inf\left\{ \int_0^1 \int_X |\dot u_t|\omega_{u_t}^n dt \right\}, 
$$
where the infimum is taken for all smooth curves $u_t(x)\in \CC^\infty ([0,1]\times X)$ with $u_t\in\mathcal{H}_{\omega}$. 
Then we define the finite energy space $\big(\CE^1(X,\om),d_1\big)$ so that it is the completion of $(\mathcal{H}_{\omega}, d_1)$. We refer to the next section for details on $\big(\CE^1(X,\om),d_1\big)$. 
In particular, the Mabuchi functional extends to the whole space $\CE^1(X,\om)$ and it is lower semi-continuous with respect to the $d_1$-topology (cf. \cite{BBEGZ_2019}).

\subsubsection{Variational approach for cscK metrics}
In the last decade pluripotential techniques and a priori estimates have been used to prove the following characterization of the existence of cscK metrics.
The following theorem is a combination of {\cite[Thm.~2.10]{Darvas_Rubinstein_2017}, \cite[Thm.~1.4]{Berman_Darvas_Lu_2020}, \cite[Thm.~1.6]{Chen_Cheng_2021_2}}:
\begin{thm}\label{thm:Coerc Existence}
Fix a K\"ahler class $\af$. 
Let $\om \in \alpha$ be a reference Kähler form. 
The following are equivalent:
\begin{enumerate}
    \item There exists a unique cscK metric $\homg \in \alpha$;
    \item $\M$ is coercive; namely, there exist $A>0, B > 0$ such that
    $
        \M(u) \geq 
        A d_1(u,0) - B
    $
    for any $u\in\CE^1_\nmlz(X,\om) := \set{u \in \CE^1(X,\om)}{\sup_X u = 0}$.
\end{enumerate}
In this case, the cscK metric is the unique minimizer of $\M$.
\end{thm}

By \cite[Thm.~1.3]{Berman_Berndtsson_2017}, the uniqueness of cscK metrics in a fixed Kähler class holds modulo $\Aut(X)^\circ$, the connected component of $\Aut(X)$ containing the identity. 
In particular, Theorem~\ref{thm:Coerc Existence} is relevant only when $\Aut(X)^\circ=\{\Id\}$. 

\subsection{Pluripotential theory on singular K\"ahler varieties} 
Let $(X,\omega)$ be an $n$-dimensional normal compact K\"ahler variety. 
In this section, we review some definitions and results in pluripotential theory on K\"ahler varieties. 
We quickly recall that a smooth form $\af$ on $X$ is a smooth form on $X^\reg$ such that $\af$ extends smoothly under any local embedding $X \underset{\loc.}{\hookrightarrow} \BC^N$. 
One also has a similar notion for smooth hermitian metrics.

\subsubsection{Finite energy space} \label{sect_finite_energy}
A function $\phi: X\rightarrow [-\infty, +\infty)$ is  $\omega$-plurisubharmonic  ($\omega$-psh) if $\phi+ u$ is plurisubharmonic where $u$ is any local potential of $\omega$, i.e. $\phi+u$ is the restriction to $X$ of an psh function defined near $\im(j)$ as above (see \cite{Demailly_1985, EGZ_2009} for more details). 
By Bedford--Taylor's theory \cite{Bedford_Taylor_1982}, the Monge--Amp\`ere operator can be extended to bounded $\om$-psh functions on smooth manifolds. 
In the singular setting, the Monge--Amp\`ere operator of locally bounded psh functions can also be defined by taking zero through singular locus (cf. \cite{Demailly_1985}).

Denote by $\PSH(X, \omega)$ the set of all $\om$-psh functions which are not identically $-\infty$. 
Set $V := \int_X \omega^n$.  
The class $\mathcal{E}(X, \omega)$ of $\omega$-psh function with full Monge--Amp\`ere mass is defined as
\[
    \mathcal{E}(X, \omega)
    = \left\{ u\in \PSH(X, \omega): {\lim_{j \to+ \infty}}^{\uparrow}  \int _X  {\1}_{\{u>-j\}} (\omega+\ddc\max\{u,-j\})^n = V \right\}.
\]
Here the measure ${\lim_{j \to+ \infty}}^{\uparrow}  {\1}_{\{u>-j\}} (\omega+\ddc\max\{u,-j\})^n =: \langle (\om + \ddc u)^n \rangle$ is the so-called \emph{non-pluripolar product} (see \cite{BEGZ_2010, BBEGZ_2019} for the singular setting and also for the non-pluripolar mixed Monge--Amp\`ere product). 
We will use the notation $\omega_u^n$ for such a measure and similarly for mixed Monge--Amp\`ere products.

For all $\varphi\in \PSH(X,\omega)\cap L^{\infty}(X)$, the Monge--Amp\`ere energy is defined by
\[
    \E(\varphi) = \frac{1}{(n+1)V} \sum_{j=0}^n \int_X \vph \om_\vph^j \w \om^{n-j},
\]
which satisfies $\E(\varphi+ c) = \E(\varphi )+ c$ for all $c \in \mathbb{R}$ and for $\varphi,\psi\in \PSH(X,\omega)\cap L^{\infty} (X) $, if $\varphi\leq \psi$ then $\E(\varphi)\leq \E(\psi)$ with equality iff $\varphi=\psi$. 
By the later property, $\E$ admits a unique extension to $\PSH(X, \omega)$ by
$$\E(\varphi):= \inf \set{\E(\psi)}{ \varphi \leq \psi, \psi\in \PSH(X,\omega) \cap L^\infty(X)}.$$
The finite energy class is defined by 
$$\mathcal{E}^1(X, \omega)= \set{\varphi \in \PSH(X, \omega)}{\E(\varphi)>-\infty)}.$$
We  have $\mathcal{E}^1(X, \omega)\subset \mathcal{E}(X, \omega)$ from a similar argument as in \cite[Prop. 10.16]{GZbook}.

\subsubsection{The strong topology}\label{ssec:Strong topology}
The space of finite energy potentials $\CE^1(X,\om)$ admits a natural strong metric topology. It is induced by the $d_1$ distance, which can be defined as (cf. \cite[Thm.~2.1]{Darvas_2017}, \cite[Thm.~B]{Dinezza_Guedj_2018})
$$d_1(u,v):= \E(u)+\E(v)-2\E(P_\omega(u,v)),$$
where $P_{\omega}(u,v):=\left( \sup\set{w\in \PSH(X,\omega)}{w\leq \min(u,v)}\right)^*$. 
From the definition, if $u \in \CE^1_\nmlz(X,\om)$, then $d_1(u,0) = -\E(u)$.

\smallskip

The strong topology is connected to the continuity of the Monge--Amp\`ere operator.
Indeed the Monge--Amp\`ere operator produces an homeomorphism between $\CE^1_\nmlz(X,\om)$ and its image, when the latter is endowed of a strong topology \cite[Prop.~2.6]{BBEGZ_2019}. In particular, for our purposes, it is useful to recall that if $(u_k)_k\in \CE^1(X,\om)$ strongly converges to $u\in \CE^1(X,\om)$ then $\om_{u_k}^n$ weakly converges to $\om_u^n$.

\subsubsection{The twisted energy and $\I$ functional}
We  recall the $\I$ functional: 
\[
    \I(u, v) = \frac{1}{V} \int_X (u - v) (\om_v^n - \om_u^n)
\]
for any $u, v \in \mathcal{E}^1(X, \omega)$.

\begin{lem}[{\cite[Lem.~1.9]{BBEGZ_2019}}]\label{lem:L21_vs_I}
There exists $c_n>0$ only depending on $n$ such that for all $u_1, u_2, v \in \CE^1(X,\om)$, 
\[
    c_n \norm{\dd(u_1 - u_2)}_{v}^2 
    \leq \I(u_1,u_2)^{1/2^{n-1}} \lt(\I(u_1, v)^{1-1/2^{n-1}} + \I(u_2, v)^{1-1/2^{n-1}}\rt)
\]
where $\norm{\dd(u_1 - u_2)}_{v}^2 = \int_X \dd (u_1 - u_2) \w \dc (u_1 - u_2) \w \om_v^{n-1}$.
\end{lem}

Letting $\eta$ be a smooth $(1,1)$-form,   as in \cite{Berman_Darvas_Lu_2017} the twisted energy
$$
\E_\eta(u):=\frac{1}{n V} \sum_{j=0}^{n-1} \int_X u \eta \w \om_u^j \w \om^{n-1-j}
$$
can be extended for all $u \in \CE^1(X,\om)$ by its strong continuity. 
More precisely, we have:

\begin{lem}
    \label{lem:Twisted Energy}
    Let $\eta$ be a smooth $(1,1)$-form on $X$ and let $C_{SL}>0$ such that $\sup_X u-C_{SL}\leq \frac{1}{V}\int_X u\om^n$. Then for any $R>0$, there exists an increasing continuous function $f_S: \BR \to \BR$  with $f_S(0)= 0$,  depending only on $S:=R+C_{SL}$ such that 
    \begin{equation}\label{eq_compare_E}
     |\E_{\eta}(u) - \E_{\eta}(v)| 
        \leq  4^n\frac{2C_\eta}{nV}\left(f_S(\I(u,v))+ \|u-v\|_{L^1(X,\omega^n)} \right),
    \end{equation}
    for all $u, v \in L_R := \set{\vph\in \PSH(X,\om) \cap L^\infty}{\int_X\vph\, \om^n=0, \E(\vph)\geqslant -R}$. 
\end{lem}

\begin{proof}
The proof of a similar result was given in \cite[Prop.~4.4]{Berman_Darvas_Lu_2017} in which the right hand side \eqref{eq_compare_E} is replaced by  $f_R(d_1(u,v))$. 

Set $\wu := u-C_{SL}, \wv := v-C_{SL}$, $\ww:=\max(\wu,\wv)$ and observe that $\wu, \wv, \ww \leq 0$, $\E(u),\E(v),\E(w)\geq -R-C_{SL}=-S$.
By definition, we have
\begin{equation*}
    \E_{\eta}(u) - \E_{\eta}(w) 
    =\E_{\eta}(\wu) - \E_{\eta}(\ww) 
    = \frac{1}{nV} \sum_{j=0}^{n-1} \int_{X} (\wu - \ww) \eta \w \om_{\wu}^j \w \om_{\ww}^{n-1-j}.
\end{equation*}
As $\om_{(\wu + \ww)/4}=\frac{1}{2}\omega+ \frac{1}{4}\omega_{\wu}+\frac{1}{4}\omega_{\ww}$ and $-C_\eta\om\leq \eta\leq C_\eta\om$, we infer that 
\begin{align*}
   \vert  \E_{\eta}(u) - \E_{\eta}(w)\vert 
    &\leq  \frac{C_\eta}{nV} \sum_{j=0}^{n-1} \int_{X} (\ww - \wu) \omega \wedge \om_{\wu}^j \w \om_{\ww}^{n-1-j}
    \\
    &\leq \frac{4^nC_\eta}{nV} \int_X (\ww-\wu) \omega^n_{(\wu+\ww)/4}. 
\end{align*}

Since $0\geq (\wu+\ww)/4 \geq \wu$, passing to a resolution to singularities of $X$, it follows from \cite[Lem.~2.7, Lem.~5.8]{BBGZ_2013} that there exists a continuous increasing function $f_S:\BR\to \BR$ depending only on $S$ with $f_S(0)=0$ such that
\begin{align*}
    \int_X (\ww-\wu) \om^n_{(\wu+\ww)/4}&\leq \int_X (\ww-\wu)\big(\om^n_{(\wu+\ww)/4} -\om^n\big) +\lVert \ww-\wu\rVert_{L^1(X,\om^n)}\\
    &\leq f_S\big(\I(\ww,\wu)\big) + \lVert \ww - \wu\rVert_{L^1(X,\om^n)}.
\end{align*}
Hence, as $\lVert \ww-\wu\rVert_{L^1(X,\om^n)} = \lVert w-u\rVert_{L^1(X,\om^n)}\leq \lVert v-u\rVert_{L^1(X,\om^n)}$ and $\I(\ww,\wu) = \I(w,u) \leq \I(v,u)$ using the locality with respect to the plurifine topology, we deduce that
$$
\lvert \E_\eta(u)-\E_\eta(w) \rvert\leq \frac{4^n C_\eta}{nV}\Big(f_S(\I(u,v))+\lVert v-u\rVert_{L^1(X,\om^n)}\Big).
$$
Replacing $u$ by $v$ and using the triangle inequality, we obtain (\ref{eq_compare_E}).
\end{proof}

\subsubsection{Entropy}\label{ssec:entropy}
The leading term in the Mabuchi functional is the \emph{entropy}. 
Given two probability 
measures $\mu,\nu\in \mathcal{P}(X)$ the entropy of $\nu$ with respect to $\mu$ is defined as
$$
    \H_\mu(\nu)
    := \int_X \log \lt(\frac{\dd\nu}{\dd\mu}\rt) \dd\nu
$$
if $\nu$ is absolutely continuous with respect to $\mu$ and as $+\infty$ otherwise. For any fixed measure $\mu$ the entropy $\H_\mu: \mathcal{P}(X)\to \BR\cup \{+\infty\}$ is non-negative and it is lower semi-continuous with respect to the weak convergence of measures \cite[Lem.~6.2.13]{DZ98}.

Letting $\varphi\in \CE(X,\omega)$, we denote by $\H(\vph)$ the entropy of $\om_\vph^n/V$ with respect to $\om^n/V$, namely, 
$$
\H(\vph)= \frac{1}{V} \int_X \log\lt(\frac{\om_\vph^n}{\om^n}\rt) \om_\vph^n
$$
if $\omega_\vph^n$ is absolutely continuous with respect to $\om^n$.

Level sets of the entropy are compact with respect to the strong topology. 
More precisely, we have the following result:
\begin{prop}\cite[Thm. 2.17 \& Cor. 2.19]{BBEGZ_2019}
\label{prop:Strong Compactness}
Fix $\mu$, a probability measure with $L^p$-density with respect to $\om^n$.
Let $\nu$ be a probability measure such that $\H_\mu(\nu)< +\infty$. Then there exists $\vph\in \CE^1(X,\om)$ such that $\om_\vph^n/V = \nu$.
Moreover for any $C>0$ the set
$
\set{u\in \CE^1_\nmlz(X,\om)}{\H_\mu(u)\leq C}
$
is strongly compact.
\end{prop}

\subsubsection{Adapted measure}
\label{sect_adapted_measure}
We recall here the definitions of adapted measure and  Ricci curvatures introduced in \cite{EGZ_2009}.
Now, we assume that $X$ has Kawamata log terminal (klt) singularities. 
Namely, $K_X$ is $\BQ$-Cartier and for any desingularization $p: Y \to X$,
\[
    K_Y = p^\ast K_X + \sum_i a_i E_i,\quad \text{with } a_i > -1
\]
where $E_i$'s are irreducible components of the exceptional divisor.

\begin{defn}
Suppose that $K_X$ is $m$-Cartier for some $m \in \BN^\ast$. 
Let $h^m$ be a smooth hermitian metric on $m K_X$.
Taking $\Om$ a local generator of $m K_X$, the adapted measure associated with $h^m$ is given by 
\[
    \mu_h 
    := \ii^{n^2} \lt(\frac{\Om \w \overline{\Om}}{\abs{\Om}_{h^m}^2}\rt)^{1/m}
\]
Note that this definition does not depend on the choice of $\Om$.
\end{defn}

The klt assumption is equivalent to the finite mass of $\mu_h$. 
Moreover, $\mu_h$ has a density $f \in L^p(X,\om^n)$ for some $p > 1$.
Precisely, on an open set $U$, we have
\[
    \om^n = g \lt(\sum_i |f_i|^2\rt)^{1/m} \ii^{n^2} \lt(\frac{\Om \w \overline{\Om}}{|\Om|_{h^m}^2}\rt)^{1/m}.
\]
where $g$ is a positive smooth function, and $f_i$ are holomorphic functions such that $V((f_i)_i) = U^\sing$.
One obtains $f = \lt(g \lt(\sum_i |f_i|^2\rt)^{1/m}\rt)^{-1}$; thus, $- \log f$ is quasi-psh.

\begin{rmk}\label{rmk_ddc_density}
There exists a constant $A > 0$ such that $-\ddc \log f \geq - A \om$; namely $-\log f$ is an $A \om$-psh function. 
\end{rmk}

The Ricci curvature induced by $\mu_h$ is expressed as
\[
	\Ric(\mu_h) := \frac{1}{m} \ddc \log |\Om|_{h^m}^2.
\]
Let $\om$ be a smooth hermitian metric on $X$. 
The Ricci curvature of $\om$ is given by 
\[
	\Ric(\om) := -\ddc \log \lt(\frac{\om^n}{\mu_h}\rt) + \Ric(\mu_h). 
\]

\begin{lem}\label{lem:compare_klt_entropy}
In the setting just described, let $\H_\mu$ be the entropy with respect to the measure $\mu:=\mu_h$. 
Then the following holds
\begin{itemize}
    \item if $\H(u) \leq B_1 $, then $\H_\mu(u) \leq B_1' = B_1 -\log(c_\mu)$,
    \item if $\H_\mu(u) \leq B_2$, then $\H(u) \leq B_2' = (1 + \max\{p-1, (p-1)^2\}B_2 + C_\mu/V + 1/c_\mu^{p-2}$,
\end{itemize}
where $c_\mu > 0$ is a lower bound of $f$ and $C_\mu$ is a constant so that $\int_X f^p \om^n \leq C_\mu$.
In particular, since $\H_\mu(u) = \H(u) + \frac{1}{V} \int_X (-\log(f)) \om_u^n$, the function $-\log(f)$ is integrable with respect to $\om_u^n$ when $\om_u^n$ has finite entropy with respect to either $\om^n$ or $\mu$.
\end{lem}

\begin{proof}
In the proof, we denote by $\chi(s) = (s+1)\log(s+1) - s$ for $s \geq 0$ and $\chi^\ast(s) = \sup_{t \geq 0} (st-\chi(t)) = e^s - s - 1$ for its Legendre duality. 
We first suppose that $\H(u) \leq B_1$. 
Set $g_1$ the density function $\om_u^n$ with respect to $\om^n$. 
We have the following estimate by direct computation
\[
    \H_\mu(u) 
    = \frac{1}{V} \int_X \log(g_1) g_1 \om^n + \frac{1}{V} \int_X (-\log(f)) g_1 \om_0^n
    \leq B_1 - \log(c_\mu).
\]
Conversely, assume that $\H_{\mu}(u) \leq B_2$. 
Let $g_2$ be the density function of $\om_u^n$ with respect to $\mu$ (i.e. $\om_u^n = g_2 f \om^n$).
Then one can infer
\[
    \H(u) 
    = \frac{1}{V} \int_X \log(g_2) g_2 f \om_0^n + \frac{1}{V} \int_X \log(f) g_2 f \om_0^n
    \leq B_2 + \frac{1}{V} \int_X \log(f) g_2 f \om_0^n.
\]
Put $q = p-1$.
By H\"older--Young inequality, we obtain the following estimate: 
\begin{align*}
    &\frac{1}{V} \int_X |\log(f)| g_2 f \om_0^n 
    \leq \frac{1}{V} \int_X \chi^\ast (q |\log(f)|) + \chi(g_2/q) f \om_0^n\\ 
    &\leq \frac{1}{V} \int_X \1_{\{f \geq 1\}} f^{1+q} \om_0^n
    + \frac{1}{V} \int_X \1_{\{f<1\}} f^{1-q} \om_0^n
    + \max{\{q, q^2\}} \frac{1}{V} \int_X \chi(g_2) f \om^n.  
\end{align*}
This finishes the proof.
\end{proof}

\section{Relative pluripotential theory in families}
\label{sect_families}

In this section, we shall review the concept of convergence of quasi-plurisubharmonic functions in families introduced in \cite{Pan_Trusiani_2023}. 
Then we define a notion of strong convergence in families and establish certain conditions implying strong convergence in families. 

\subsection{Weak convergence}\label{ssec:Weak Convergence}
In the sequel, we always assume that a family $\pi: \CX \to \BD$ fits in the following setup:

\begin{taggedbigset}{(GSN)}\label{sett:general_sett_normal_fibre}
Let $\CX$ be an $(n+1)$-dimensional variety.
Suppose that $\pi: \CX \to \BD$ is a proper, surjective, holomorphic map such that each (schematic) fibre $X_t := \pi^{-1}(t)$ is a K\"ahler variety for any $t \in \BD$. 
Let $\om$ be a hermitian metric on $\CX$ such that, for each $t \in \BD$, the induced metric $\om_t := \om_{|X_t}$ on $X_t$ is K\"ahler. 
In addition, assume that $X_0$ is normal.
\end{taggedbigset}

We provide a few remarks of standard facts on families and maps between fibres:
\begin{rmk}\label{rmk:constant_volume}
The volume function $\BD \ni t \mapsto V_t := \int_{X_t} \om_t^n$ is continuous (cf. \cite[Sec.~1.4]{Pan_2022}).
Similar continuity result holds for $t \mapsto \int_{X_t} \Ta \w \om^{n-1}$ where $\Theta$ is a smooth $(1,1)$-form on $\CX$.
If $\om$ is closed, then $V_t$ is independent of $t \in \BD$ (cf. \cite[Lem.~2.2]{DGG2023}).
\end{rmk}

\begin{rmk}
Normality is open on the base $\BD$ as the map $\pi$ is flat; namely, $X_t$ is normal for all $t$ sufficiently close to zero. 
On the other hand, if $X_t$ is normal for every $t \in \BD$, then so is $\CX$. 
Therefore, up to shrinking $\BD$, we may assume that $\CX$ and $(X_t)_{t \in \BD}$ are normal. 
\end{rmk}

\begin{rmk}\label{rmk:symplectic_diffeo}
Denote by $\CZ$ the singular locus of $\pi$.
Recall that $\pi: \CX \to \BD$ is a submersion on $X_0^\reg$. 
By the tubular neighborhood theorem, we have the following: 
\begin{enumerate}
    \item For all $U_0 \Subset X_0^\reg$, there exists $\CU \subset \CX \setminus \CZ$, a constant $\dt_U > 0$ and a diffeomorphism $F^U: U_0 \times \BD_{\dt_U} \to \CU$ such that the diagram {\small
    \[
    \begin{tikzcd}
        U_0 \times \BD_{\dt_U} \ar[rr, "F^U"] \ar[rd, "\pr_2"']& & \CU \subset \CX \setminus \CZ \ar[ld,"\pi"]\\
        &\BD_{\dt_U}&
    \end{tikzcd}
    \]  }
    commutes. 
    In particular, for all $t \in \BD_{\dt_U}$, $F^U_t := F^U(\cdot, t): U_0 \to U_t =: \CU \cap X_t$ is diffeomorphic onto its image $U_t$.
    \item If $U_0 \Subset V_0 \Subset X_0^\reg$, we have $\dt_U \geq \dt_V > 0$ and $F^U(x,t) = F^V(x,t)$ for all $(x,t) \in U_0 \times \BD_{\dt_V}$.
\end{enumerate}
\end{rmk}

We recall uniform integrability results of Skoda--Zeriahi and Sup-$L^1$ comparison of quasi-psh functions in families from \cite[Thm.~2.9]{DGG2023} and \cite[Cor.~4.8]{Ou_2022}\footnote{The references \cite{DGG2023, Ou_2022} deal with the case $\om$ being K\"ahler on $\CX$ but the proof extends to $\om$ being hermitian and fibrewise closed.
Precisely, the uniform sup-$L^1$ comparison is obtained in \cite[Sec.~3.5]{DGG2023} on a fixed K\"ahler variety when heat kernel $H(x,\cdot,t) \in L^2_1(X)$ and $\int_X H(x,\cdot,t) \om^n = 1$, $\forall t>0$ and desired properties of the heat kernel is then given by Ou \cite[Cor.~4.6]{Ou_2022}.
The family version of Skoda--Zeriahi integrability for hermitian metrics has been obtained in \cite[Prop.~3.3]{Pan_2023} based on the sup-$L^1$ comparison.}: 

\begin{thm}\label{thm:SL_and_Skoda_in_family}
In Setting~\ref{sett:general_sett_normal_fibre}, there exist constants $C_{SL} > 0, \af > 0$, and $C_\af > 0$ such that 
\[
    \sup_{X_t} \psi_t - C_{SL} \leq \frac{1}{V_t} \int_{X_t} \psi_t \om_t^n,
    \quad\text{and} \quad
    \int_{X_t} e^{- \af (\psi_t - \sup_{X_t} \psi_t)} \om_t^n \leq C_\af
\]
for all $t \in \bBD_{1/2}$ and for every $\psi_t \in \PSH(X_t, \om_t)$.
\end{thm}

Near singular locus of $\pi: \CX \to \BD$, we have a neighborhood with small capacity (cf. \cite[Lem.~3.5]{Pan_Trusiani_2023}):
\begin{lem}\label{lem:small_capacity_near_sing}
Let $\CZ$ be the singular locus of $\pi$. 
Up to shrinking $\BD$, for all $\vep > 0$, there exists $\CW_\vep$, a open neighborhood of $\CZ$,
such that for all $t \in \BD$, $\CAP_{\om_t}(\CW_\vep \cap X_t) < \vep$. 
In addition, one can ask $\CW_\vep \subset \CW_{\vep'}$ if $0 < \vep < \vep'$ and $\cap _{\vep >0} \CW_\vep = \CZ$.
\end{lem}

\subsubsection{Convergence in families}
We recall here the definition of convergence of quasi-psh functions in families introduced in \cite[Sec.~2.2]{Pan_Trusiani_2023}.
We fix a few notations. 
Let $\CZ$ be the singular set of the map $\pi$.
For each point $x \in X_{0}^\reg$, up to shrinking $\BD$, there is a chart $U_0 \Subset X_0^\reg$ containing $x$, an open subset $\CU \Subset \CX \setminus \CZ$ with $\CU \cap X_0 = U_0$, and an isomorphism $G: U_0 \times \BD \to \CU$ such that the diagram {\small
\[
\begin{tikzcd}
    U_0 \times \BD \ar[rr, "G"] \ar[rd, "\pr_2"']&& \CU \ar[ld, "\pi"]\\
    &\BD&
\end{tikzcd}
\] }
commutes and $G_{|U_0} = \Id_{U_0}$.
We denote by $G_t: U_0 \overset{\sim}{\longrightarrow} U_t := \CU \cap X_t$ the isomorphism induced by $G$. 

\begin{defn}\label{defn:conv_in_family}
For all $k \in \BN$, let $u_{t_k}$ be a $\om_{t_k}$-psh function on $X_{t_k}$ and $t_k \to 0$ as $k \to +\infty$.
We say that the sequence $(u_{t_k})_k$ converges to $u_{0} \in \PSH(X_0, \om_{0})$ in $L^1$ (resp. $\CC^0$, $\CC^\infty$) if for all data $(U_0, G, \CU)$ as above, $u_{t_k} \circ G_{t_k}$ converges to $u_{0}$ in $L^1(U_0)$ (resp. $\CC^0(U_0)$, $\CC^\infty(U_0)$).
\end{defn}

In the sequel, the above sense of $L^1$ (resp. $\CC^0$, $\CC^\infty$) convergence is called the {\it convergence in families} (resp. {\it $\CC^0$, $\CC^\infty$-convergence in families}) or we say a sequence {\it converging in the family sense} (resp. {\it $\CC^0$, $\CC^\infty$-converging in the family sense}). 
We simply denote $t_k$ (resp. $X_{t_k}$, $\om_{t_k}$) by $k$ (resp. $X_k$, $\om_k$).
We write $(u_k)_{k \in \BN^\ast} \in \PSH_{\fibre}(\CX,\om)$ for a sequence of $u_k \in \PSH(X_k,\om_k)$ and similarly for $(u_k)_k \in \CE^1_{\fibre}(\CX,\om)$.

\subsubsection{Mildly singular setting} \label{sect_mild_sing}
We now review the integrability properties of some canonical densities in mildly singular families.

\begin{taggedbigset}{(klt)}\label{sett:klt}
Under Setting~\ref{sett:general_sett_normal_fibre}, further assume that $\CX$ is $\BQ$-Gorenstein and $X_0$ has at most klt singularities.
\end{taggedbigset}

\begin{rmk}\label{rmk:klt_open}
In Setting~\ref{sett:klt}, by inversion of adjunction (cf. \cite[Thm.~4.9]{Kollar_2013}), 
$\CX$ has klt singularities near $X_0$.
Moreover, $X_t$ has klt singularities for all $t$ close to $0$ (cf. \cite[Cor.~4.10]{Kollar_2013}).
\end{rmk}

Suppose $\CX$ is $m$-Gorenstein for some $m \in \BN^\ast$. 
Let $h$ be a smooth hermitian metric on $m K_{\CX/\BD}$. 
Taking $\Om$ a local generator of $K_{\CX/\BD}$, the adapted measure with respect to $h^m$ is defined as 
\[
    \mu_{t, h_t} = \ii^{n^2} \lt(\frac{\Om_t \w \overline{\Om}_t}{|\Om_t|_{h_t}^2} \rt)^{1/m}
\]
where $\Om_t$ (resp. $h_t$) is the restriction of $\Om$ (resp. $h$) to $X_t$.
From Remark~\ref{rmk:klt_open}, up to shrinking $\BD$, on each fibre $X_t$, one has $\mu_t = f_t \om_t^n$ with $f_t \in L^{p_t}(X_t,\om_t^n)$ for some $p_t > 1$.  
Moreover, by a klt version of \cite[Lem.~4.4]{DGG2023}, one can find a uniform $p>1$ and a uniform constant $C_p > 0$ such that 
\begin{equation}\label{eq:klt_Lp_estimate}
    \int_{X_t} f_t^p \om_t^n \leq C_p.
\end{equation}
Following the construction as in Section~\ref{sect_adapted_measure},  up to shrinking $\BD$, there is a uniform constant $A>0$ such that $-\log f_t \in \PSH(X_t, A \om_t)$ for any $t \in \BD$.

Combining Theorem~\ref{thm:SL_and_Skoda_in_family} and \eqref{eq:klt_Lp_estimate}, one can deduce the following statement:

\begin{cor}\label{cor:klt_Skoda}
Under Setting~\ref{sett:klt}, there exist constants $\bt > 0$ and $C_\bt > 0$ such that for all $t$ close to $0$,
\[
    \int_{X_t} e^{-\bt(\psi_t - \sup_{X_t} \psi_t)} \dd\mu_{t,h_t} \leq C_\bt.
\]
\end{cor}

We now recall a Demailly--Koll\'ar type result from \cite[Prop.~D]{Pan_Trusiani_2023}\footnote{In \cite{Pan_Trusiani_2023}, the statement only deals with a $\om$ being K\"ahler on $\CX$; however the proof depends only on the control of Lelong number of the limiting function. Therefore, the proof basically does not change.}:
\begin{prop}\label{prop:Demailly-Kollar}
In Setting~\ref{sett:general_sett_normal_fibre}, if $u_k \in \CE^1(X_k, \om_k)$ converges to $u_0 \in \CE^1(X_0, \om_0)$, then for all $\gm > 0$,
\[
    \int_{X_k} e^{-\gm u_k} \om_k^n \xrightarrow[k \to +\infty]{} \int_{X_0} e^{-\gm u_0} \om_0^n.
\]
Moreover, there exists a uniform constant $\af > 0$ so that if $u_k \in \PSH(X_k, \om_k)$ converges to $u_0 \in \PSH(X_0, \om_0)$, 
\[
    \int_{X_k} e^{-\af u_k} \om_k^n \xrightarrow[k \to +\infty]{} \int_{X_0} e^{-\af u_0} \om_0^n.
\]
Similarly, in Setting~\ref{sett:klt}, the above two convergences are still valid with respect to $\mu_t$ for an appropriate uniform constant $\af > 0$. 
\end{prop}

\subsubsection{Hartogs' lemma along diffeomorphisms between fibres}
We establish a Hartogs' type result along families of diffeomorphisms:
\begin{lem}\label{lem:diffeo_hartogs}
If $u_k \in \PSH(X_k, \om_k)$ converges to $u_0 \in \PSH(X_0 , \om_0)$, then for all $U_0 \Subset X_0^\reg$ and for all $t$-family of diffeomorphisms $\Phi_t: U_0 \to U_t \Subset X_t^\reg$ such that $\Phi_0 = \Id_{U_{0}}$, 
\[
    \limsup_{k \to +\infty} u_k \circ \Phi_k (x)
    \leq u_0(x)
\]
for all $x \in U_0$ and the equality holds almost everywhere.
In particular, for all $p \in [1,+\infty)$, one has 
\[
    \int_{U_0} \abs{u_k \circ \Phi_k - u}^p \om_0^n \to 0.  
\]
\end{lem}

\begin{proof}
Without loss of generality, we may assume that $u_k$ and $u_0$ are negative.
First, we observe that $\limsup_{k \to + \infty} u_k \circ \Phi_k(x) \leq u_0(x)$ for all $x \in U_0$.
Indeed, by implicit function theorem, one can find an open neighborhood $W_0 \Subset \Om_0$ of $x$ and a $t$-family of local biholomorphism $F_t: W_0 \to W_t \Subset U_t$ for $t$ sufficiently close to $0$. 
Set $\phi_k := u_k \circ F_k$ and $x_k := (F_k)^{-1} \circ \Phi_k(x)$ which converges to $x$ as $k \to +\infty$.
By the Hartogs' lemma (cf. \cite[Thm~1.46]{GZbook}), we obtain 
\begin{equation}\label{eq:limsup}
    \limsup_{k \to +\infty} u_k \circ \Phi_k(x) 
    = \limsup_{k \to +\infty} \phi_k(x_k)
    \leq u_0(x).
\end{equation}

We consider $W_0$ and $F_t$ as above.
Fix a relatively compact open subset $V \Subset W_0$.
We write $V_\vep$ the $\vep$-neighborhood of $V$ with respect to $\dist_{\om_0}$.
Fix $\vep > 0$. 
Since $(F_t)^{-1} \circ \Phi_{t|W_0} \to \Id_{W_0}$ as $t \to 0$,
one can see that $(F_k)^{-1}(\Phi_k(V)) \subset V_\vep$ and $(\Phi_k^{-1} \circ F_k)^\ast(\om_0^n) \leq (1+\vep) \om_0^n$ for all $k \gg 1$.
Then we have
\begin{align*}
    &(1+\vep) \int_{V_\vep} |u_0|^p \om_0^n 
    = \lim_{k \to + \infty} \int_{V_\vep} \abs{u_k \circ F_k}^p (1+\vep) \om_0^n \\
    &\geq \lim_{k \to +\infty} \int_{F_k^{-1}(\Phi_k(V))} \abs{u_k \circ F_k}^p (\Phi_k^{-1} \circ F_k)^\ast(\om_0^n) 
    = \lim_{k \to +\infty} \int_V \abs{u_k \circ \Phi_k}^p \om_0^n\\
    &\geq \int_V \liminf_{k \to +\infty} \abs{u_k \circ \Phi_k}^p \om_0^n
    \geq \int_V |u_0|^p \om_0^n.
\end{align*}
Here the second inequality comes from Fatou's lemma. By monotone convergence theorem, we derive 
\[
    \lim_{k \to +\infty} \int_V \abs{u_k \circ \Phi_k}^p \om_0^n 
    = \int_V |u_0|^p \om_{0}^n
\] 
by taking $\vep \to 0$.
Similarly, one can also get
\[
    \lim_{k \to + \infty} \int_V u_k \circ \Phi_k\, \om_0^n
    = \int_V \limsup_{k \to +\infty} u_k \circ \Phi_k \,\om_0^n 
    = \int_V u_0\, \om_0^n.
\]
This also implies that \eqref{eq:limsup} has equality almost everywhere. 

By Theorem~\ref{thm:SL_and_Skoda_in_family}, $(u_k \circ \Phi_k)_k$ is uniformly bounded in $L^p(U_0)$.  
By Fatou lemma, any weak limit $v$ is less than $u_0$. 
On the other hand, one also has $\int_V v \om_0^n = \lim \int_V u_k \circ \Phi_k \om_k^n = \int_V u_0 \om_0^n$; hence $v = u_0$.
Since $(u_k \circ \Phi)_k$ converges to $u_0$ weakly in $L^p$ and $\lim_{k \to +\infty} \int_V \abs{u_k \circ \Phi_k}^p \om_0^n = \int_V |u_0|^p \om_0^n$ for all $V \Subset U_0$, one can infer that $(u_k \circ \Phi_k)_k$ converges to $u_0$ in $L_\loc^p(U_0)$.
\end{proof}

\subsection{Strong topology}
We now introduce a notion of strong convergence in families: 
\begin{defn}
A sequence $(u_k)_k \in \CE^1_{\fibre}(\CX, \om)$ converges strongly in families to $u_0 \in \CE^1(X_0, \om_0)$ if $(u_k)_k$ converges to $u_0$ in families and $(\E_k(u_k))_k$ converges to $\E_0(u_0)$. 
\end{defn}

Arguing as in \cite[Sec.~5.3]{Pan_Trusiani_2023}, we prove that having a uniform $L^\infty$-estimate, a uniform Laplacian estimate outside the singular locus of $\pi$, and sufficient regularity on the densities of their Monge--Amp\`ere, it is enough to get strong convergence:

\begin{prop}
\label{prop:Unif Laplacian implies strong}
Let $(u_k)_k\in \CE^1_{\fibre}(\CX,\om)$ be a sequence converging in families to $u_0\in \CE^1(X_0,\om_0)$, and let $\CZ$ be the singular locus of $\pi$.  
If $\norm{u_k}_{L^\infty(X_k)} \leq M$ uniformly and for any $\CU \Subset \CX \setminus \CZ $ open set there exists constants $C_\CU > 0$, and $\af = \af(\CU) \in (0,1)$ such that for any $k \in \mathbb{N}$
\[
    \tr_{\om_k} (\om_k+\ddc u_k)_{|X_k \cap \CU} \leq C_\CU
    \quad\text{and}\quad
    \norm{\frac{(\om_k + \ddc u_k)^n}{\om_k^n}}_{\CC^{\af}(\CU \cap X_k)} \leq C_\CU
\]
then $u_k$ strongly converges to $u_0$. 
\end{prop}

We recall Chern--Lu inequality (see e.g. \cite[Prop.~7.2]{Rubinstein_2014}) which will play an important role in several places:
\begin{prop}\label{prop:Chern-Lu}
Let $X$ be a complex manifold endowed with two K\"ahler metrics $\om$, $\homg$. 
Assume that there is a constant $C_3 \in \BR$ such that $\Bisec(\om) \leq C_3$ then 
\[
    \Dt_{\homg} \log \tr_{\homg} \om 
    \geq \frac{\hg^{i \bar{l}} \hg^{k \bar{j}} \hR_{i \bar{j}} g_{k \bar{l}}}{\tr_{\homg} \om} - 2C_3 \tr_{\homg} \om 
\]
In addition, if there are $C_1, C_2 \in \BR$ such that $\Ric(\homg) \geq - C_1 \homg - C_2 \om$, then 
\[
    \Dt_{\homg} \log \tr_{\homg} \om \geq - C_1 - (C_2 + 2 C_3) \tr_\homg \om.
\]
\end{prop}

The following $\CC^2$-estimate will be used for having strong convergences in families.

\begin{thm}\label{prop:uniform_Lap_est}
Let $t_k\to 0$ as $k\to +\infty$ and let $\CZ$ be the singular locus of $\pi$.  Let $(u_k)_k\in \CE^1_{\fibre}(\CX,\om)$ be a sequence such that
\[
    (\om_k + \ddc u_k)^n = e^{v^+_k - v^-_k}\om_k^n 
\]
for any $k\in \mathbb{N}$ where $v_k^\pm \in L^1(X_k)$.  
If there exist uniform constants $A> 0$, and $p > 1$ such that for any $k\in \mathbb{N}$
\begin{enumerate}
    \item $v_k^\pm \in \PSH(X_k, A \om_k)$,
    \item $\sup_{X_k} v_k^+\leq A$,
    \item $\|e^{-v_k^-}\|_{L^p(X_k, \om_k^n)} \leq A$,
\end{enumerate}
then for any $\CU \Subset \CX \setminus \CZ$ relatively compact open set there exists a uniform constant $C_\CU > 0$ such that for any $k$ large
\[
    \tr_{\om_k} (\om_k + \ddc u_k)_{|\CU \cap X_k} \leq C_\CU.
\]
\end{thm}

\begin{proof}
The idea goes back to \cite{Paun_2008} (see also \cite[Appx.~B]{BBEGZ_2019}). 
However, along a degenerate family $\pi: \CX \to \BD$ and its resolution $p: \CY \to \CX$, we would not have a uniform lower bound of bisectional curvature along each fibre, since the central fibre over the resolution could have several components. 
Following the strategy of \cite[Prop.~2.1]{Guenancia_2016} with an argument by using Chern--Lu inequality, instead of using Siu's inequality as in \cite{Paun_2008, BBEGZ_2019}, one can obtain
\begin{equation}\label{eq:fibrewise_Lap_estimate}
    \om_\vep 
    \leq \exp(C_2 + 2C_1 \norm{\vph_\vep}_{L^\infty}) \cdot \exp(- (\Psi + \psi'^+ + C_1 \psi)) \homg_\vep,
\end{equation} 
where $C_1 = (A_1 + A_- +2B + 1)$ and $C_2 = C_2(C_1) := \log(nC_1) + \sup_Y (\Psi + \psi'^+ + C_1 \psi)$.

In the family setting, we take $p: \CY \to \CX$ a log-resolution of $(\CX, X_0)$. 
Since $\rho := \pi \circ p: \CY \to \BD$ is proper and subjective, by generic smoothness, up to shrink $\BD$,  $Y_t = \rho^{-1}(t)$ is smooth for all $t \neq 0$ and $Y_0 = \rho^{-1}(0)$ may have several irreducible components; hence $p_t = p_{|Y_t}: Y_t \to X_t$ is a resolution of singularity for each $t \neq 0$. 
Denote by $E$ the exceptional divisor of $p$. 
One can find a function $\psi \in \PSH(\CY, p^\ast \om)$ which is smooth on $\CY \setminus E$ and a K\"ahler metric $\om_\CY$ on $\CY$ such that $p^\ast\om + \ddc \psi = \om_\CY$ on $\CY \setminus E$. 
Since $p^\ast \om$ is semi-positive, there is a constant $c_1 > 0$ so that $c_1 p^\ast \om \leq \om_\CY$ up to shrinking $\BD$.

Let $E$ be the exceptional divisor of $p$ and $E_i$'s are its irreducible components.
Since $p: \CY \to \CX$ is holomorphic, $p^\ast \om^n \w \ii \dd t \w \dd \bar{t}$ and $\prod_{i} |s_{E_i}|_{h_{E_i}}^{2 \af_i} \om_\CY^n \w \ii \dd t \w \dd \bar{t}$ are comparable up to a smooth bounded function, where $\af_i \in \BQ_{>0}$, $s_{E_i}$ is a canonical section cutting out the divisor $E_i$, and $h_{E_i}$ is a smooth hermitian metric on $\CO(E_i)$. 
One can find hermitian metric $h_\CY$ on $K_{Y/\BD}$ and it satisfies the following
\[
    p^\ast \om^n = \ii^{n^2} \prod_{i} |s_i|^{2 \af_i}_{h_{E_i}} \frac{\Om_{\CY/\BD} \w \overline{\Om_{\CY/\BD}}}{|\Om_{\CY/\BD}|_h^2}
\]
where $\Om_{\CY/\BD}$ is a local generator of $K_{\CY/\BD}$.

Denote by $\psi^+ = p^\ast v^+$, $\psi^- = p^\ast v^-$, and $\Psi = \sum_i \af_i \log |s_i|_{h_{E_i}}^2$.
By Demailly's regularization theorem \cite{Demailly_1992}, there exist three decreasing sequences of smooth functions $(\psi_l^+)_l$, $(\psi_l^-)_l$ and $(\Psi_l)_l$ such that
\begin{itemize}
    \item $\lim_{l \to +\infty} \psi_l^{\pm} = \psi^{\pm}$ and $\lim_{l \to +\infty} \Psi_l = \Psi$ on $\CY$;
    \item $\ddc \psi_l^\pm \geq -A_{\pm}' \om_\CY$ and $\ddc \Psi_l \geq -A_2 \om_\CY$ for some $A_{\pm}' > 0$ and $A_2$ under control;
    \item $\sup_\CY \psi_l^{\pm} \leq A_\pm'$ and $\sup_\CY \Psi_l \leq A_2$
\end{itemize}
For each $t \neq 0$, we consider
\[
    \nu_{t,l} 
    = \ii^{n^2} \frac{\Om_{Y_t} \w \overline{\Om_{Y_t}}}{{|\Om_{Y_t}|_{h_\CY}}_{|Y_t}^2} 
    \cdot e^{\psi^+_l - \psi^-_l + \Psi_l}.
\]
Then we have
{\small
\begin{align*}
    \Ric(\nu_{t,l}) 
    &= - \Ta(K_{\CY/\BD}, h_\CY)_{|Y_t} - \ddc \psi_l^+ + \ddc \psi_l^- - \ddc \Psi_l\\
    &\geq -(A_1 + A_-') \om_{\CY,t} - \ddc (\psi_l^+ + \Psi_l)
\end{align*}
}
where $A_1$ is a constant such that $- \Ta(K_{\CY/\BD}, h_\CY) \geq - A_1 \om_\CY$. 
One can find a constant $B > 0$ such that $\Bisec(\om_{\CY,t}) \leq B$ for all $t \neq 0$ since the bisectional curvature is decreasing under the restriction to holomorphic submanifolds.

By \cite{EGZ_2009}, there is a unique solution $\vph_{k,l} \in \PSH(Y_k, p_k^\ast \om_k) \cap L^\infty(Y_k)$ and a normalizing constant $c_{k,l} \in \BR_{>0}$ such that  
\[
    (p_k^\ast \om_k + \ddc \vph_{k,l})^n = c_{k,l} \nu_{k,l},
    \quad\text{with}\quad \sup_{Y_k} \vph_{k,l} = 0.
\]
For any $\vep \in (0,1]$, Yau's theorem \cite{Yau_1978} yields a unique $\vph_{k,l,\vep} \in \PSH(Y_k, p_k^\ast \om_k + \vep \om_{\CY,k}) \cap \CC^\infty(Y_k)$ solving the following perturbed equation:
\[
    (p_k^\ast \om_k + \vep \om_{\CY,k}+ \ddc \vph_{k,l,\vep})^n 
    = c_{k,l,\vep} \nu_{k,l},
    \quad\text{with}\quad \sup_{Y_k} \vph_{k,l,\vep} = 0,
\]
where $c_{k,l,\vep} \in \BR_{>0}$ is a normalizing constant. 
On each $k$ fixed, we claim that there is a constant $D_{k,l,\vep} > 0$ such that $\norm{\vph_{k,l,\vep}}_{L^\infty} \leq D_{k,l,\vep}$ for all $l \in \BN$ and $\vep \in [0,1]$ and satisfying
\[
    \limsup_{l \to +\infty} \limsup_{\vep \to 0} D_{k,l,\vep} \leq D
\]
for some $D$ uniform in $k$.
By \eqref{eq:fibrewise_Lap_estimate}, we have constants $C_1 = C_1(A_\pm', A_1, A_2, A',B) > 0$ and $C_2 = C_2(C_1)$ and the following estimate on $Y_t \setminus E$:
{\small
\[
    (1+\vep) \om_{\CY,k} 
    \leq \exp(C_2 + 2 C_1 D_{k,l,\vep}) \exp(- (\Psi_l + \psi^+_l + C_1 \psi)) (p_k^\ast\om_k + \vep \om_{\CY,k} + \ddc \vph_{k,l,\vep}).
\]}
For each $l$ fixed, as $\vep \to 0$, after passing to a subsequence, $\vph_{k,l,\vep}$ converges to $\vph_{k,l}$ in $L^1$ and in $\CC_{\loc}^\infty(Y_k \setminus E)$. 
Hence, 
\[
    \om_{\CY,k} 
    \leq \exp(C_2 + 2 C_1 \limsup_{\vep \to 0} D_{k,l,\vep}) \exp(- (\Psi_l + \psi^+_l + C_1 \psi)) (p_k^\ast\om_k + \ddc \vph_{k,l})
\] 
and by the same reason, when $l \to +\infty$, we obtain
\begin{equation}\label{eq:Lap_one_fibre}
    c_1 p_k^\ast \om_k
    \leq \om_{\CY,k} 
    \leq \exp(C_2 + 2C_1 D) \exp(-(\Psi + \psi^+ + C_1 \psi)) (p_k^\ast \om_k + \ddc p_k^\ast u_k).
\end{equation}

Now, it suffices to find the constants $D_{k,l,\vep}$ with $\norm{\vph_{k,l,\vep}}_{L^\infty} \leq D_{k,l,\vep}$ and 
\[
    \limsup_{l \to +\infty} \limsup_{\vep \to 0} D_{k,l,\vep} \leq D
\] 
for a uniform $D>0$. 
Recall that there is a very precise $L^\infty$-estimate on the solution to a complex Monge--Amp\`ere equation (cf. \cite[Thm.~A]{DGG2023}).
The estimate depends only on the $L^p$-estimte of the density and a Skoda's integrability estimate. 
By our construction, the dependence on $L^p$-estimate varies continuously with respect to $l$. 
Therefore, the most crucial point to check is only on the Skoda's integrability.
Following the same argument as in \cite[Prop.~2.3]{DGG2023}, one can obtain a uniform control on the Lelong numbers; namely, for any $u_t \in \PSH(Y_t, p_t^\ast \om_t + \om_{\CY,t})$, 
\[
    \sup_{t \in \BD_{1/2}^\ast} \sup_{y \in Y_t} \nu(u_t, y) < +\infty.
\]
This yields a uniform constant $\af > 0 $ such that for any $k, l \in \BN^\ast$, $\vep \in [0,1]$,  {\small 
\[
    A_{k,\vep,l} := \sup \set{\int_{Y_k} e^{-\af u} \dd \nu_{k,l}}{u \in \PSH(Y_k, p_k^\ast \om_k + \vep \om_{\CY,k}) \text{ with } \sup_{Y_k} u = 0} < + \infty.
\] }
After slightly decreasing the value of $\af$, by Demailly--Koll\'ar \cite{Demailly_Kollar_2001}, $u \mapsto \int_{Y_k} e^{-\af u} \dd\nu_{k,l}$ is continuous with respect to $L^1$-topology. 
Hence, one can deduce $\limsup_{\vep \to 0} A_{k,\vep,l} \leq A_{k,l}$ with  {\small $$A_{k,l} := \sup \set{\int_{Y_k} e^{-\af u} \dd \nu_{k,l}}{u \in \PSH(Y_k, p_k^\ast \om_k) \text{ with } \sup_{Y_k} u = 0}$$}
and similarly, {\small $$\limsup_{l \to +\infty} A_{k,l} \leq A_k := \sup \set{\int_{Y_k} e^{-\af u} \dd \nu_k}{u \in \PSH(Y_k, p_k^\ast \om_k) \text{ with } \sup_{Y_k} u = 0}.$$}
Since the reference form $p_k^\ast \om_k$ and $\nu_k = p_k^\ast (e^{v_k^+ - v_k^-} \om_k^n)$ come from $\pi: \CX \to \BD$, the uniform control of $A_k$ can be showed by taking an appropriate choice of $\af$ in Theorem~\ref{thm:SL_and_Skoda_in_family} and $\|e^{-v_k^-}\|_{L^p(X_k,\om_k^n)} \leq A$.
This completes the proof.
\end{proof}

We now consider the family $\pi: \CX \to \BD$ satisfying Setting~\ref{sett:klt}. 
Let $h$ be a hermitian metric on $m K_{\CX/\BD}$ and $\mu_t$ be adapted measure induced by $h$ on $X_t$. 
For convenience, we also normalize $\mu_t$ to have mass $V$. 
Denote by $f_t$ the density of $\mu_t$ with respect to $\om_t^n$.
Up to shrinking $\BD$, for each $t \in \BD$, there exists a solution $\psi_t \in \PSH(X_t, \om_t) \cap L^\infty(X_t)$ to the following complex Monge--Amp\`ere equation
\begin{equation}
	\label{eqn:Sol to adapted measure}
    (\om_t + \ddc_t \psi_t)^n = \mu_t, \quad \int_{X_t} \psi_t \om_t^n = 0.
\end{equation}
Furthermore, by \cite{DGG2023}, one can also find a uniform constant $M > 0$ so that $\osc_{X_t} \psi_t \leq M$ for each $t \in \BD$.
The continuity of $t \mapsto \E_t(\psi_t)$ near the origin follows immediately as a consequence of the combination of Propositions~\ref{prop:uniform_Lap_est} and \ref{prop:Unif Laplacian implies strong}:
\begin{cor}
\label{cor:Strong continuity of psi}
Let $\psi_t\in \CE^1(X_t,\om_t)$ be the solution to $(\ref{eqn:Sol to adapted measure})$.  Then $\psi_t$ converges strongly to $\psi_0$ as $t\to 0$.
\end{cor}

\section{Entropy in families}
\label{sect_entropy_family}
This section aims to prove the strong compactness of uniformly bounded entropy sequences in families. 
Then, we analyze the semi-continuity property of entropy along strong convergence sequences in families.  

\subsection{Preparation}
We first recall some useful tools for proving the strong compactness in families.

\subsubsection{Orlicz spaces and Luxembourg norms}
Let us quickly recall a few facts on Orlicz spaces. 
Set a convex non-decreasing weight $\chi: [0,+\infty] \to [0,+\infty]$ such that $\chi(0) = 0$ and $\chi(+\infty) = +\infty$. 
Its conjugate weight $\chi^\ast: [0, +\infty] \to [0,+\infty]$ is the Legendre transformation of $\chi$.
Let $\mu$ be a positive measure on $X$. 
The Orlicz space $L^\chi(\mu)$ is defined as the set of all measurable functions $f$ on $X$ such that $\int_X \chi(\vep |f|) \dd\mu < +\infty$ for some $\vep > 0$.
The Luxembourg norm on $L^\chi(\mu)$ is given by
\[
    \norm{f}_{L^\chi(\mu)} = \inf\set{c > 0}{\int_X \chi\lt(\frac{|f|}{c}\rt) \dd\mu \leq 1}.
\]
For any measurable functions $f,g \in L^\chi(\mu)$ and $h \in L^{\chi^\ast}(\mu)$, we have the Minkowski inequality and H\"older--Young inequality:
\[
    \norm{f+g}_{L^\chi(\mu)} \leq \norm{f}_{L^\chi(\mu)} + \norm{g}_{L^\chi(\mu)}
    \quad\text{and}\quad
    \int_X |f g| \dd\mu 
    \leq 2 \norm{f}_{L^\chi(\mu)} \norm{g}_{L^{\chi^\ast}(\mu)}.
\]

\begin{rmk}\label{rmk:L_chi_entropy}
Take $\chi(s) = (s+1)\log(s+1) -s$, then $\chi^\ast(s) = e^s - s - 1$ and $\chi(s) \leq s\log(s) + 1-\log(e-1)$. 
One can check $s \chi'(s) \leq 2\chi(s)$ for all $s > 0$. 
Following the same computation as in \cite[Prop.~1.4]{Darvas_book}, one obtains that for each $0< c <1$ and $C>1$, we have
\begin{equation}\label{eq:chi_ineq}
    c^2 \chi(s) \leq \chi(c s) \leq c \chi(s), 
    \quad\text{and}\quad
    C \chi(s) \leq \chi(C s) \leq C^2 \chi(s),
    \quad \forall s \geq 0.
\end{equation}
By the convexity of $\chi$, \eqref{eq:chi_ineq}, and $\chi(0)=0$ one can check that  
\begin{equation}\label{eq:chi_ineq_3}
    \chi(s+t) \leq 2(\chi(s) + \chi(t)) 
    \quad \forall s,t \geq 0
\end{equation}

If $f \geq 0$ is a function with $\int_X f \log(f) \dd\mu \leq A$, then $\int_X \chi(f) \dd\mu \leq A + \mu(X)(1-\log(e-1))$.
Set $A' := \max\{A + \mu(X)(1-\log(e-1)), 2\}$.
By \eqref{eq:chi_ineq} we have
$
    \int_X \chi(f/A') \dd\mu \leq \int_X \frac{\chi(f)}{A'} \dd\mu \leq 1 
$
and this implies $\|f\|_{L^\chi(\mu)} \leq A'$ which only depends on $A$ and $\mu(X)$. 
Also, if $\int \chi(f) \dd\mu \leq B$, then for all $D>0$, $\int_{X} \chi(f + D) \leq B'$ which only depend on $B, D, \mu(X)$ by \eqref{eq:chi_ineq_3}. 

If $(f_j)_j$ is a sequence of functions such that $\int_{X} \chi(|f_j - f|) \dd\mu$ tends to $0$ as $j \to +\infty$, then $\norm{f_j - f}_{L^\chi(\mu)} \to 0$ when $j \to +\infty$.
\end{rmk}

\subsubsection{Uniform Moser--Trudinger inequalities in families}
For $p>1$, we recall the definition of $p$-energy as follows: 
\[
    \E_p(u) := \int_X |u|^p \om_u^n.
\]

Combining Theorems~\ref{thm:SL_and_Skoda_in_family}, and \cite[Thm.~2.12]{DiNezza_Lu_2022}, we obtain the following uniform estimates in families for normalized potentials with uniformly bounded entropy:

\begin{thm}\label{thm:DGL}
Fix $B>0$ and $p = \frac{n}{n-1}$. 
Up to shrinking $\BD$, there exist constants $\gm$ and $C_B$ such that for each $t \in \BD$ and for all $\vph_t \in \PSH(X_t, \om_t)$ with $\H_t(\vph_t) < B$ and $\sup_{X_t} \vph_t = -1$, 
\[
    \int_{X_t} e^{\gm (-\vph_t)^p} \om_t^n \leq C_B
    \quad\text{and}\quad
    \E_{t,p}(\vph_t) \leq C_B.
\]
\end{thm}

\subsubsection{Poincaré inequality in families}
The next tool that we need is a uniform Poincar\'e constant in families.
The study of Poincar\'e constant in families goes back to Yoshikawa \cite{Yoshikawa_1997} and Ruan--Zhang \cite{Ruan_Zhang_2011}. 
For convenience, the reader is also referred to \cite[Prop.~3.10]{DGG2023}.
Although these references deal with K\"ahler cases, the proof does not rely on the K\"ahler condition.

\begin{lem}\label{lem:Poincare_ineq}
Fix $K \Subset \BD$. 
There exists a uniform Poincar\'e constant $C_P(K)$ such that  all $t \in K$,
\[
    \forall f \in L^2_1(X_t^\reg) \text{ and } \int_{X_t} f \om_t^n = 0, \quad \int_{X_t} |f|^2 \om_t^n \leq C_P \int_{X_t} \abs{\dd f}_{\om_t}^2 \om_t^n.     
\]
\end{lem}

\subsubsection{Regularization of finite entropy potentials and densities}
In this section, we provide a regularization process for finite entropy potentials and densities. 

\begin{lem}\label{lem:regularization_fini_entropy}
Let $(X, \omega)$ be a compact K\"ahler variety. 
Assume that $f^0$ is a probability density with respect to $\om^n$, and $f^0$ has finite entropy.  
There there exists a one-parameter smooth approximations $(f^\bt)_\bt$ converging to $f^0$ in $L^\chi$ as $\bt \to 0$.
Moreover, let $\vph^\bt$ be the unique solution to 
\[
    (\om + \ddc \vph^\bt)^n = f^\bt \om^n \text{ and } \int_X \vph^\bt \om^n = 0. 
\]
We have $\vph^\bt$ converges strongly to $\vph^0$ as $\bt \to 0$. 
\end{lem}

\begin{proof}
With loss of generality, we assume $\int_X \vph \om^n = 0$. 
For each $\vep>0$, one can find $W^\vep$ an open set in $X$ containing $X^\sing$ such that $W^{\vep'} \subset W^\vep$ if $\vep' < \vep$ and $\cap_{\vep>0} W^\vep = X^\sing$. 
For each $\vep > 0$, we choose $(U^\vep)_{\vep >0}$ a relatively compact exhaustion of $X^\reg$ and satisfying $U^\vep \cup W^\vep = X$. 
For each $\vep > 0$, there exists a finite collection of local charts $(U_{j}^\vep)_{j \in J_\vep}$ covering $U^\vep$.
For each $\vep$, take $\eta^\vep$ a cut-off so that $\eta^\vep$ increases as $\vep$ decreases to $0$ and 
\[
    \eta^\vep \equiv 1 \,\,\text{on } U^\vep \setminus W_\vep
    \quad\text{and}\quad
    \eta^\vep \equiv 0 \,\,\text{on } (W_\vep \cap X) \setminus U^\vep.
\]
Choose $(\eta^\vep_j)_{j \in J_\vep}$ cut-offs satisfying $\sum_j \eta^\vep_j = \eta^\vep$, $\supp \eta^\vep_j \subset U^\vep_{0,j}$ and $(\rho_j^\vep)_{j \in J_\vep}$ standard modifiers on $U_{0,j}^\vep$.
We consider 
\[
    g^{\vep,\dt} = \vep + \sum_{j \in J_\vep} (\eta^\vep_j \min\{f, 1/\vep\})\ast(\rho^\vep_j)_\dt
    \quad\text{and}\quad
    f^{\vep,\dt} = \frac{V \cdot g^{\vep,\dt}}{\int_{X} g^{\vep,\dt} \om^n}.
\]
Then for an $\vep$ fixed, $f^{\vep,\dt}$ can extend smoothly to whole $X$ by taking $\frac{\vep V}{\int_{X} g^{\vep,\dt} \om_0^n}$ through $X^\sing$.
We denote its extension by $f^{\vep,\dt}$ and set $\vph^{\vep,\dt}$ a solution to 
\[
    (\om + \ddc \vph^{\vep,\dt})^n = f^{\vep,\dt} \om^n
    \quad\text{with}\quad
    \int_X \vph^{\vep,\dt} \om^n = 0.
\]
We have the following properties:
\begin{itemize}
    \item $f^\vep$ converges in $L^\chi$ to $f$ as $\vep \to 0$ and $f^{\vep,\dt}$ converges in $L^\chi$ to $f^\vep$ when $\dt \to 0$;
    \item $\vph^\vep$ converges strongly to $\vph$ as $\vep \to 0$ and $\vph^{\vep,\dt}$ converges strongly to $\vph^\vep$ when $\dt \to 0$.
\end{itemize}

We check the above assertions in the following.
Since $g^\vep \leq \vep + f$ and $g^\vep$ converges to $f$ almost everywhere, by Lebesgue dominated convergence theorem, $\int_{X} \chi(g^\vep) \om^n \xrightarrow[\vep \to 0]{} \int_{X} \chi(f) \om^n < B'$ and $\int_{X} g^\vep \om^n \xrightarrow[\vep \to 0]{} \int_{X} f \om^n = V$.
We now consider $\vep$ small so that $V/2 \leq \int_{X} g^\vep \om^n \leq 2V$.
We compute {\small
\begin{align*}
    \abs{f^\vep - f} 
    &\leq \frac{\vep + \eta^\vep \abs{\min\{f, 1/\vep\} - f \cdot (\int_{X} g^\vep \om^n/V)} + (1-\eta^\vep) f (\int_{X} g^\vep \om^n/V)}{\int_{X} g^\vep \om^n/V}\\
    &\leq 2 \vep + 10 f.
\end{align*} }
By \eqref{eq:chi_ineq} and \eqref{eq:chi_ineq_3}, $\chi(|f^\vep - f|) \leq 2\chi(2\vep) + 2\chi(10 f) \leq 8 \chi(\vep) + 200 \chi(f)$.
Since $\chi(|f^\vep - f|)$ converges to $0$ almost everywhere, by Lebesgue dominated convergence theorem, we get $\int_{X} \chi(|f^\vep - f|) \om^n \to 0$ as $\vep \to 0$.
By Remark~\ref{rmk:L_chi_entropy}, we also have $\norm{f^\vep - f}_{L^\chi(X, \om^n)} \to 0$ when $\vep \to 0$.
On the other hand, one has $\chi(s) \leq s^2$ for all $s \geq 0$. 
Fix $\vep > 0$.
It is standard that $\|f^{\vep,\dt} - f^\vep\|_{L^2(X,\om^n)}$ converges to $0$ as $\dt \to 0$ and thus, $\int_{X} \chi(|f^{\vep,\dt} - f|) \om^n \to 0$ as $\dt \to 0$.
Again, from Remark~\ref{rmk:L_chi_entropy}, $\|f^{\vep,\dt} - f^\vep\|_{L^\chi(X, \om^n)} \to 0$ when $\vep \to 0$.

Put $\xi(s) = s\log(s)$. 
By the some argument as in Remark~\ref{rmk:L_chi_entropy}, one has $\xi(Cs) \leq C^2 \xi(s)$ for all $s \geq e$ and $\xi(s + t) \leq 2(\xi(s) + \xi(s))$ for any $s, t \geq e$.
We have $g^\vep \leq e + f$ and $\int_{X} g^\vep \geq V/2$ for all $\vep$ small; hence for $\vep$ small, we obtain a control of entropy as follows: {\small
\begin{align*}
    \H_0(f^\vep \om^n) 
    &\leq \int_{X} \1_{\{f^\vep \geq e\}} \xi(f^\vep) \om^n + eV
    \leq \int_{X} 4 \xi(e + f) \om^n + eV\\
    &\leq 8 \int_{X} \1_{\{f \geq e\}} \xi(f) \om^n + (2e + 2e\log 2)V + 9eV
    \leq 8 B + (19 + 2\log 2)eV.
\end{align*}}
For all sequence $\vep_j \to 0$ as $j \to +\infty$, one can extract a subsequential limit of $(\vph^{\vep_j})_j$ which converges strongly to $\vph$ as $j \to +\infty$.
As all the subsequence converges towards $\vph$, we have $\vph^\vep \to \vph$ strongly as $\vep \to 0$.
For a fixed $\vep$, as $f^\vep$ is bounded and $f^{\vep,\dt} \to f^\vep$ in $L^p$ for any $p>0$, one can derive $\vph^{\vep,\dt} \to \vph^\vep$ uniformly when $\dt \to 0$ by the stability estimate (cf. \cite[Thm.~12.21]{GZbook}); thus $\vph^{\vep,\dt}$ also converges strongly towards $\vph^\vep$.

Then for each $\bt > 0$, one can find $\vep_\bt > 0$ such that
\[
    d_1(\vph, \vph^{\vep_\bt}) < \bt/2
    \quad\text{and}\quad
    \|f - f^{\vep_\bt}\|_{L^\chi(X, \om^n)} < \bt/2.
\]
Similarly, for each $\bt$ and $\vep_\bt$ fixed, there exists $\dt_\bt$ such that
\[
    d_1(\vph^{\vep_\bt}, \vph^{\vep_\bt, \dt_\bt}) < \bt/2
    \quad\text{and}\quad
    \|f^{\vep_\bt} - f^{\vep_\bt, \dt_\bt}\|_{L^\chi(X, \om^n)} < \bt/2.
\]
Finally, one simply pick $\vph^\bt = \vph^{\vep_\bt, \dt_\bt}$ and $f^\bt = f^{\vep_\bt, \dt_\bt}$.
\end{proof}

\subsection{Strong compactness}
The goal of this section is to establish the following strong compactness of uniformly bounded entropy sequence in families: 

\begin{thm}\label{thm:strong_compact_fami}
Let $(u_k)_k \in \PSH_\fibre(\CX,\om)$ be a sequence with $t_k \to 0$ as $k \to +\infty$ and $(\sup_{X_k} u_k)_k$ uniformly bounded.
If $(\H_k(u_k))_k$ is uniformly bounded, then there exists a subsequence of $(u_k)_k$ converging strongly in families to a function $u_0 \in \CE^1(X_0, \om_0)$.
\end{thm}

\begin{proof}
Without loss of generality, we assume that $\int_{X_k} u_k \om_k^n = 0$ for every $k$.
By \cite[Prop.~2.8]{Pan_Trusiani_2023}, there exists a subsequence converging in the family sense to a function $u_0 \in \PSH(X_0, \om_0)$.
For each $k$, we denote by $f_k$ the density of $\om_{k, u_k}^n$ with respect to $\om_k^n$.
Take $B > 0$ a uniform constant so that $\H_k(f_k \om_k^n) \leq B$.
By Remark~\ref{rmk:L_chi_entropy}, there is a uniform constant $B' > 0$ such that $\norm{f_k}_{L^\chi(X_k, \om_k^n)} \leq B'$, where $\chi(s) = (s+1)\log(s+1) -s$.

From Remark~\ref{rmk:symplectic_diffeo}, we have a smooth family of diffeomorphisms $F_t^U: U_0 \to U_t$.
Let $(U^{\vep}_0)_\vep$ be a relatively compact exhaustion of $X_0^\reg$.
Then for each $\vep$, we have diffeomorphisms $F_t^{\vep}: U_0^{\vep} \to U_t^{\vep}$ for $|t| < r_\vep$. 
After extracting and relabeling, one can assume $\vep_k \searrow 0$ as $k \to +\infty$, and  if $k \leq l$ then $r_{\vep_k} \geq r_{\vep_l}$ and $F^{\vep_k}_t = F^{\vep_l}_t$ on $U_0^{\vep_k}$ for all $|t| < r_{\vep_l}$ and $\lt(1 - a_k\rt) \om_0^n \leq (F^{\vep_k}_k)^\ast \om_k^n \leq \lt(1 + a_k\rt) \om_0^n$ on $U_0^{\vep_k}$ for some $a_k \searrow 0$ as $k \to +\infty$.
Let $f_{k,0}$ be a function defined by 
\[
    f_{k,0} = 
    \begin{cases}
        h_k := (F_k^{\vep_k})^\ast f_k \lt(1-a_k\rt) V_0/V_k& \text{ on } U_0^{\vep_k}\\
        c_k := V_0 - \int_{U_0^k} h_k \om_0^n \geq 0& \text{ on } X_0 \setminus U_0^{\vep_k}
    \end{cases}.
\] 
By H\"older--Young inequality, one can check $c_k \to 0$ as $k \to +\infty$.
Then we have {\small
\begin{align*}
    \H_0(f_{k,0} \om_0^n) 
    &= \frac{1}{V_0} \int_{U_0^{\vep_k}} h_k \log(h_k) \om_0^n 
    + \underbrace{c_k \log(c_k) \Vol_{\om_0}(X_0 \setminus U_0^k)/V_0}_{=: A_{1,k}}\\
    &\leq \frac{(1-a_k)}{V_k} \int_{U_k^{\vep_k}} f_k \log(f_k) (F_k^{\vep_k})_\ast \om_0^n 
    + \underbrace{\log\lt(\frac{(1-a_k)V_0}{V_k}\rt)}_{=: A_{2,k}} + A_{1,k}\\
    &\leq (1-a_k) \H_k(f_k \om_k^n) + \underbrace{\frac{(1-a_k)}{V_k} \int_{U^{\vep_k}_k} (f_k \log f_k +e^{-1}) b_k \om_k^n}_{=: A_{3,k}} + A_{1,k} + A_{2,k} 
\end{align*} }
for some positive $b_k \to 0$ as $k \to +\infty$.
By the entropy control, one can check $A_{3,k} \to 0$ as $k \to +\infty$, and it is obvious that $A_{1,k}$ and $A_{2,k}$ converge to $0$ when $k \to +\infty$. 
Proposition \ref{prop:Strong Compactness} implies that, up to a subsequence, $f_{k,0} \om_0^n$ converges weakly to a measure $\mu_0$.
By the lower semi-continuity of entropy and the computation above, 
\[
    \H_0(\mu_0) 
    \leq \liminf_{k \to + \infty} \H_0(f_{k,0} \om_0^n) \leq \liminf_{k \to + \infty} \H_k(f_k \om_k^n) \leq B,
\]
and we can write $\mu_0 = f_0 \om_0^n$ for some function $f_0\in L^1(X_0,\om_0^n)$.
By \cite[Thm.~A]{BBGZ_2013} and \cite[Thm.~2.17 \& Cor.~2.19]{BBEGZ_2019}, there exists $\vph_0 \in \CE^1(X_0,\om_0)$ solving $(\om_0 + \ddc_0 \vph_0)^n = f_0 \om_0^n$ with $\int_{X_0} \vph_0 \om_0^n = 0$.

\noindent{\bf Step 1: Regularization.}
By Lemma \ref{lem:regularization_fini_entropy}, one has smooth approximations $(f^\bt_0)_\bt$ of $f_0$ and $\vph^\beta_0$ satisfying 
\[ 
    (\om_0+\ddc \vph_0^\bt )^n= f_0^\beta \om_0^n
\]
and $\vph_0^\beta$ strongly converges to $\vph_0$ as $\beta\rightarrow 0$. 
In addition, we also have smooth approximations $(f^\bt_{k,0})_\bt$ of $f_{k,0}$. 
By the construction, $f^\bt_{k,0}$ is constant on $X_0 \setminus U_0^k$.
Consider $f^\bt_k$ the smooth function obtained by the extension of $((F^{\vep_k}_k)^{-1})^\ast f_{k,0}^\bt \times V_k/V_0$ as a constant through $X_k \setminus F^{\vep_k}_k(U_0^{\vep_k})$.
Set $\vph^\beta_k$ be a solution to 
\[ 
    (\om_k+\ddc \vph_k^\bt )^n= f_{k}^\beta \om_k^n \quad \text{  with } \int_{X_k} \vph^\bt_k\om_k^n=0.
\]
By the construction of $\vph^\bt_0$ and $f^\bt_0$, we have
\begin{equation}\label{eq:fcn_beta_parameter}
    d_1(\vph_0,\vph^\bt_0) < \bt
    \quad\text{and}\quad
    \|f_0 - f^\bt_0\|_{L^\chi(X_0, \om_0^n)} < \bt.
\end{equation}
Follows from Proposition~\ref{prop:Unif Laplacian implies strong}, one can derive the following Lemma:

\begin{lem}\label{lem:cptness_appox_E_conv}
As $k \to + \infty$, for each $\bt$ fixed, $\vph^{\bt}_k$ converges to $\vph^{\bt}_0$ smoothly in the family sense, and $\lim_{k \to +\infty} \E_k(\vph_k^{\bt}) = \E_0(\vph_0^{\bt})$. 
\end{lem}

\noindent{\bf Step 2: Comparing $\vph_0$ and $u_0$.}
We now aim to show $\vph_0 \equiv u_0$.
We shall denote the truncations by $u_{k,C} := \max\{u_k,-C\}$ and the similar notation for $u_0$. 

\smallskip
We separate $\I_k(u_k, \vph_k^\bt)$ into the following two terms:
\begin{align*}
    V_k \I_k(u_k, \vph_k^\bt) 
    &= \int_{X_k} (u_k - \vph_k^\bt) (f_k^\bt - f_k) \om_k^n \\
    &= \underbrace{\int_{X_k} (u_{k,1/\ld} - \vph_k^\bt) (f_k^\bt - f_k) \om_k^n}_{=: (\RN{1})_{k,\ld,\bt}} 
    + \underbrace{\int_{X_k} (u_k - u_{k,1/\ld}) (f_k^\bt - f_k) \om_k^n}_{=: (\RN{2})_{k,\ld,\bt}}.
\end{align*}

\noindent{\bf Substep 2.1: analyzing $(\RN{1})_{k,\ld,\bt}$.}
We now focus on $(\RN{1})_{k,\ld,\bt}$.
We would like to understand first the convergence of $(\RN{1})_{k,\ld,\bt}$ as $k \to +\infty$.
Consider {\small
\begin{align*}
    (\RN{1})_{k,\ld,\bt} 
    &= \int_{X_0} (u_{0,1/\ld} - \vph_0^\bt) (f^\bt_0 - f_0) \om_0^n \\
    &\quad + \underbrace{\int_{U_k^\zt} (u_{k,1/\ld} - \vph_k^\bt) (f_k^\bt - f_k)\om_k^n - \int_{U_0^\zt} (u_{0,1/\ld} - \vph_0^\bt) (f^\bt_0 - f_0) \om_0^n}_{=: (\RN{1}')_{k,\ld,\bt,\zt}} \\
    &\quad + \underbrace{\int_{X_k \setminus U_k^\zt} (u_{k,1/\ld} - \vph_k^\bt) (f_k^\bt - f_k)\om_k^n - \int_{X_0 \setminus U_0^\zt} (u_{0,1/\ld} - \vph_0^\bt) (f^\bt_0 - f_0) \om_0^n}_{=: (\RN{1}'')_{k,\ld,\bt,\zt}}
\end{align*} }
where $\zt$ is a new parameter that we shall determine later.
Compute $(\RN{1}')_{k,\ld,\bt,\zt}$:
{\small
\begin{align*}
    &(\RN{1}')_{k,\ld,\bt,\zt} 
    = \int_{U_0^\zt} \lt[ F_k^\ast (u_{k,1/\ld} - \vph_k^\bt) F_k^\ast(f_k^\bt - f_k) - (u_{0,1/\ld} - \vph_0^\bt) (f^\bt_0 - f_0) \rt] \om_0^n \\
    &\qquad\qquad\qquad
    + \underbrace{\int_{U_0^\zt} F_k^\ast (u_{k,1/\ld} - \vph_k^\bt) F_k^\ast(f_k^\bt - f_k) (F_k^\ast \om_k^n - \om_0^n)}_{=: J_{1,k}}\\
    &= \underbrace{\int_{U_0^\zt} \lt[F_k^\ast (u_{k,1/\ld} - \vph_k^\bt) - (u_{0,1/\ld} - \vph_0^\bt)\rt] F_k^\ast(f_k^\bt - f_k) \om_0^n}_{=: J'_{k,\ld,\bt,\zt}} 
    + \underbrace{\int_{U_0^\zt} (u_{0,1/\ld} - \vph_0^\bt) \lt[F_k^\ast (f^\bt_k - f_k) - (f^\bt_0 - f_0)\rt] \om_0^n}_{_{=: J''_{k,\ld,\bt,\zt}}} + J_{1,k}.
\end{align*}
}%
By Theorem~\ref{thm:SL_and_Skoda_in_family} and the comparison of $F_k^\ast\om_k^n$ and $\om_0^n$, one can easily get $\lim_{k \to +\infty} |J_{1,k}| = 0$.
By H\"older--Young inequality, 
\[
    |J'_{k,\ld,\bt,\zt}| \leq 2 \|F_k^\ast(u_{k,1/\ld} - \vph_k^\bt) - (u_{0,1/\ld} - \vph_0^\bt)\|_{L^{\chi^\ast}(U_0^\zt, \om_0^n)} \|f_k^\bt - f_k\|_{L^\chi(U_0^\zt, \om_0^n)}
\]
We first claim that 
\begin{equation}\label{eq:str_top_key1}
    \limsup_{k \to +\infty} \|F_k^\ast(u_{k,1/\ld} - \vph_k^\bt) - (u_{0,1/\ld} - \vph_0^\bt)\|_{L^{\chi^\ast}(U_0^\zt, \om_0^n)} = 0,
\end{equation}
and thus, by using $\|f_k^\bt - f_k\|_{L^\chi(U_0^\zt, \om_0^n)} \leq C(B)$ for a constant $C(B)>0$ independent of $k$, we obtain
\begin{equation}\label{eq:RN1}
    \limsup_{k \to +\infty} |J'_{k,\ld,\bt,\zt}| = 0.
\end{equation}
\smallskip
We now check \eqref{eq:str_top_key1}.
Indeed, for all $\vep > 0$, by Egorov's theorem, there exists $E_\vep \subset U_0^\zt$ such that $F_k^\ast u_{k,1/\ld}$ (resp. $F_k^\ast \vph_k^\bt$) converges uniformly to $u_{0,1/\ld}$ (resp. $\vph_0^\bt$) on $E_\vep$ and $\int_{U_0^\zt \setminus E_\vep} \om_0^n < \vep$; thus, for any $c > 0$, 
{\small
\begin{align*}
    &\int_{U_0^\zt} \chi^\ast(|(F_k^\ast u_{k,1/\ld} - u_{0,1/\ld}) - (F_k^\ast\vph_k^\bt - \vph_0^\bt)|/c) \om_0^n \\
    &= \int_{E_\vep} \chi^\ast(|(F_k^\ast u_{k,1/\ld} - u_{0,1/\ld}) - (F_k^\ast\vph_k^\bt - \vph_0^\bt)|/c) \om_0^n 
    + \int_{U_0^\zt \setminus E_\vep} \chi^\ast(|(F_k^\ast u_{k,1/\ld} - u_{0,1/\ld}) - (F_k^\ast\vph_k^\bt - \vph_0^\bt)|/c) \om_0^n.
\end{align*}
}%
By the uniform convergence on $E_\vep$, the first part on the RHS goes to zero as $k \to +\infty$. 
Since $\chi^\ast(s) = e^s - s -1$, 
\[
    \int_{U_0^\zt \setminus E_\vep} \chi^\ast(|(F_k^\ast u_k - u_0) - (F_k^\ast \vph_k^\bt - \vph_0^\bt)|/c) \om_0^n
    \leq 
    \int_{U_0^\zt \setminus E_\vep} e^{-\Psi_k/c} \om_0^n
\]
where $\Psi_k := F_k^\ast u_{k,1/\ld} + u_{0,1/\ld} + F_k^\ast \vph_k^\bt + \vph_0^\bt$. 
Note that $\Psi_k$ converge to $\Psi_0 := 2 u_{0,1/\ld} + 2 \vph_0^\bt$ in $L^1(U_0^\zt, \om_0^n)$ and $(\Psi_k)_k$ are uniformly bounded. 
Then for an arbitrary $p>1$, by H\"older inequality, 
\[
    \int_{U_0^\zt \setminus E_\vep} e^{-\Psi_k / c} \om_0^n 
    \leq \Vol_{\om_0}(U_0^\zt \setminus E_\vep)^{1/q} \lt(\int_{U^\zt_0} e^{-\frac{p \Psi_k}{c}}\om_0^n\rt)^{1/p} 
    \leq \vep^{1/q} \lt(\int_{U^\zt_0} e^{-\frac{p \Psi_k}{c}}\om_0^n\rt)^{1/p},
\]
where $q>1$ is such that $1/p + 1/q = 1$.
By the construction, 
$\lt(\int_{U_0^\zt} e^{-p \Psi_0/c} \om_0^n\rt)_k$ is uniformly bounded. 
Therefore, for all $c > 0$,  
\begin{align*}
    \limsup_{k \to + \infty} \int_{U_0^\zt} \chi^\ast \lt(\frac{|(F_k^\ast u_{k,1/\ld} - u_{0,1/\ld}) - (F_k^\ast \vph_k^\bt - \vph_0)|}{c}\rt) \om_0^n = 0,
\end{align*}
and thus, \eqref{eq:str_top_key1} holds.

\smallskip
Now, we verify $\lim_{k \to + \infty} J''_{k,\bt,\ld,\zt} = 0$.
We recall that quasi-psh functions are quasi-continuous (see \cite{Bedford_Taylor_1982}); namely, for each quasi-psh function $\psi$, for all $\vep > 0$, there exists an open set $G_{\psi,\vep} \subset U_0^\zt$ of capacity smaller than $\vep$ (i.e. $\CAP_{\om_0}(G_{\psi,\vep}) < \vep$) such that $\psi$ is continuous on $G_{\psi,\vep}$.
Hence, for all $\vep > 0$, one can find $G_\vep$ an open subset of $Y$ such that $u_{0,1/\ld} - \vph_0^\bt$ is continuous on $G_\vep$. 
We have
{\small
\begin{align*}
    J''_{k,\bt,\ld,\zt} &= \int_{U_0^\zt} (u_{0,1/\ld} - \vph_0^\bt) [(F_k^\ast f_k^\bt - F_k^\ast f_k) - (f_0^\bt - f_0)] \om_0^n\\
    &= \int_{G_\vep} (u_{0,1/\ld} - \vph_0^\bt) [(F_k^\ast f_k^\bt - F_k^\ast f_k) - (f_0^\bt - f_0)] \om_0^n
    + \underbrace{\int_{U_0^\zt \setminus G_\vep} (u_{0,1/\ld} - \vph_0^\bt) [(F_k^\ast f_k^\bt - F_k^\ast f_k) - (f_0^\bt - f_0)] \om_0^n}_{=: K_{k,\bt,\ld,\zt}}.
\end{align*}
}%
The first term on the RHS converges to zero since $F_k^\ast f_k$ converges weakly to $f_0$ and $F_k^\ast f_k^\bt$ converges smoothly to $f_0^\bt$ by the construction. 
By H\"older--Young inequality, 
\[
    |K_{k,\bt,\ld,\zt}|
    \leq 2\|\1_{U_0^\zt \setminus G_\vep} (u_{0,1/\ld} - \vph_0^\bt)\|_{L^{\chi^\ast}(U_0^\zt, \om_0^n)} \|(F_k^\ast f_k^\bt - F_k^\ast f_k) - (f_0^\bt - f_0)\|_{L^\chi(U_0^\zt, \om_0^n)}.
\]
For all $c > 0$, we have
\[
    \int_{U_0^\zt} \1_{U_0^\zt \setminus G_\vep} \chi^\ast\lt(|u_{0,1/\ld} - \vph_0^\bt|/c\rt) \om_0^n 
    \leq \int_{U_0^\zt} \1_{U_0^\zt \setminus G_\vep} e^{-\Phi_0/c} \om_0^n
\]
where $\Phi_0 = u_{0,1/\ld} + \vph_0^\bt$. 
By H\"older inequality, 
\[
    \int_{U_0^\zt} \1_{U_0^\zt \setminus G_\vep} e^{-\Phi_0/c} \om_0^n
    \leq \lt(\Vol_{\om_0}(U_0^\zt \setminus G_\vep)\rt)^{1/q} 
    \lt(\int_{U_0^\zt} e^{-p \Phi_0/c} \om_0^n\rt)^{1/p}
    \leq \vep^{1/q} \lt(\int_{U_0^\zt} e^{-p \Phi_0/c} \om_0^n\rt)^{1/p}.
\]
Since $\Phi_0$ is bounded, $\int_{U_0^\zt} e^{-p \Phi_0/c} \om_0^n < +\infty$ for all $c > 0$. 
Therefore, for each $c>0$ small, one can choose $\vep$ sufficiently small such that $\vep^{1/q} \lt(\int_{U_0^\zt} e^{-p \Phi_0/c} \om_0^n\rt)^{1/p} \leq 1$; hence, 
\[
    \limsup_{\vep \to 0} \|\1_{U_0^\zt \setminus G_\vep} (u_{0,1/\ld} - \vph_0^\bt)\|_{L^{\chi^\ast}(U_0^\zt,\om_0^n)} = 0.
\]
On the other hand, by the uniform control of the entropies, there is a constant $C'(B) > 0$ such that $\|(F_k^\ast f_k^\bt - F_k^\ast f_k) - (f_0^\bt - f_0)\|_{L^\chi(U_0^\zt,\om_0^n)} \leq C'(B)$ and this ensures that 
\begin{equation}\label{eq:RN2}
    \limsup_{k \to +\infty} |J''_{k,\bt,\ld,\zt}| = 0.
\end{equation}
Hence, we obtain 
\begin{equation}\label{eq:cptness_I'}
    \lim_{k \to +\infty} (\RN{1}')_{k,\ld,\bt,\zt} = 0.
\end{equation}
There is a constant $C_1(\ld,\bt)>0$ such that $|u_{k,1/\ld}|+|\vph_k^\bt| \leq C_1(\ld,\bt)$ for all $k$ sufficiently large. 
By Theorem~\ref{thm:SL_and_Skoda_in_family}, $\CAP_{\om_t}(X_t \cap \CW_\zt) < \zt$ and H\"older--Young inequality for all $t \in \BD$, we have the following estimates 
\begin{align*}
    &\abs{\int_{X_k \setminus U_k^\zt} (u_{k,1/\ld} - \vph_k^\bt) (f_k^\bt - f_k) \om_k^n}
    \leq 2 C_1(\ld,\bt) \|\1_{X_k \cap \CW_\zt}\|_{L^{\chi^\ast}(X_k,\om_k^n)} \|f_k^\bt - f_k\|_{L^\chi(X_k,\om_k^n)}. 
\end{align*} 
Note that $\|\1_{X_k \cap \CW_\zt}\|_{L^{\chi^\ast}(X_k,\om_k^n)} = \frac{1}{(\chi^\ast)^{-1}(1/\Vol_{\om_k}(X_k\cap \CW_\zt))} \to 0$ as $\zt \to 0$ and $\|f_k^\bt - f_k\|_{L^\chi(X_k,\om_k^n)} \leq C''(B)$ for some constant uniform $C''(B)>0$ by the entropy control.
Hence, 
\[
    \lim_{\zt \to 0} \limsup_{k \to +\infty}
    \abs{\int_{X_k \setminus U_k^\zt} (u_{k,1/\ld} - \vph_k^\bt) (f_k^\bt - f_k) \om_k^n}
    =0
\]
Similarly, we also have
\[
    \lim_{\zt \to 0} \limsup_{k \to +\infty} \abs{\int_{X_0 \setminus U_0^\zt} (u_{0,1/\ld} - \vph_0^\bt) (f_0^\bt - f_0) \om_0^n} = 0.
\]
Therefore, fixing $\ld$ and $\bt$, for any $\epsilon > 0$, one can choose $\zt(\epsilon, \ld, \bt)$ small so that
\begin{equation}\label{eq:cptness_I''}
    \abs{(\RN{1}'')_{k,\ld,\bt,\zt}} < \epsilon.
\end{equation}
Combining \eqref{eq:cptness_I'} and \eqref{eq:cptness_I''}, one derive
\begin{equation}\label{eq:cptness_I_1}
    \lim_{k\to+\infty} (\RN{1})_{k,\ld,\bt} 
    = \int_{X_0} (u_{0,1/\ld} - \vph_0^\bt) (f_0^\bt - f_0) \om_0^n.
\end{equation} 
By the monotone convergence theorem, 
\begin{equation}\label{eq:cptness_I_2}
    \lim_{\ld \to 0} \lim_{k\to+\infty} (\RN{1})_{k,\ld,\bt} 
    = \int_{X_0} (u_0 - \vph_0^\bt) (f_0^\bt - f_0) \om_0^n.
\end{equation}
Now, we are going to show $\int_{X_0} (u_0 - \vph_0^\bt) (f_0^\bt - f_0) \om_0^n$ converges to $0$ as $\bt$ goes to $0$.
By H\"older--Young inequality and \eqref{eq:fcn_beta_parameter}, we have {\small
\begin{align*}
   \abs{\int_{X_0} (u_0 - \vph_0^\bt) (f_0^\bt - f_0) \om_0^n}
    &\leq 2 \|u_0 - \vph_0^\bt\|_{L^{\chi^\ast}(X_0, \om_0^n)} \|f_0^\bt - f_0\|_{L^\chi(X_0, \om_0^n)}\\
    &\leq 2 \bt \|u_0 - \vph_0^\bt\|_{L^{\chi^\ast}(X_0, \om_0^n)}. 
\end{align*} }
Recall that by the definition
\[
    \|u_0 - \vph_0^\bt\|_{L^{\chi^\ast}(X_0,\om_0^n)}
    = \inf\set{c > 0}{\int_{X_0} \chi^\ast\lt(\frac{|u_0 - \vph_0^\bt|}{c}\rt) \om_0^n \leq 1}.
\]
Using Proposition~\ref{prop:Demailly-Kollar}, for any $c > 0$, one has 
\[
    \int_{X_0} e^{\frac{2|\vph_0^\bt|}{c}} \om_0^n 
    \xrightarrow[\bt \to 0]{} 
    \int_{X_0} e^{\frac{2|\vph_0|}{c}} \om_0^n.
\] 
Choose $c_0 > 0$ large enough so that $\int_{X_0} e^{\frac{2|u_0|}{c_0}} \om_0^n < V_0 + \vep$ and $\int_{X_0} e^{\frac{2|\vph_0|}{c_0}} \om_0^n < V_0 + \vep/2$. 
Then take $\bt>0$ sufficiently small such that $\int_{X_0} e^{\frac{2|\vph_0^\bt|}{c_0}} \om_0^n < V_0 + \vep$. 
By H\"older inequality, one can infer {\small
\begin{align*}
    \int_{X_0} e^{\frac{|u_0-\vph_0^\bt|}{c_0}} \om_0^n
    &\leq \lt(\int_{X_0} e^{\frac{2|u_0|}{c_0}} \om_0^n\rt)^{1/2} \lt(\int_{X_0} e^{\frac{2|\vph_0^\bt|}{c_0}} \om_0^n\rt)^{1/2}\\
    &\leq V_0 +\vep
    \leq 1 + V_0 + \frac{1}{c_0} \int_{X_0} |u_0 - \vph_0^\bt| \om_0^n 
\end{align*} }
and it implies that $\|u_0 - \vph_0^\bt\|_{L^{\chi^\ast}(X_0, \om_0^n)} \leq c_0$ for all $\bt$ close to $0$.
Therefore, 
\begin{equation}\label{eq:cptness_I_3}
    \abs{\int_{X_0} (u_0 - \vph_0^\bt) (f_0^\bt - f_0) \om_0^n} \leq c_0 \bt.
\end{equation}
Combining \eqref{eq:cptness_I_1}, \eqref{eq:cptness_I_2} and \eqref{eq:cptness_I_3}, we show the following limit of the first term $(\RN{1})_{k,\ld,\bt}$ is zero: 
\begin{equation}\label{eq:cptness_I_lim}
    \lim_{\bt \to 0} \lim_{\ld \to 0} \lim_{k \to +\infty} (\RN{1})_{k,\ld,\bt} = 0.
\end{equation}

\noindent{\bf Substep 2.2: analyzing $(\RN{2})_{k,\ld,\bt}$.}
Now, we estimate the second term $(\RN{2})_{k,\ld,\bt}$: 
\begin{align*}
    (\RN{2})_{k,\ld,\bt} 
    &= \int_{X_k} (u_k - u_{k,1/\ld}) f_k^\bt \om_k^n + \int_{X_k} \1_{\{u_k < -1/\ld\}} (-u_k - 1/\ld) f_k \om_k^n
\end{align*}
By the construction of $f_k^\bt$ and the monotone convergence theorem, we have 
\[
    \lim_{\ld \to 0} \lim_{k \to +\infty} \int_{X_k} (u_k - u_{k,1/\ld}) f_k^\bt \om_k^n = \lim_{\ld \to 0} \int_{X_0} (u_0 - u_{0,1/\ld}) f_0^\bt \om_0^n = 0. 
\]
By Chebyshev's inequality, we derive
\[
    \int_{X_k} \1_{\{u_k < -1/\ld\}} (-u_k - 1/\ld) f_k \om_k^n 
    \leq \int_{X_k} \1_{\{u_k < -1/\ld\}} (-u_k) \om_{k,u_k}^n
    \leq \ld^q \int_{X_k} |u_k|^{1+q} \om_{k,u_k}^n
\]
for any $q > 0$. 
Put $q = \frac{1}{n-1}$. 
By Theorem~\ref{thm:SL_and_Skoda_in_family} and Theorem~\ref{thm:DGL}, 
\[
    \ld^q \int_{X_k} |u_k|^{1+q} \om_{k,u_k}^n 
    \leq \ld^q (C_B + V C_{SL}^{1+q}).
\]
Hence, 
\begin{equation}\label{eq:cptness_II_lim}
    \lim_{\ld \to 0} \limsup_{k\to+\infty} |(\RN{2})_{k,\ld,\bt}| = 0.
\end{equation}
Combining \eqref{eq:cptness_I_lim}, and \eqref{eq:cptness_II_lim}, we deduce
\begin{equation}\label{eq:cptness_sum_123}
    \lim_{\bt \to 0} \limsup_{k \to + \infty} \I_k(u_k, \vph_k^\bt) = 0.
\end{equation}

\noindent{\bf Substep 2.3: conclusion.}
By Lemma~\ref{lem:L21_vs_I} and Lemma~\ref{lem:Poincare_ineq}
{\small
\begin{align*}
    \frac{1}{C_P} \int_{X_k} |u_k - \vph_k^\bt|^2 \om_k^n 
    &\leq \int_{X_k} \abs{\dd (u_k - \vph_k^\bt)}_{\om_k}^2 \om_k^n\\ 
    &\leq c_n \I_k(u_k, \vph_k^\bt)^{1/2^{n-1}} \lt(\I_k(u_k,0)^{1-1/2^{n-1}} + \I_k(\vph_k^\bt,0)^{1-1/2^{n-1}}\rt)
\end{align*}
}
Using the uniform control of entropy, Theorem~\ref{thm:DGL} and Theorem~\ref{thm:SL_and_Skoda_in_family}, we have 
\[
    \I_k(u_k,0) \leq \frac{1}{V_k} \int_{X_k} |u_k| \om_{k,u_k}^n \leq \frac{C_B + C_{SL} + 1}{V_k}.
\]
Then letting $k \to +\infty$, 
\begin{align*}
    \int_{X_0} |u_0 - \vph_0^\bt|^2 \om_0^2 
    &\leq c_n C_P \limsup_{k \to +\infty} \I_k(u_k, \vph_k^\bt)^{1/2^{n-1}} \\
    &\qquad \times \lt( \lt(\frac{C_B + C_{SL} + 1}{V_0}\rt)^{1-1/2^{n-1}} + \I_0(\vph_0^\bt,0)^{1-1/2^{n-1}} \rt).
\end{align*}
We get $\int_{X_0} |u_0 - \vph_0|^2 \om_0^n = 0$ by using \eqref{eq:cptness_sum_123} and the strong convergence of $\vph_0^\bt$ to $\vph_0$.
We obtain $u_0 = \vph_0$ almost everywhere and thus $u_0 \equiv \vph_0$ by the plurisubharmonicity.

Finally, one can check that $\E_k(u_k)$ converges to $\E_0(u_0)$.
Indeed, similar as Lemma~\ref{lem:Twisted Energy}, we have $|\E_k(u_k) - \E_k(\vph_k^\bt)| \leq C (f_S(\I_k(u_k, \vph_k^\bt)) + \|u_k - \vph_k^\bt\|_{L^1(X_k, \om_k^n)})$, for some uniform constant $C, S$ and $f$ is increasing continuous function with $f(0) = 0$; thus,  
{\small
\begin{align*}
    \abs{\E_k(u_k) - \E_0(u_0)}
    &\leq \abs{\E_k(u_k) - \E_k(\vph_k^\bt)}
    + \abs{\E_k(\vph_k^\bt) - \E_0(\vph_0^\bt)}
    + \abs{\E_0(\vph_0^\bt) - \E_0(u_0)}\\  
    &\leq C \lt[f_S (\I_k(u_k, \vph_k^\bt)) + \|u_k - \vph_k^\bt\|_{L^1(X_k, \om_k^n)}\rt]
    + \abs{\E_k(\vph_k^\bt) - \E_0(\vph_0^\bt)}
    + \abs{\E_0(\vph_0^\bt) - \E_0(u_0)}.
\end{align*} }
By \eqref{eq:cptness_sum_123}, Claim~\ref{lem:cptness_appox_E_conv}, \eqref{eq:fcn_beta_parameter} and $\vph_0 \equiv u_0$, we conclude
\[
    \limsup_{k \to 0} |\E_k(u_k) - \E_0(u_0)| = 0
\]
and this completes the proof. 
\end{proof}

\subsection{Semi-continuity properties}
This section focuses on the (semi-)continuity of entropy, adapted entropy, and twisted energy in the context of strong convergence in families. 

\subsubsection{Continuity of the Monge--Amp\`ere operator}
For a fixed variety, it is known that the Monge--Amp\`ere operator is continuous along a strong convergent sequence \cite[Prop.~2.6]{BBEGZ_2019}. 
Here, we extend such results to the family setting.

\begin{defn}
A sequence of measures $\mu_k$ on a sequence of fibres $X_k$ converges weakly in families towards a measure $\mu_0$ on $X_0$ if for all continuous function $\eta$ on $\CX$, $\lim_{k \to +\infty} \int_{X_k} \eta_{|X_k} \dd\mu_k = \int_{X_0} \eta_{|X_0} \dd\mu_0$. 
Namely, the sequence currents $((\om_k + \ddc_k u_k)^n \w [X_k])_k$ converges to $(\om_0 + \ddc_0 u_0)^n \w [X_0]$ in the sense of distributions on $\CX$.
\end{defn}

\begin{lem}\label{lem:weak_conv_MA}
If $(u_k)_k \in \CE^1_{\fibre}(\CX,\om)$ converges strongly to $u_0 \in \CE^1(X_0,\om_0)$ and $(\H_k(u_k))_k$ is uniformly bounded from above, then $\om_{k, u_k}^n$ converges weakly in families to $\om_{0, u_0}^n$.
\end{lem}

\begin{proof}
One may assume that $\int_{X_k} u_k \om_k^n = 0$. 
Following the same argument at the beginning of the proof of Theorem~\ref{thm:strong_compact_fami}, one can construct a density function $f_0$ on $X_0$ by $f_k$'s on nearby fibres and $\H_0(f_0\om_0^n) \leq \liminf_{k \to +\infty} \H_k(u_k) \leq B$ for some constant $B$.
Let $\vph_0$ be a potential solving $(\om_0 + \ddc_0 \vph_0)^n = f_0 \om_0^n$ with $\int_{X_0} \vph_0 \om_0^n = 0$.
From the proof of Theorem~\ref{thm:strong_compact_fami}, we have obtained a subsequence of $(u_k)_k$, which converges strongly in families towards $\vph_0$ and thus $\vph_0 = u_0$. 
Fix a continuous function $\eta$ on $\CX$. 
Set $\eta_t = \eta_{|X_t}$.
We compute
{\small
\begin{align*}
    &\abs{\int_{X_k} \eta_k \om_{k,u_k}^n - \int_{X_0} \eta_0 \om_{0,u_0}^n}\\ 
    &\leq \abs{\int_{U_0^\vep} (F^\ast_k\eta_k \cdot F_k^\ast f_k -  \eta_0 f_0) \om_0^n} 
    + \abs{\int_{U_0^\vep} F_k^\ast(\eta_k f_k) (F_k^\ast\om_k^n - \om_0^n)}\\
    &\quad+ \abs{\int_{X_k \setminus U_k^\vep} \eta_k f_k \om_k^n} + \abs{\int_{X_0 \setminus U_0^\vep} \eta_0 f_0 \om_0^n}\\
    &\leq \int_{U_0^\vep} \abs{F^\ast_k \eta_k - \eta_0} F_k^\ast f_k \om_0^n + \abs{\int_{X_0} \eta_0 (F_k^\ast f_k - f_0) \om_0^n} + \abs{\int_{U_0^\vep} F_k^\ast(\eta_k f_k) (F_k^\ast\om_k^n - \om_0^n)} \\ 
    &\quad + \abs{\int_{X_k \setminus U_k^\vep} \eta_k f_k \om_k^n} + \abs{\int_{X_0 \setminus U_0^\vep} \eta_0 f_0 \om_0^n}.  
\end{align*}
}%
When $k \to +\infty$, the first part tends to zero by the continuity of $\eta$. 
The second part converges to zero since $f_0$ is a weak limit of $F_k^\ast f_k$.
The third term goes to zero by the $L^1$ control on $f_k$ and the comparison of $F_k^\ast \om_k^n$ and $\om_0^n$. 
By H\"older--Young inequality, one can also check that the fourth and fifth terms converge to $0$ after letting $k\to+\infty$ and then $\vep \to 0$.
\end{proof}

\subsubsection{Strong lower semi-continuity of entropy}
We now prove the lower semi-continuity of entropy with respect the strong convergent sequences in families. 

\begin{lem}\label{lem:semi_conti_entropy}
If $(u_k)_k \in \CE^1_{\fibre}(\CX, \om)$ converges strongly to $u_0 \in \CE^1(X_0,\om_0)$, then 
\[
    \H_0(u_0) \leq \liminf_{k \to +\infty} \H_k(u_k).
\]
\end{lem}
\begin{proof}
Without loss of generality, we may assume that $\liminf_{k \to +\infty} \H_k(u_k) < +\infty$; otherwise, we are done.
Hence, up to extracting a subsequence, there is a constant $B>0$ such that $\H_k(u_k) \leq B$ for every $k$. 
One may also assume that $\int_{X_k} u_k \om_k^n = 0$. 
Following the same argument at the beginning of the proof of Theorem~\ref{thm:strong_compact_fami}, one can find a density function $f_0$ on $X_0$ such that $f_0 \om_0^n$ is the weak limit of $f_{k,0} \om_0^n$ and 
\[
    \H_0(f_0\om_0^n) 
    \leq \liminf_{k \to +\infty} \H_0(f_{k,0} \om_0^n) 
    \leq \liminf_{k \to +\infty} \H_k(u_k).
\]
Let $\vph_0$ be a potential solving $(\om_0 + \ddc_0 \vph_0)^n = f_0 \om_0^n$ with $\int_{X_0} \vph_0 \om_0^n = 0$.
From the proof of Theorem~\ref{thm:strong_compact_fami}, we have obtained a subsequence of $(u_k)_k$ which converges strongly in families towards $\vph_0$. 
Therefore, $\vph_0$ is identically equal to $u_0$, and this completes the proof.
\end{proof}

\subsubsection{Strong compactness with respect to adapted measures}
We now suppose that the family $\pi: \CX \to \BD$ satisfies Setting~\ref{sett:klt}.
Recall that for any smooth hermitian metric $h$ on $m K_{\CX/\BD}$, on each $X_t$, we have adapted measure $\mu_t$. 
Up to a positive multiple, one can further assume that $\mu_t$'s have total mass $V_t$.
From Lemma~\ref{lem:compare_klt_entropy} and the construction of adapted measures, we immediately get the following lemma: 

\begin{lem}\label{lem:compare_klt_entropy_fami}
Under Setting~\ref{sett:klt}, fix a smooth hermitian metric $h$ on $m K_{\CX/\BD}$. 
Let $(\mu_t)_t$ be the adapted measures induced by $h$ and normalized by mass $V$. 
Then
\[
    (\H_{k}(u_k))_{k} \text{ is bounded}
    \iff 
    (\H_{k, \mu_k}(u_k))_{k} \text{ is bounded}.
\]
\end{lem}

Combining Theorem~\ref{thm:strong_compact_fami} and Lemma~\ref{lem:compare_klt_entropy_fami}, we obtain:

\begin{cor}\label{cor:strong_compact_adapted_fami}
Under Setting~\ref{sett:klt}, fix a smooth hermitian metric $h$ on $m K_{\CX/\BD}$. 
Let $(\mu_t)_t$ be the adapted measures of mass $V$ induced by $h$. 
Let $(u_k)_k \in \PSH_\fibre(\CX,\om)$ be a sequence with $t_k \to 0$ as $k \to +\infty$ and $(\sup_{X_k} u_k)_k$ uniformly bounded.
If $(\H_{k,\mu_k}(u_k))_k$ is uniformly bounded, then there exists a subsequence of $(u_k)_k$ converging strongly in families to a function $u_0 \in \CE^1(X_0, \om_0)$.
\end{cor}

\subsubsection{Strong lower semi-continuity with respect to adapted measures.}
\begin{lem}\label{lem:lsc_ent_mu}
Under Setting~\ref{sett:klt}, fix a smooth hermitian metric $h$ on $m K_{\CX/\BD}$. 
Let $(\mu_t)_t$ be the adapted measures of mass $V$ induced by $h$. 
If $(u_k)_k \in \CE^1_{\fibre}(\CX,\om)$ converges strongly to $u_0 \in \CE^1(X_0,\om_0)$ in the family sense, then 
\[
    \H_{0,\mu_0}(u_0) \leq \liminf_{k \to +\infty} \H_{k, \mu_k}(u_k).
\]
\end{lem}

\begin{proof}
Without loss of generality, we may suppose that $(\H_{k,\mu_k}(u_k))_k$ is uniformly bounded.
Set $m_t$ the $L^p$-density so that $\mu_t = m_t \om_t^n$.
One has
\[
    \H_{t,\mu_t}(u_t) = \H_t(u_t) - \frac{1}{V_t} \int_{X_t} \log(m_t) \om_{t, u_t}^n. 
\]
From Lemma~\ref{lem:semi_conti_entropy}, we have obtained $\H_0(u_k) \leq \liminf_{k \to +\infty} \H_k(u_k)$.
Therefore, to obtain the lower semi-continuity of entropy with respect to adapted measures, it suffices to show that 
\[
    \int_{X_k} \log(m_k) \om_{k, u_k}^n 
    \to \int_{X_0} \log(m_0) \om_{0, u_0}^n
\]
as $k \to +\infty$.
Write $f_k$ the density of $\om_{k,u_k}^n$ with respect to $\mu_k$.
From Lemma~\ref{lem:compare_klt_entropy_fami}, since $(\H_{k,\mu_k}(u_k))_k$ is bounded, $(\H_k(u_k))_k$ and $(\int_{X_k} \log(m_k) \om_{k, u_k}^n)_k$ are bounded as well.
Let $B$ be a constant such that $\H_k(u_k) \leq B$ and $\H_{k,\mu_k}(u_k) \leq B$ for all $k$.
Taking $L > 0$, we cut the integral into the following two parts
\[
    \int_{X_k} \log(m_k) \om_{k, u_k}^n
    = \int_{X_k} \min\{\log(m_k),L\} \om_{k,u_k}^n + \int_{X_k} (\log(m_k) - L)_+ \om_{k,u_k}^n.
\]
By the construction of $m_t$, for each fixed $L>0$, there is a continuous function $M_L$ on $\CX$ such that ${M_L}_{|X_t} = \min\{\log(m_t) ,L\}$. 
Therefore, by Lemma~\ref{lem:weak_conv_MA}, we obtain the continuity of the first integral 
$$
    \lim_{k \to +\infty}  \int_{X_k} \min\{\log(m_k), L\} \om_{k,u_k}^n
    = \int_{X_0} \min\{\log(m_0), L\} \om_{0,u_0}^n. 
$$
Now, we deal with the second term. 
By Lebesgue integral formula, the second term can be expressed as
\begin{align*}
    \int_{X_k} (\log(m_k) - L)_+ \om_{k,u_k}^n
    = \int_0^{+\infty} \om_{k,u_k}^n (\{\log(m_k) > s + L\}) \dd s.
\end{align*}
We now follow the idea as in \cite[Lem.~2.4]{DiNezza_Lu_2022} to estimate $\om_{k,u_k}^n (\{\log(m_k) > s + L\})$.
Let $E \subset X_k$ be a Borel set with $\Vol_{\om_k} (E) > 0$ and $A := (\Vol_{\om_k}(E)/V_k)^{-1/2}$. 
Using the estimate in the proof of \cite[Lem.~2.4]{DiNezza_Lu_2022}, we have
\begin{equation}\label{eq:conti_Ric_ent_1}
\begin{split}
    \om_{k,u_k}^n(E) 
    &\leq 2(B + V_k e^{-1}) (-\log(\Vol_{\om_k}(E)/V_k))^{-1}+ V (\Vol_{\om_k}(E)/V_k)^{1/2}.
\end{split}
\end{equation}
Fix $b = \sqrt{2}/\log(2)$ which is a numerical constant so that 
\begin{equation}\label{eq:conti_Ric_ent_2}
    (-\log(t))^{-1} \leq b t^{1/2} 
    \quad \text{for any } t \in \lt[0, 1/2\rt].
\end{equation}
By \eqref{eq:klt_Lp_estimate}, there exists uniform constants $p>1$ and $C_p > 0$, such that $\int_{X_t} e^{p \log(m_t)} \om_t^n \leq C_p$ for any $t \in \BD$.
Hence, for all $\ld \in\BN^\ast$, {\small
\begin{equation}\label{eq:conti_Ric_ent_3}
    \Vol_{\om_k}(\{\log(m_k) > s+L\})
    \leq \frac{1}{(s+L)^\ld} \int_{X_k} (\log(m_k)_+)^\ld \om_k^n
    \leq \frac{\ld! C_p}{p^\ld(s+L)^\ld}.
\end{equation}}
Choose $\ld = 3$, $L>0$ sufficiently large so that $\frac{6 C_p}{p^3 L^3} < 1/2$. 
Then combining \eqref{eq:conti_Ric_ent_1}, \eqref{eq:conti_Ric_ent_2} and \eqref{eq:conti_Ric_ent_3}, one can derive  {\small
\begin{align*}
    \int_{X_k} (\log(m_k) - L)_+ \om_{k,u_k}^n
    &= \int_0^{+\infty} \om_{k,u_k}^n (\{\log(m_k) > s + L\}) \dd s\\
    &\leq \frac{6C_p (2b(B + V_k e^{-1}) +V_k)}{p^3 V_k^{1/2}} \int_0^{+\infty} \frac{1}{(s+L)^{3/2}} \dd s 
    = \frac{D}{L^{1/2}}
\end{align*} }
where $D = \frac{12C_p (2b(B + V_k e^{-1}) + V_k)}{p^3 V_k^{1/2}}$. 
Similarly, we also have $\int_{X_0} (\log(m_0) - L)_+ \om_{0,u_0}^n \leq \frac{D}{L^{1/2}}$ as the entropy $\H_0(u_0) \leq B$ by Lemma~\ref{lem:semi_conti_entropy}.
This completes the proof.
\end{proof}

\subsubsection{Strong continuity of twisted energy}

Let $\eta$ be a smooth $(1,1)$-form on $\CX$. 
Recall that the $\eta$-energy on each fibre $X_t$ is defined as
\[
    \E_{t, \eta_t}(u_t) := \frac{1}{nV} \sum_{j = 0}^{n-1} \int_{X_t} u_t \eta_t \w (\om_t + \ddc_t u_t)^j \w \om_t^{n-1-j} 
\]
where $\eta_t = \eta_{|X_t}$.

\begin{lem}\label{lem:continuity_twisted_energy} 
If $(u_k)_k \in \CE^1_{\fibre}(\CX,\om)$ converges strongly to $u_0 \in \CE^1(X_0,\om_0)$ in the family sense and $(\H_k(u_k))_k$ is bounded, then $(\E_{k,\eta_k}(u_k))_k$ converges to $\E_{0,\eta_0}(u_0)$. 
\end{lem}

\begin{proof}
Following the construction in the proof of Theorem~\ref{thm:strong_compact_fami}, recall that we have approximations $\vph_k^\bt$ such that
\begin{itemize}
    \item For each $\bt$ fixed, $(\vph_k^\bt)_k$ are uniformly bounded; 
    \item As $k \to +\infty$, $\vph_k^\bt$ converges smoothly to $\vph_0^\bt$ in the family sense; 
    \item As $\bt \to 0$, $\vph_0^\bt$ converges strongly to $u_0$; 
    \item $\lim_{\bt \to 0} \limsup_{k \to +\infty} \I_k(u_k, \vph_k^\bt) = 0$.
\end{itemize}
One can deduce $\lim_{k \to +\infty} \E_{k,\eta_k} (\vph_k^\bt) = \E_0(\vph_k^\bt)$ by the uniform control of $(\vph_k^\bt)_k$ and the smooth convergence in families.
By Lemma~\ref{lem:Twisted Energy}, we have the following estimate
{\small
\begin{align*}
    &|\E_{k,\eta_k} (u_k) - \E_{0,\eta_0} (u_0)|\\ 
    &\leq |\E_{k,\eta_k} (u_k) - \E_{k,\eta_k}(\vph_k^\bt)| 
    + |\E_{k,\eta_k} (\vph_k^\bt) - \E_{0,\eta_0}(\vph_0^\bt)| 
    + |\E_{0,\eta_0} (u_0) - \E_{0,\eta_0}(\vph_0^\bt)|\\
    &\leq C (f_R(\I_k(u_k, \vph_k^\bt)) + \|u_k-\vph_k^\bt\|_{L^1(X_k, \om_k^n)} + f_R(\I_0(u_0, \vph_0^\bt)) + \|u_0 - \vph_0^\bt\|_{L^1(X_0, \omega_0^n)})\\ 
    &\qquad+ |\E_{k,\eta_k} (\vph_k^\bt) - \E_{0,\eta_0}(\vph_0^\bt)|.
\end{align*}
}%
All in all, the result follows from $\lim_{\bt \to 0} \limsup_{k \to +\infty} \I_k(u_k, \vph_k^\bt) = 0$, and the strong convergence of $\vph_0^\bt$ to $u_0$.
\end{proof}

\section{A variational approach for singular cscK metrics}
\label{sect_singular_cscK}

In this section, we first give a definition for the singular cscK metrics and extend the variational approach to the singular setting. 
Then we prove Theorem~\ref{bigthm:openness_coercivity}. 

We now assume $(X,\om)$ is a normal compact K\"ahler variety with klt singularities.
Let $\mu_h$ be a mass $V$ adapted measure associated with a hermitian metric $h$ on $mK_X$, which is Cartier. 
\begin{defn}\label{defn:cscK}
We say that $\om_\vph$ is a singular cscK metric if $\om_\vph$ is a genuine Kähler metric on $X^\reg$ and $\vph \in \PSH(X,\om) \cap L^\infty(X)$ satisfies the following equation on $X^\reg$
\[
    \bar{s} \om_\vph^n = n \Ric(\om_\vph) \w \om_\vph^{n-1},
\]
where  $\Ric(\om_\vph) = -\ddc \log \lt(\frac{\om_\vph^n}{\mu_h}\rt) + \Ric(\mu_h)$ and $ \bar{s} = n \frac{c_1(X) \,\cdot\, [\om]^{n-1}}{[\om]^n} = \frac{n}{V} \int_X \Ric(\mu_h)\wedge \omega^{n-1}$.
For the precise definition of $c_1(X)$, we refer to \cite[Def.~5.11]{EGZ_2009}.
\end{defn}

\subsection{Mabuchi functional}\label{ssec:Mabuchi}
In this section, we define the Mabuchi functional with respect to $(X,\om)$.

\subsubsection{Globally bounded potentials}
We first define the Mabuchi functional for functions $u\in \PSH(X,\om)\cap L^\infty(X)$.

\begin{defn}
The Mabuchi functional for $(X,\om)$ is $\M:\PSH(X,\om) \cap L^\infty(X) \to \BR \cup \{+\infty\}$ defined as
\[
    \M(u):=\H(u)+\bar{s}_\om\E(u)-n\E_{\Ric} (u),
\]
where $\E$ is the Monge--Amp\`ere energy,  $\H$ is the entropy, $\bar{s}_\om := \frac{n}{V} \int_X \langle \Ric(\om)\wedge \om^{n-1} \rangle$
is the mean value of the scalar curvature with respect to $\om$ and where
\[
    \E_{\Ric}(u) := \frac{1}{n V} \sum_{j=0}^{n-1} \int_X u \langle \Ric(\om) \w \om_u^j \w \om^{n-1-j} \rangle
\]
is the twisted Ricci energy.
Note that $\langle \,\, \rangle$ denotes the non-pluripolar product. 
\end{defn}
We will observe in the proof of Proposition \ref{prop:Change metrics bounded case Mabuchi} that one necessarily has $\bar{s}_\om = \bar{s} = n\frac{c_1(X)\cdot [\om]^{n-1}}{[\om]^n}$.

\subsubsection{A particular change of reference metric}
Let $m\geq 1$ be an integer such that $mK_X$ is locally free, let $h^m$ be a smooth metric on $mK_X$ and let $\mu_h$ be the associated adapted measure (see Section \ref{sect_adapted_measure}).
Let $\Ta:=\Ric(\mu_h)\in c_1(X)$ be the Ricci form of the smooth metric $h$, and recall that $\Ric(\omega)=\Theta+\ddc\log f$ where $A\om - \ddc \log f\geq 0$. Let also $\psi \in \PSH(X,\om)$ such that
$$
(\om+\ddc \psi)^n=\mu_h,
$$
i.e. $\Ric(\om_\psi)=\Ric(\mu_h)=\Theta$. Note that $\psi$ is globally bounded as consequence of \cite[Thm.~B]{BEGZ_2010}. Set $\wom:=\om+\ddc \psi$.

Letting $\PSH(X,\om) \ni u \to \wu := u-\psi \in \PSH(X,\wom)$ be the bijection corresponding to the change of reference form, the Mabuchi functional can be defined for globally bounded functions in $\CE^1(X,\wom)$ as
\[
\wM(\wu):=\H_\mu(\wu)+\Bar{s}_\psi\wE(\wu)-n \wE_\Theta(\wu)
\]
where 
\begin{enumerate}
    \item $\mu:=\mu_h$ is the adapted measure;
    \item $\wE(\wu):= \frac{1}{(n+1)V}\sum_{j=0}^n \int_X \wu \wom_{\wu}^j \wedge \wom^{n-j}$ is the Monge--Amp\`ere energy with respect to $\wom$;
    \item $\wE_\Theta(\wu):= \frac{1}{nV}\sum_{j=0}^{n-1} \int_X \wu \Theta \wedge \wom_{\wu}^j \wedge \wom^{n-j-1}$ is the twisted Ricci energy with respect to $\wom$;
    \item $\bar{s}_\psi:= n\frac{\int_X \Theta\wedge \om_\psi^{n-1}}{\int_X \om_\psi^n} = \bar{s}$ is the mean value of the scalar curvature with respect to $\om_\psi$;
\end{enumerate}
Moreover, a direct computation leads to
\begin{equation}
    \label{eqn:Useful Formula}
    \wM(\wu)=\H_\mu(u)+\bar{s}\big(\E(u)-\E(\psi)\big)-n\big(\E_\Theta(u)-\E_\Theta(\psi)\big)
\end{equation}
for any $u\in \PSH(X,\om) \cap L^\infty(X)$.
We now compare $\M$ and $\wM$.
\begin{prop}\label{prop:Change metrics bounded case Mabuchi}
In the setting just described, for any $u \in \PSH(X,\om) \cap L^\infty(X)$,
\begin{equation}\label{eqn:Change_Reference_Bounded}
	\wM(\wu)=\M(u)-\M(\psi).
\end{equation}
\end{prop}
Note that the klt condition implies that $\om_\psi^n=\mu_h$ has $L^p$-density with respect to $\om^n$ for $p>1$.  In particular $\psi$ has finite entropy with respect to $\om^n$.  Hence $\M(\psi)$ is a finite quantity. 
\begin{proof}
Let $f$ such that $\mu_h= f\om^n$. Recall that $f\in \CC^\infty(X^\reg)$ and $X^\sing$ is contained in an analytic set of codimension at least 2 by the normality of $X$.

As said above we have $\Ric(\om)=\Theta+\ddc \log f$ and there exists $A\gg 1$ large enough such that $A\om-\ddc\log f \geq 0$.
\begin{claim}\label{claim:NP vs BT}
The following equality holds for any $j=0, \cdots, n-1$ and any $u \in \PSH(X,\om) \cap L^\infty(X)$
\[
    \langle \big(A\om -\ddc \log f\big)\wedge \om_u^j\wedge \om_\psi^{n-j-1}\rangle=\big(A\om -\ddc \log f\big)\wedge \om_u^j\wedge \om_\psi^{n-j-1}.
\]
\end{claim}
In Claim \ref{claim:NP vs BT}, the RHS is the Bedford--Taylor product, which is well defined as $u, \psi$ are bounded. 
We first assume the claim and conclude the proof. 
We obtain
\begin{align*}
    \Bar{s}_\om 
    &= \frac{n}{V}\int_X \langle \Ric(\om)\wedge \om^{n-1}\rangle\\ 
    &= \frac{n}{V}\int_X (\Theta+A\omega)\wedge \om^{n-1}-\frac{n}{V}\int_X \langle (A\om-\ddc \log f)\wedge \om^{n-1}\rangle\\
    &= \frac{n}{V}\int_X (\Theta+A\omega)\wedge \om^{n-1}-\frac{n}{V}\int_X (A\om-\ddc \log f)\wedge \om^{n-1}\\
    &= \frac{n}{V}\int_X \Theta\wedge \om^{n-1}=\Bar{s}.
\end{align*}
Similarly, by approximation and \cite[Thm 2.6]{Demailly_1985}, after integration by parts, we get
$$
\sum_{j=0}^{n-1}\int_X (u-\psi) \ddc \log f \wedge \om_u^j \wedge \om_\psi^{n-j-1}=\int_X \log f \om_u^n- \int_X \log f \om_\psi^n,
$$
which leads to
{\small
\begin{align*}
    & n \E_{\Ric}(u) - n \E_{\Ric}(\psi)
    = \frac{1}{V} \sum_{j=0}^{n-1} \int_X (u-\psi)\langle(\Theta + \ddc \log f) \wedge \om_u^j \wedge \om_\psi^{n-j-1} \rangle\\
    &= \frac{1}{V} \sum_{j=0}^{n-1} \int_X (u-\psi)(\Theta + A \omega) \wedge \om_u^j \wedge \om_\psi^{n-j-1}\\ 
    &\qquad - \frac{1}{V} \sum_{j=0}^{n-1} \int_X (u-\psi) \langle(A\om -\ddc \log f) \wedge \om_u^j \wedge \om_\psi^{n-j-1}\rangle\\
    &= \frac{1}{V} \sum_{j=0}^{n-1} \int_X (u-\psi)\Theta \wedge \om_u^j \wedge \om_\psi^{n-j-1} + \frac{1}{V}\int_X \log f \om_u^n - \frac{1}{V}\int_X \log f \om_\psi^n\\
    &= n \E_\Theta(u) - n\E_\Theta(\psi) + \frac{1}{V}\int_X \log f \om_u^n - \H(\psi)
\end{align*}
}%
where in the last equality we recognized the entropy of $\mu_h=\om_\psi^n$ with respect to $\om^n$.

Next, as $\H(u)=\H_\mu(u)+ \frac{1}{V}\int_X \log f \om_u^n$ for any $u$ with finite entropy, we deduce that
\begin{align*}
    \M(u) - \M(\psi)
    &= \H(u) - \H(\psi) + \Bar{s} \big(\E(u)-\E(\psi)\big) - n \big(\E_{\Ric}(u) - \E_{\Ric}(\psi)\big)\\
    &=\H_\mu(u) + \Bar{s} \big(\E(u) - \E(\psi)\big) -n \big(\E_\Theta(u) - \E_\Theta(\psi)\big)=\wM(\wu),
\end{align*}
where the last equality is (\ref{eqn:Useful Formula}).

In remains to check Claim \ref{claim:NP vs BT}. 
Since the two measures coincide on $X^\reg$, it suffices to verify the RHS puts no mass on $X^\sing$. 
Note that $- \log f$ has analytic singularities given by an ideal sheaf $\CI$ vanishing over $X^\sing$. 
Let $p: Y \to X$ be a log-resolution of $(X,\CI)$ and $E_i$'s be irreducible components of the exceptional divisor.
One has $p^\ast(A \om - \ddc \log f) = \ta + \sum_i \ld_i [E_i]$ for some $\ta$ smooth and $\ld_i \geq 0$. 
Hence,  {\small
\[
    \int_X (A \om - \ddc \log f) \w \om_u^j \w \om_\psi^{n-j} = \int_Y \ta \w p^\ast(\om_u^j \w \om_\psi^{n-1-j}) + \sum_i \ld_i \int_{E_i} p^\ast (\om_u^j \w \om_\psi^{n-j}). 
\]}
For each $i$, $\int_{E_i} p^\ast (\om_u^j \w \om_\psi^{n-1-j}) = \int_{E_i} p^\ast \om^{n-1} = 0$; thus the proof completes.
\end{proof}

\subsubsection{The extended Mabuchi functional}
The following result is the analog of \cite[Lem.~3.1]{Berman_Darvas_Lu_2017} for singular varieties.
\begin{lem}
    \label{lem:BDL17}
    Let $u\in \CE^1(X,\om)$. Then there exists $(u_k)_k\in \CE^1(X,\om)\cap L^{\infty}(X)$ such that $u_k$ strongly converges to $u$, $\H(u_k)\to \H(u)$ and $\H_\mu(u_k)\to \H_\mu(u)$.
\end{lem}
One can derive the continuity of $\H$ as a consequence of Lemma~\ref{lem:regularization_fini_entropy}. 
To check the continuity of $\H_\mu$, one can simply follow the same argument as in the proof of Lemma~\ref{lem:lsc_ent_mu}.

Similarly to the smooth setting, Lemma \ref{lem:Twisted Energy} implies that the twisted Ricci energy $\E_\Theta$ extends to a strongly continuous functional on $\CE^1(X,\om)$. 
Thus $\wM$ extends to a functional on $\CE^1(X,\wom)$, and Proposition \ref{prop:Change metrics bounded case Mabuchi} leads to the following result.
\begin{prop}\label{prop:Change Reference Final case}
    The Mabuchi functional uniquely extends to a translation invariant and strongly lower semi-continuous functional $\M:\CE^1(X,\om)\to \BR \cup \{+\infty\}$. 
    Moreover,
    \begin{equation}
	   \label{eqn:Change_Reference}
	   \wM(\wu)=\M(u)-\M(\psi)
    \end{equation}
    for any $u\in \mathcal{E}^1(X,\om)$. 
    In particular, finite entropy potentials $u\in \mathcal{E}^1(X,\om)$ have finite twisted Ricci energy $\E_{\Ric}(u)>-\infty$.
\end{prop}

\begin{proof}
    By Lemma \ref{lem:BDL17}, for any $u\in \mathcal{E}^1(X,\om)$, we have a sequence $(u_k)_k\in \CE^1(X,\om)\cap L^\infty(X)$ such that $u_k$ strongly converges to $u$, $\H(u_k)\to \H(u)$ and $\H_\mu(u_k)\to \H_\mu(u)$. 
    In particular, $\wM(\wu_k)\to \wM(\wu)$. 
    If $u$ does not have  finite entropy, Proposition \ref{prop:Change metrics bounded case Mabuchi} implies that $\M(u_k)\to +\infty$ (see also Lemma \ref{lem:compare_klt_entropy}), while
    $$
    \M(u_k)\longrightarrow\wM(\wu)+\M(\psi)
    $$
    if $u$ has finite entropy. As the entropies, the Monge--Amp\`ere energy and the twisted Ricci energy $\E_\Theta$ are continuous along the sequence $(u_k)_k$, one can deduce that $\E_{\Ric} (u_k)\to \E_{\Ric} (u)$. 
    Hence, we have $\E_{\Ric} (u)>-\infty$ if $u$ has finite entropy and
    $$
    \M(u_k)\longrightarrow \M(u):=
    \begin{cases}
        \H(u)+\Bar{s}\E(u)-n\E_{\Ric} (u) &\text{ if } \H(u)<+\infty,\\
        +\infty &\text{ otherwise.}
    \end{cases},
    $$
    The equation (\ref{eqn:Change_Reference}) follows.

Next, the translation invariance of $\M$ is easy to verify, and (\ref{eqn:Change_Reference}) implies that $\M$ is strongly lower semi-continuous if and only if $\wM$ is strongly lower semi-continuous. 
To conclude, we recall that $\E$, $\E_\Theta$ are strongly continuous (see also Lemma \ref{lem:Twisted Energy}), and the entropy $\H_\mu$ is lower semi-continuous along strong convergent sequences (Section \ref{ssec:entropy}). 
\end{proof}

\subsubsection{Minimizers and coercivity}

We say that the Mabuchi functional $\M$ is \emph{coercive} if there exist $A > 0, B \geq 0$ such that
$$
    \M(u) \geq - A \E(u) - B
$$
for any $u\in \CE^1_\nmlz(X,\om)$.  The constant $A>0$ is said to be the \emph{slope} of $\M$. 

Recall that, given $u_0,u_1\in \CE^1(X,\om)$ a map $(0,1)\ni t\to u_t\in \CE^1(X,\om)$ is a weak subgeodesic segment joining $u_0,u_1$ if  $\limsup_{t\to 0^+}u_t\leq u_0$, and $\limsup_{t\to 1^-} u_t\leq u_1$ and if the function
$$
U(z,\tau):=u_{-\log \lvert \tau\rvert}(z)
$$
is $p_1^*\om$-psh on $X\times \{\tau\in \mathbb{C}^* \, : \, -\log\lvert \tau\rvert\in (0,1) \}$ where $p_1$ is the projection to the first component. The largest weak subgeodesic segment joining $u_0,u_1$ is called \emph{weak geodesic} and it exists (see e.g. \cite{Darvas_2017, Dinezza_Guedj_2018}).

The following convexity result is a key point of the variational approach.
\begin{prop}
\label{prop:Geod_Conv}
The Mabuchi functional $\M$ is convex along weak geodesic segments joining globally bounded functions in $\CE^1(X,\om)$.
\end{prop}
\begin{proof}
    Let $F:Y\to X$ be a resolution of singularities given by a sequence of blow-ups along smooth centers. As seen in Step 1 of Proposition \ref{prop:uniform_Lap_est} there exists $\ww \in \PSH(Y,\pi^*\omega)\cap \CC^\infty\big(Y\setminus \Exc(F)\big)$ and a Kähler form $\eta$ on $Y$ such that
    $$
    F^*\omega+\ddc \ww=\eta+\sum_{j=1}^m a_j[E_j]
    $$
    where $a_j>0$ and $(E_j)_{j=1,\dots,m}$ are the exceptional prime divisors. As the fibres of $F$ are connected $\ww = w\circ F$ for $w\in \PSH(X,\omega)$. Moreover the lift $\widetilde{\mu}$ of $\mu$, i.e. the pushforward by $\pi^{-1}$ of $\mathbf{1}_{X\setminus F\big(\Exc(F)\big)}\mu$, is given by  $ \prod_{j=1}^m \lvert s_j\rvert^{2b_j}_{h_j}\eta^n$ where $h_j$ are smooth hermitian metrics on $\mathcal{O}_Y(E_j)$, $s_j$ are holomorphic sections, and $b_j\in \BR_{>-1}$ (see also \cite[Lem.~3.2]{BBEGZ_2019}). 
    
    Set $v^+:=\sum_{\{j\,:\, b_j\geq 0\}}b_j\log \lvert s_j\rvert_{h_j}, v^-:=-\sum_{\{j\,:\, b_j<0\}}b_j\log \lvert s_j\rvert_{h_j}$. As $F^*(\omega_w)\geq \eta$, there exists a constant $B>0$ such that
    $$
    BF^*(\omega_w)+\ddc v^\pm\geq B\eta+\ddc v^\pm \geq 0.
    $$
    In particular the functions $v^\pm \circ F^{-1} + Bw\in \CC^\infty(X^\reg)$ extends to $B\omega$-psh functions $g^\mp$. Observe that by construction the measure $e^{g^+-g^-}\mu$ lifts to $ e^{v^--v^+}\widetilde{\mu}=\eta^n$.
    Let now $\hat{\phi}\in \PSH(X,\om) \cap L^\infty(X)$ be a solution to 
    $$
    (\om +\ddc \hat{\phi})^n=\frac{V}{\int_Y \eta^n}e^{g^+-g^-}\mu.
    $$
    Clearly $\phi:=\hat{\psi}\circ F\in \PSH(Y,F^*\omega)$ satisfies $(F^*\om+\ddc\phi)^n=\frac{V}{\int_Y\eta^n} \eta^n$. 
    
    Therefore we are in a similar situation of Propositions \ref{prop:Change metrics bounded case Mabuchi}, \ref{prop:Change Reference Final case}. Indeed, the same proofs yield
    $$
    \hat{\M}(\hat{u})=\wM(\wu)+C_1=\M(u)+C_2
    $$
    for some uniform constants $C_1,C_2$ and for any $u\in \CE^1(X,\om)$ where $\hat{u}:=u-\hat{\psi}\in \CE^1(X,\om_{\hat{\psi}})$ and where $\hat{\M}$ is the Mabuchi functional with respect to $(X,\om_{\hat{\psi}})$. 
    In particular, $\M$ is geodesically convex if and only if $\hat{\M}$ is geodesically convex. Hence, to conclude the proof it is enough to check that $\hat{\M}$ is geodesically convex.

    Set $\hat{\om}:=F^*\om_{\hat{\psi}}$. We claim that the pullback map $\CE^1(X,\om_{\hat{\psi}})\ni\hat{u} \mapsto v:=\hat{u}\circ F\in \CE^1(Y,\hat{\om})$ transforms the Mabuchi functional $\hat{\M}$ with respect to $(X,\om_{\hat{\psi}})$ into the Mabuchi functional $\M^Y$ with respect to $(Y,\hat{\om})$. The latter is defined by $\M^Y(v)=\H^Y(v)+\Bar{s}^Y \E^Y(v)-n \E_{\Ric(\eta)}^Y(v)$ where 
    \begin{enumerate}
        \item $\H^Y$ is the entropy with respect to $\frac{V}{\int_Y \eta^n}\eta^n=(F^*\om_{\hat{\psi}})^n$;
        \item $\Bar{s}^Y= \frac{c_1(Y)\cdot [\hat{\om}]^{n-1}}{[\hat{\om}]^n}$ is the average of the scalar curvature of $(Y,\hat{\om})$;
        \item $\E^Y(v):=\frac{1}{(n+1)V}\sum_{j=0}^n \int_Y v \, \hat{\om}^j\wedge \hat{\om}_v^{n-j} $ is the Monge--Amp\`ere energy with respect to $(Y,\hat{\om})$;
        \item $\E^Y_{\Ric(\eta)}(v):=\frac{1}{nV}\sum_{j=0}^{n-1} \int_Y v \, \Ric(\eta)\wedge\hat{\om}^j\wedge \hat{\om}_v^{n-j-1} $ is the twisted Ricci energy with respect to $(Y,\hat{\om})$, as $\Ric(\hat{\om})=\Ric(\eta)$.
    \end{enumerate}
Indeed, it is immediate to check that $\H^Y(v)=\H_{\om_{\hat{\psi}}^n}(\hat{u})$ where as above $v:=\hat{u}\circ F$. 
Similarly the properties of the non-pluripolar product gives $\E^Y(v)=\hat{\E}(\hat{u})$ where by $\hat{\E}$ we mean the Monge--Amp\`ere energy with respect to $(X,\om_{\hat{\psi}})$. 
Then the projection formula and \cite[Thm~2.4]{DDL_2018b} give $\bar{s}^Y=\frac{F_*c_1(Y)\cdot [\om]^{n-1}}{[\om]^n}=\Bar{s}$. 
Finally, as $F^*\Ric(\om_{\hat{\psi}})= \Ric(\eta)$ over $Y\setminus \Exc(F)$ and the non-pluripolar product does not put mass on pluripolar sets, we gain $\E^Y_{\Ric(\eta)}(v)=\hat{\E}_{\Ric(\om_{\hat{\psi}})}(\hat{u})$. Hence $\M^Y(v)=\hat{\M}(\hat{u})$, i.e. the claim is proved.

    Next, the pullback map also produces a bijection among weak geodesic segments, i.e. it can be easily checked that $(\hat{u}_t)_{t\in(0,1)}$ is a weak geodesic segment joining $\hat{u}_0,\hat{u}_1\in \CE^1(X,\om_{\hat{\psi}})$ if and only if $(v_t)_{t\in (0,1)}:=(\hat{u}_t\circ F)_{t\in(0,1)}$ is a weak geodesic segment joining $v_0:=\hat{u}_0\circ F, v_1:=\hat{u}_1\circ F\in \CE^1(X,\om_{\hat{\psi}})$. 
    Therefore the convexity of $\hat{\M}$ along weak geodesics is equivalent to the convexity of $\M^Y$ along weak geodesics. 
     It follows from an extension of the result by Berman--Berndtsson \cite{Berman_Berndtsson_2017} on the convexity of Mabuchi functional (see \cite[Thm.~4.2]{DiNezza_Lu_2022}) that $\M^Y$ is convex along weak geodesics joining globally bounded functions with finite entropy. 
    This clearly implies that $\M^Y$ is convex along weak geodesics joining globally bounded functions, which concludes the proof.
\end{proof}

\begin{rmk}
As a consequence of Proposition \ref{prop:Change Reference Final case} and with the same notations,  $\M$ is coercive with slope $A>0$ if and only if $\wM$ is coercive with slope $A>0$.  
Moreover, combining Proposition \ref{prop:Change Reference Final case} with Proposition \ref{prop:Geod_Conv} gives the convexity of $\wM$ along weak geodesic segments in $\CE^1(X,\wom)$.
\end{rmk}

The existence of minimizers of $\M$ relates to its coercivity as follows: 
\begin{thm}
\label{thm:Coercivity and minimizers Singular Setting}
The implications $(i)\Rightarrow (ii)\Rightarrow (iii)$ hold regarding the following statements.
\begin{enumerate}
    \item $\M$ admits a unique minimizer;
    \item $\M$ is coercive;
    \item $\M$ admits a minimizer.
\end{enumerate}
\end{thm}

\begin{proof}
    \textbf{Proof of $(i)\Rightarrow (ii)$.} We follow the same strategy of \cite{Darvas_Rubinstein_2017}. Assume $\M$ admits a unique minimizer $u\in \CE^1(X,\omega)$. Set 
    {\small
    $$
    A:=\inf\Big\{\frac{\M(v)-\M(u)}{d_1(u,v)}\, : \, d_1(u,v)\geq 1, v\in \PSH(X,\omega)\cap L^\infty(X), \sup_X v=0\Big\}\in \BR_{\geq 0}.
    $$}
    Letting $v\in\PSH(X,\omega)\cap L^\infty(X) $ with $\sup_Xv=0$, the triangle inequality yields 
    $
    \M(v)\geq Ad_1(v,0)-B
    $
    for any $v\in\PSH(X,\omega)\cap L^\infty(X)$ with $\sup_X v=0$ setting $B:=Ad_1(u,0)+A-\M(u)$. Lemma \ref{lem:BDL17} then gives that any $v\in \CE^1_\nmlz(X,\omega)$ can be strongly continuously approximated by a sequence $(v_k)_k\in \CE^1_\nmlz(X,\omega)\cap L^\infty(X)$ such that $\M(v)=\lim_{k\to +\infty}\M(v_k)$ (see also the proof of Proposition \ref{prop:Change Reference Final case}). We deduce that
    $
    \M(v)\geq Ad_1(v,0)-B
    $
    for any $v\in \CE^1_\nmlz(X,\omega)$. Thus it remains to prove that $A>0$.\\
    Suppose by contradiction there exist $(v_k)_k\in \CE^1_\nmlz(X,\omega)\cap L^\infty(X)$ such that $d_1(v_k,u)\geq 1$ and such that
    $$
    \frac{\M(v_k)-\M(u)}{d_1(u,v_k)}\longrightarrow 0.
    $$
    We first note that $d_1(u,v_k)\to +\infty$. Indeed otherwise, unless considering a subsequence, it would follows that
    $$
    0\leq \frac{1}{C_1}\big(\M(v_k)-\M(u)\big)\leq\frac{\M(v_k)-\M(u)}{d_1(u,v_k)}\longrightarrow 0
    $$
    for a constant $C_1\geq 1$.
    Thus $\M(v_k)\leq C_2$ would give $v_k\to v$ strongly and $\M(v)\leq \liminf_{k\to +\infty}\M(v_k)= \M(u),$ by combining the strong compactness of Theorem \ref{thm:strong_compact_fami} and the strong lower semi-continuity of $\M$ (Proposition \ref{prop:Change Reference Final case}). However, $u$ is the unique minimizer of $\M$ and $d_1(u,v)=\lim_{k\to +\infty} d_1(u,v_k)\geq 1$ provides a contradiction.\\
    Then let $(u_j)_j\in\CE^1_\nmlz(X,\omega)\cap L^\infty(X)$ such that $u_j\to u$ strongly and $\M(u_j)\to\M(u)$ (Lemma \ref{lem:BDL17}). For $j\gg 1$ such that $d_1(u_j,u)\leq 1/2$, let also $w_{k,j}\in \CE^1_\nmlz(X,\omega)$ be the element on the unit-speed geodesic joining $u_j$ and $v_k$ such that $d_1(u_j,w_{k,j})=1/2$. By convexity of $\M$ (Proposition \ref{prop:Geod_Conv}) we gain {\small
    \begin{align*}
        &\M(w_{k,j}) 
        \leq \frac{1}{2d_1(v_k,u_j)}\M(v_k)+\Big(1-\frac{1}{2d_1(v_k,u_j)}\Big)\M(u_j) \\
        &= \frac{d_1(v_k,u)}{2d_1(v_k,u_j)}\Big(\frac{\M(v_k)-\M(u)}{d_1(u,v_k)}\Big)+\frac{1}{2d_1(v_k,u_j)}\M(u)+  \Big(1-\frac{1}{2d_1(v_k,u_j)}\Big)\M(u_j).
    \end{align*} }
    As $d_1(v_k,u)\leq d_1(v_k,u_j)+1/2$, we obtain that $\M(w_{k,j})\leq C$ uniformly in $k$ and in $j$. As before, from Theorem \ref{thm:strong_compact_fami} and Proposition \ref{prop:Change Reference Final case} we deduce that $w_{k,j}\to w_j$ strongly as $k\to +\infty$ and that $ \M(w_j)\leq\liminf_{k\to +\infty}\M(w_{k,j})\leq \M(u_j) $. 
    Thus, it follows that $w_j\to w $ strongly and that
    $$
    \M(w)\leq \liminf_{k\to +\infty}\M(w_j)\leq\lim_{k\to +\infty}\M(u_j)=\M(u).
    $$
    Since $u$ is the unique minimizer of $\M$, the contradiction follows from
    $$
    d_1(u,w)=\lim_{j\to +\infty}d_1(u_j,w_j)=\lim_{j\to +\infty}\lim_{k\to +\infty}d_1(u_j,w_{k,j})= 1/2.
    $$
    \textbf{Proof of $(ii)\Rightarrow (iii)$.} Letting $(u_k)_k\in \CE^1_\nmlz(X,\omega)$ such that $\M(u_k)\to \inf_{v\in\CE^1_\nmlz(X,\omega)}\M(v)$, we obtain $\M(u_k)\leq C_3$ uniformly and the coercivity gives $d_1(u_k,0)\leq C_4$ uniformly. Then, as seen in Propositions \ref{prop:Change metrics bounded case Mabuchi}, Proposition \ref{prop:Change Reference Final case} (and with the same notations),
    \begin{align*}
    -n\E_{\Ric}(u_k)+\H(u_k)=&-\E_\Theta(u_k)+\H_\mu(u_k)+n\E_\Theta(\psi)-n\E_{\Ric}(\psi)+\H(\psi)\\
    =&-n\E_\Theta(u_k)+\H_\mu(u_k)+C_5. 
    \end{align*}
    In particular, as $\E_\Theta(u_k)$ is controlled by $d_1(u_k,0)\leq C_4$, we deduce that
    $
    \M(u_k)\geq C_6+\H_\mu(u_k).
    $
    Hence the strong compactness of Corollary \ref{cor:strong_compact_adapted_fami} and the strong lower semi-continuity of Proposition \ref{prop:Change Reference Final case} give that $u_k\to u$ strongly and
    $$
    \M(u)\leq\liminf_{k\to +\infty}\M(u_k)=\inf_{v\in\CE^1(X,\omega)}\M(v),
    $$
    concludes the proof.
\end{proof}

\subsection{Proof of openness of coercivity (Theorem~\ref{bigthm:openness_coercivity})}
The strong topology has been exploited to prove the openness of the coercivity of the Ding functional, which is related to Fano K\"ahler--Einstein potentials under certain parameters variations (c.f. \cite[Thm.~B]{Trusiani_2022}, \cite[Thm.~C]{Trusiani_2023_Continuity}). 
In our setting, with the properties proved for the family of Mabuchi functional $(\M_t)_t$ with respect to the strong topology varying metric spaces, one can follow a similar strategy to prove our Theorem~\ref{bigthm:openness_coercivity}.

In the sequel, we shall always assume that $\pi: \CX \to \BD$ satisfies Setting~\ref{sett:klt}.
Fix a smooth hermitian metric $h$ on $m K_{\CX/\BD}$.
Take $\Ta \in c_1(-K_{\CX/\BD})$ the curvature form and $\mu_t$ the adapted measures induced by $h$ with normalized mass $V$ on each $X_t$. 
By \cite{EGZ_2009}, up to shrinking $\BD$, on each $X_t$, there exists a unique $\psi_t \in \PSH(X_t, \om_t) \cap L^\infty(X_t)$ solving
\[
    (\om_t + \ddc_t \psi_t)^n = \mu_t,
    \quad\text{and}\quad
    \int_{X_t} \psi_t \om_t^n = 0.
\]
Set $\wom_t = \om_t + \ddc_t \psi_t$.

We denote by $(\M_t)_t$ the family of Mabuchi functionals defined with respect to $(X_t,\om_t)$, and similarly to Section \ref{ssec:Mabuchi} with $(\wM_t)_t$ we indicate the family of Mabuchi functionals with respect to $(X_t,\om_{\psi_t})$.

\begin{prop}\label{prop:Lower Semicontinuity Mabuchi families}
    The family of Mabuchi functionals $(\M_t)_t$ is lower semi-continuous with respect to the strong topology in families. 
\end{prop}

\begin{proof}
    By construction there is a smooth closed $(1,1)$-form $\Theta$ such that $\Ric(\mu_t)=\Theta_t$ for any $t\in \BD$, where we clearly denoted by $\Theta_t$ the restriction of $\Theta$ to $X_t$.
    We first claim that $\M_k(\psi_k)\to \M_0(\psi_0)$. Proposition \ref{prop:Change Reference Final case} gives
    $$
    -\M_k(\psi_k)=\M_k(0)-\M_k(\psi_k)=\wM_k(\widetilde{0})=\H_{\mu_k}(0)-\Bar{s}_k \E_k(\psi_k)+n\E_{\Theta_k}(\psi_k)
    $$
    where the last equality follows from (\ref{eqn:Useful Formula}). As by Corollary \ref{cor:Strong continuity of psi} $\psi_k$ converges strongly to $\psi_0$ and obviously $(\H_k(\psi_k))_k$ is uniformly bounded, we deduce that $\E_k(\psi_k)\to \E_0(\psi_0)$ and $\E_{\Theta_k}(\psi_k)\to \E_{\Theta_0}(\psi_0)$ (Lemma \ref{lem:continuity_twisted_energy}). 
    Note also that $\Bar{s}_t=n\frac{c_1(X_t)\cdot [\om_t]^{n-1}}{[\om_t]^n}$ does not depend on $t$ by Remark \ref{rmk:constant_volume}. Thus to have $\M_k(\psi_k)\to \M_0(\psi_0)$ it remains to prove that $\H_{\mu_k}(0)\to \H_{\mu_0}(0)$ and this follows from the proof of Lemma \ref{lem:lsc_ent_mu}.

    Next, without loss of generality, one can assume $\liminf_{k\to+\infty} \M_k(u_k) \leq B$ for some $B > 0$, and then we consider a subsequence still denoted by $u_k$ so that $(\M_k(u_k))_k$ converges to the liminf of the original sequence. 
    Letting $C_1>0$ such that $\Theta_k\geq - C_1 \om_k$ and letting $v_k:=u_k-\sup_X u_k$ we have
    \begin{align*}
    nV_k \E_{k,\Theta_k}(v_k)
    &=\sum_{j=0}^{n-1}\int_{X_k}v_k \Theta_k \wedge \om_k^j \wedge \om_{k,v_k}^{n-j-1}\\
    &\leq -C_1 \sum_{j=0}^{n-1} \int_{X_k} v_k \om_k^{j+1}\wedge \om_{k,v_k}^{n-j-1}
    \leq -C_1 V_k(n+1) \E_k(v_k).    
    \end{align*}
    Thus, since by Proposition \ref{prop:Change Reference Final case} we have $\lim_{k\to +\infty}\wM_k(\wu_k)=\lim_{k\to +\infty}\M_k(u_k)-\M_0(\psi_0)\leq B-\M_0(\psi_0)$, combining (\ref{eqn:Useful Formula}) and $\E_k(u_k)\geq -C_2$ we deduce $\H_\mu(u_k)\leq C_3$ uniformly along this subsequence $(u_k)_k$.
    Hence Lemmas \ref{lem:lsc_ent_mu} and \ref{lem:continuity_twisted_energy} lead to the strong lower semi-continuity of 
    $(\wM_t)_t$ in families, again exploiting formula (\ref{eqn:Useful Formula}). Therefore letting $(u_k)_k\in \CE^1_{\fibre}(\CX,\om)$ strongly converging to $u_0\in \CE^1(X_0,\om_0)$,
    we get
    $$
    \liminf_{k\to +\infty}\M_k(u_k)=\liminf_{k\to+\infty}\big(\wM_k(\wu_k)+\M_k(\psi_k)\big)\geq \wM_0(\wu_0)+\M_0(\psi_0)=\M_0(u_0),
    $$
    which concludes the proof.
\end{proof}

Now, we establish a uniform coercivity with an almost optimal slope of $(\M_t)_t$:

\begin{thm}
\label{thm:Openness Coercivity}
The coercivity threshold
\[
    \sm_t := \sup \set{A \in \BR}{\M_t \geq A(-\E_t) - B \text{ on $\CE^1_{\nmlz}(X_t,\om_t)$, for some $B \in \BR$}}
\]
is lower semi-continuous at $t=0$. 
Moreover, for any $A < \sm_0$, there exists $B>0$ and $r>0$ such that for each $t \in \BD_r$, for any $u \in \CE^1_{\nmlz}(X_t, \om_t)$, $\M_t(u) \geq A(-\E_t(u)) - B$. 
\end{thm}

\begin{proof}
Assume by contradiction that $t\to \sigma_t$ is not lower semi-continuous at $t=0$. 
Without loss of generality, we assume $\sigma_0>-\infty$, i.e. there exist $A_0\in \BR, B_0\in \BR$ such that $\M_0(u)\geq -A_0\E_0(u)-B_0$ for any $u\in \CE^1_\nmlz(X_0,\om_0)$. Suppose by contradiction that there exist $A<A_0$, $t_k\to 0$, and $u_k\in \CE^1_\nmlz(X_k,\om_k)$ such that
\begin{equation}\label{eqn:Number 1 Bis}
    \M_k(u_k)< A(-\E_k(u_k))-B_k
\end{equation}
for any $k\in \BN$, where $B_k\to +\infty$. Thanks to Lemma \ref{lem:BDL17}, one can also suppose that $u_k\in L^{\infty}(X_k)$.
By Proposition \ref{prop:Change Reference Final case} and (\ref{eqn:Useful Formula}), for any $k$, for any $w \in \PSH_{\nmlz}(X_k,\om_k) \cap L^\infty(X_k)$, we have
\begin{align*}
    \M_k(w) &= \wM_k(\ww) + \M_k(\psi_k)\\
    &= \H_{\mu_k}(w) + \Bar{s}\E_k(w)- \Bar{s}\E_k(\psi_k)-n\E_{k,\Theta_k}(w) + n\E_{k,\Theta_k}(\psi_k) +\M_k(\psi_k).
\end{align*}
Moreover, letting $C_1>0$ such that $\Theta_k\geq - C_1 \om_k$,  we have $nV_k \E_{k,\Theta_k}(w)\leq -C_1V_k(n+1)\E_k(w)$ as seen during the proof of Proposition \ref{prop:Lower Semicontinuity Mabuchi families}.
Hence, one can derive 
\begin{equation}\label{eq:mabuchi_lower_bound Bis}
    \M_k(w) \geq \H_{\mu_k}(w) + (\bar{s} + C_1) \E_k(w) - C_2
\end{equation}
where $C_2$ is a uniform lower bound for $-\bar{s}\E_k(\psi_k) + n \E_{k,\Ta_k}(\psi_k) + \M_k(\psi_k)$ for all $k$. In particular, choosing $C_1>0$ big enough, from (\ref{eqn:Number 1 Bis}) and (\ref{eq:mabuchi_lower_bound Bis}) we deduce that $-\E_k(u_k)\to +\infty$. Let $D>0$ be a finite number to be fixed later. 
Consider $g_k(s)$ a unit-speed geodesic connecting $0$ and $u_k$ in $\CE^1_\nmlz(X_k, \om_k)$. For any $k\gg 1$ large enough, set $v_k = g_k(D)$. 
Since $\M_k$ is convex along $g_k(s)$ by Proposition \ref{prop:Geod_Conv}, 
\begin{equation}\label{eqn:Number 3 Bis}
    \M_k(v_k) \leq \frac{D}{d_k}\M_k(u_k)+\frac{d_k-D}{d_k}\M_k(0)
    \leq D A.
\end{equation}
    On the other hand, by \eqref{eq:mabuchi_lower_bound Bis} we have $\H_{\mu_k}(v_k)$ is uniformly bounded, and Corollary \ref{cor:strong_compact_adapted_fami} gives that $(v_k)_k$ strongly subconverges to a function $v_0\in \CE^1(X_0,\om_0)$. Without loss of generality, we assume that $v_k$ strongly converges to $v_0$. 
    As $\E_0(v_0)=\lim_{k\to +\infty}\E_k(v_k)=-D$, by Proposition \ref{prop:Lower Semicontinuity Mabuchi families} and (\ref{eqn:Number 3 Bis}) we gain {\small
    \begin{align*}
        AD
        &\geq \liminf_{k \to +\infty} \M_k(v_k)
        \geq \M_0(v_0)\\
        &= \M_0(v_0 - \sup_{X_0} v_0) 
        \geq A_0(-\E_0(v_0-\sup_{X_0} v_0))-B_0 
        = A_0D - B_0+ A_0\sup_{X_0} v_0.
    \end{align*} }
    However as consequence of Theorem \ref{thm:SL_and_Skoda_in_family} we have $0\geq \sup_{X_0}v_0\geq -C_{SL}$ (see \cite[Prop.~2.8, Lem.~2.11]{Pan_Trusiani_2023}). Hence we get a contradiction taking $D = \frac{B_0 + \lvert A_0\rvert C_{SL} + 1}{A_0 - A}$.
\end{proof}

\subsection{Proof of Theorem~\ref{bigthm:openness_classes}}\label{ssec:Proposition B}

Let $X$ be a normal compact Kähler variety with klt singularities.

Any smooth Kähler form defines a non-zero element in $H^0(X,\CC_X^\infty/\PH_X)$ where $\CC_X^\infty$ is the sheaf of continuous functions on $X$ given as restriction of smooth functions under local embedding, while $\PH_X$ represents its subsheaf of restriction of pluriharmonic functions. 
As in literature (cf. \cite[Sec.~5.2]{EGZ_2009}, \cite[Sec.~3.3]{Guedj_Guenancia_Zeriahi_2023}), the Kähler cone $\CK_X$  consists of classes $\alpha\in H^1(X,\PH_X)$ such that $[\om]=\alpha$ where
$$
H^0(X,\CC_X^\infty/\PH_X)\overset{[\cdot]}{\longrightarrow} H^1(X,\PH_X)
$$
is induced by the short exact sequence
$
0\to \PH_X\to  \CC_X^\infty \to \CC_X^\infty/\PH_X\to 0.
$
Since $X$ is normal, $\PH_X$ coincides with the sheaf of real parts of holomorphic functions and $H^1(X,\PH_X)$ is a finite dimensional vector space \cite[Lem.~4.6.1]{Boucksom_Guedj_2013}.

\begin{rmk}
To prove the lower semi-continuity of coercivity threshold with respect to K\"ahler classes, it is natural to consider a trivial family $X \times \BD^m \to \BD^m$ with a background metric $\om_0 + \sum_{i=1}^m |t_i|^2 \eta_i$ where $m = \dim H^1(X, \PH_X)$, $(\eta_i)_i$ are smooth closed $(1,1)$-forms such that $([\eta_i])_i$ forms a basis in $H^1(X,\PH_X)$. 
The argument for trivial families goes the same as we did before, but since we do not treat the situation with higher dimensional base, we give a comprehensive proof to Theorem~\ref{bigthm:openness_classes} in this section for the reader's convenience. 
\end{rmk}

Let $U\Subset \CK_X$ be a relatively compact open set in the Kähler cone and let $\{\omega_t\}_{t\in U}$ be a smooth family of Kähler forms such that $\omega_t$ is a representative of the cohomology class associated to $t$. To lighten notation we set $V_t:=\int_X\omega_t^n$ for any $t\in U$, and similarly for other quantities/functionals.
\begin{lem}
\label{lem:Openness in Kähler cone}
    Let $t_1,t_2\in U$, let $\nu$ be a probability measure on $X$ and let $u_i\in \CE_\nmlz^1(X,\omega_{t_i})$ be solution of
    $$
    \big(\omega_{t_i}+\ddc u_{i}\big)^n=V_{t_i}\, d\nu
    $$
    for $i=1,2$. Let also $\delta>1$ such that $\frac{1}{\delta}\omega_{t_2}\leq\omega_{t_1}\leq \delta\omega_{t_2}$. Then
    \begin{equation}\label{eqn:Easy}
        \E_{t_1}(u_1)\geq \frac{V_{t_2}\delta^{n+1}}{V_{t_1}}\E_{t_2}\Big(\frac{u_1}{\delta}\Big)\geq \delta^{2n+2}\E_{t_1}(u_1)
    \end{equation}
    and there exist non-negative constants $C_1,C_2$ only depending on $n,X, \delta, \omega_2, \E_2(u_2)$,\\ $\E_1(u_1), V_{t_2}/V_{t_1}$ such that
    $$
    \I_{t_2}\Big(\frac{u_1}{\delta},u_2\Big)\leq C_1 \Big(1-\frac{V_{t_1}}{V_{t_2}}\Big)+C_2(\delta-1).
    $$
\end{lem}
\begin{proof}
    For simplicity set $\omega_i:=\omega_{t_i}$ and similarly for other quantities. As $u_i\leq 0$ and $\omega_1\leq \delta\omega_2, \omega_2\leq \delta \omega_1$, the inequalities (\ref{eqn:Easy}) follow immediately from the definitions of the Monge--Amp\`ere energy and from its monotonicity property.

    Next, we have
    {
    \begin{align*}
        &\Big(\omega_2+\ddc \frac{u_1}{\delta}\Big)^n 
        = \Big(\frac{\omega_1+\ddc u_1}{\delta}+\frac{\delta\omega_2-\omega_1}{\delta}\Big)^n\\
        &= \frac{1}{\delta^n}\sum_{j=0}^n\binom{n}{j}(\omega_1+\ddc u_1)^j\wedge (\delta \omega_2-\omega_1)^{n-j}\\
        &= \frac{V_1 \, (\omega_2+\ddc u_2)^n}{V_2\delta^n}+\frac{1}{\delta^n}\sum_{j=0}^{n-1}\binom{n}{j}(\omega_1+\ddc u_1)^j\wedge (\delta\omega_2-\omega_1)^{n-j}.
    \end{align*}
    }%
    Set $S_j:=\binom{n}{j}(\omega_1+\ddc u_1)^j\wedge (\delta\omega_2-\omega_1)^{n-j}$ for $j=0,\dots, n-1$. By definition of $\I_2$, it follows that
    {\small
    $$
    \I_2\Big(\frac{u_1}{\delta},u_2\Big)=\Big(1-\frac{V_1}{V_2\delta^n}\Big)\int_X \Big(\frac{u_1}{\delta}-u_2\Big)(\omega_2+\ddc u_2)^n+\frac{1}{\delta^n}\sum_{j=0}^{n-1} \int_X \Big(u_2-\frac{u_1}{\delta}\Big) S_j.
    $$
    }%
    By \cite[Lem.~2.7]{BBGZ_2013} it follows that $\int_X \big(\frac{u_1}{\delta}-u_2\big)(\om_2+\ddc u_2)^n\leq C_1$ for $C_1>0$ only depending on $\E_2(u_1/\delta), \E_2(u_2), X, n, \om_2, V_2/V_1$. The inequalities (\ref{eqn:Easy}) implies that $C_1$ only depends on $\E_1(u_1), \E_2(u_2), X, n, \om_2, \delta$. Similarly, thanks to the inequalities $0\leq \delta\omega_2-\omega_1\leq \frac{\delta^2-1}{\delta}\omega_2$, $0\leq \omega_1+\ddc u_1\leq \delta(\omega_2+\ddc \frac{u_1}{\delta})$, we obtain
    {\small
    \begin{align*}
        \frac{1}{\delta^n}\sum_{j=0}^{n-1}\int_X \Big(u_2-\frac{u_1}{\delta}\Big) S_j&\leq \frac{1}{\delta^n}\sum_{j=0}^{n-j} \binom{n}{j}\delta^j \frac{(\delta^2-1)^{n-j}}{\delta^{n-1}}\int_X \Big\lvert u_2-\frac{u_1}{\delta}\Big\rvert \big(\omega_2+\ddc \frac{u_1}{\delta}\big)^j\wedge \omega_2^{n-j}\\
        &\leq C_3(\delta-1)\sum_{j=0}^{n-1}\int_X \Big\lvert u_2-\frac{u_1}{\delta}\Big\rvert \big(\omega_2+\ddc \frac{u_1}{\delta}\big)^j\wedge \omega_2^{n-j}
    \end{align*}
    }%
    for a uniform constant $C_3$ depending only on $\delta, n$. Similarly to before, \cite[Lem.~2.7]{BBGZ_2013} gives that  $\int_X \Big\lvert u_2-\frac{u_1}{\delta}\Big\rvert \big(\omega_2+\ddc \frac{u_1}{\delta}\big)^j\wedge \omega_2^{n-j}$ is bounded from above by a constant only depending on $\E_1(u_1), \E_2(u_2), n, \omega_2, X, \delta, V_1/V_2$.
\end{proof}
\begin{rmk}\label{rmk:Openness in Kähler cone}
    Assume that we have the same setting of Lemma \ref{lem:Openness in Kähler cone}, but suppose further that $u_i\in L^\infty(X)$. If $T$ is a closed and postive $(1,1)$-current, then it follows immediately from what said above that
    $$
    \E_{t_1,T}(u_1)\geq \frac{V_{t_2}\delta^n}{V_{t_1}}\E_{t_2,T}\Big(\frac{u_1}{\delta}\Big)\geq \delta^{2n}\E_{t_1,T}(u_1)
    $$
    where we denote by $\E_{t,T}(v):=\frac{1}{n V_t} \sum_{j=0}^{n-1} \int_X v\, T\wedge\omega_{t,v}^j\wedge \omega_t^{n-j-1}$ the twisted energy.
\end{rmk}

We now prove Theorem~\ref{bigthm:openness_classes}:
\begin{thm}
    Let $X$ be a normal compact K\"ahler variety with klt singularities, and let $\{\om_t\}_{t\in \CK}$ be a family of Kähler forms. Then the map {\small
    $$
    \CK_X \ni t \mapsto \sigma_t:=\sup\big\{A\in \BR\, \vert \M_{\om_t}\geq -A\E_{\om_t}-B \text{ on } \CE^1_\nmlz(X_t,\om_t), \text{ for some } B\in \BR\big\}
    $$}
    is lower semi-continuous.  Moreover, for any open set $U\subset\CK_X$, any smooth family of Kähler form $\{\om_t\}_{t\in U}$, any $t_0\in U$ and any $A < \sm_{t_0}$, there exists $B>0$ and an open set $t_0\in U' \subset U$ such that $\M_{\om_t}(u) \geq A(-\E_{\om_t}(u)) - B$ for each $t \in U'$ and for any $u \in \CE^1_{\nmlz}(X, \om_t)$,
\end{thm}
\begin{proof}
    Let $\om\in \CK_X$, let $U\Subset \CK_X$ be a relatively compact open set containing $[\om]$, and let $\{\omega_t\}_{t\in U}$ be a smooth family of Kähler form as above such that $\omega_{t_0}=\om$ for $t_0\in U$. 
    Without loss of generality, we also assume that $\sigma_{t_0}>-\infty,$ i.e. that there exists $A_0 \in \BR, B_0 \in \BR$ such that $\M(u):=\M_{\om}(u)\geq -A_0\E(u)-B_0$ for any $u\in \CE^1_\nmlz(X,\om)$. Then, by contradiction suppose that $t\to \sigma_t$ is not lower semi-continuous at $t_0$, i.e. there exist a sequence $\om_k:=\om_{t_k}$ converging to $\om$, a coefficient $A<A_0$ and a sequence $u_k\in \CE^1_\nmlz(X,\om_k)$ such that
    \begin{equation}
        \label{eqn:Openness in the Kähler cone 1}
        \M_k(u_k):=\M_{\om_k}(u_k)< -A\E_k(u_k)-B_k
    \end{equation}
    for any $k\in \BN$, where $B_k\to +\infty$.

    Let $m\geq 1$ be an integer such that $mK_X$ is locally free, let $h_m$ be a smooth metric on $m K_X$ and let $\mu$ be the associated adapted measure. Set $\Theta=\Ric(\mu)\in c_1(X)$ for the Ricci form of the metric $h$.
    For any $k\in \BN$, let also $\psi_k\in \CE^1_\nmlz(X,\om_k)$ be the solution of
    $$
    (\om_k+\ddc \psi_k)^n= V_k \, \mu.
    $$
    As seen in Proposition \ref{prop:Change Reference Final case} we have $\M_k(u)=\wM_k(\wu)+\M_k(\psi_t)$ for any $u\in \CE^1(X,\om_k)$ where 
    $$
    \M_k(u)=\H_k(u)+\Bar{s}_k \E_k(u)-n\E_{k,\Ric(\om_k)}(u)
    $$
    is the Mabuchi functional of $(X,\om_k)$ while
    $$    \wM_k(\wu)=\H_\mu(u)+\bar{s}_k\big(\E_k(u)-\E_k(\psi_k)\big)-n\big(\E_{k,\Theta}(u)-\E_{k,\Theta}(\psi_k)\big).
    $$
    Clearly by $\E_k, \E_{k,\Theta}$, we denote respectively the Monge--Amp\`ere energy and the $\Theta$-twisted energy with respect to $(X,\om_k)$. With obvious notations, we also have $\M(u)=\wM(u)+\M(\psi)$ for any $u\in \CE^1(X,\om)$.

    We claim that $\E_k(\psi_k) \to \E(\psi), \E_{k,\Theta}(\psi_k)\to \E_\Theta(\psi)$ as $k\to +\infty$. As $\om_k\to \om$ and $\{\om_t\}_{t\in U}$ is a smooth family, let $\delta_k\geq 1$, $\delta_k\to 1$ such that $\frac{1}{\delta_k}\om\leq \om_k\leq \delta_k\om$. 
    Then Lemma \ref{lem:Openness in Kähler cone} gives $\I\big(\psi,\frac{\psi_k}{\delta_k}\big)\to 0$ as $k\to +\infty$, i.e. $\psi_k/\delta_k\to \psi$ strongly (see \cite[Prop.~2.3]{BBEGZ_2019}). Moreover, Lemma~\ref{lem:Openness in Kähler cone} also gives
    {
    $$
    |\E_k(\psi_k)-\E(\psi)| 
    \leq \abs{\E_k(\psi_k)-\E\lt(\frac{\psi_k}{\delta_k}\rt)}
    + \abs{\E\lt(\frac{\psi_k}{\delta_k}\rt)-\E(\psi)} 
    \xrightarrow[k \to +\infty]{} 0.
    $$
    }%
    Similarly, writing $\Theta=T_1-T_2$ as difference of closed and positive currents, from Remark \ref{rmk:Openness in Kähler cone} we get
    {\small
    \begin{align*}
        &\Big\lvert \E_{k,\Theta}(\psi_k)-\E_{\Theta}(\psi) \Big\rvert
        \leq \Big\lvert \E_{k,T_1}(\psi_k)-\E_{T_1}(\psi) \Big\rvert +\Big\lvert \E_{k,T_2}(\psi_k)-\E_{T_2}(\psi) \Big\rvert 
    \end{align*}
    } 
    and the RHS tends to $0$
    as $k\to +\infty$. The claim follows.

    Next, proceeding exactly as in Theorem~\ref{thm:Openness Coercivity}, we uniformly bound $\E_{k, \Theta}(w)$ from above in terms of $-\E_k(w)$ if $w\in \CE^1_\nmlz(X,\om_k)\cap L^\infty(X)$. This leads to
    \begin{equation}
        \label{eqn:Openness in the Kähler cone 2}
        \M_k(w)\geq \H_\mu(w)+C_1 \E_k(w)-C_2
    \end{equation}
    for any $w\in \CE^1_\nmlz(X,\om_k)\cap L^\infty(X)$ and for any $k\in \BN$, where $C_1,C_2$ are uniform positive constants. In particular, from (\ref{eqn:Openness in the Kähler cone 1}), (\ref{eqn:Openness in the Kähler cone 2}) we deduce that $-\E_k(u_k)\to +\infty$, as $C_1>0$ can be taken arbitrarily big. Following the proof of Theorem \ref{thm:Openness Coercivity}, the convexity of $\M_k$ along weak geodesics of Proposition \ref{prop:Geod_Conv} yields $ \M_k(v_k)\leq DA $, where $v_k\in \CE^1_\nmlz(X,\om_k)$ is the element on the unit-speed geodesic connecting $0$ and $u_k$ such that $-\E_k(v_k)=D$ for $D>0$ to be chosen later. Note that from (\ref{eqn:Openness in the Kähler cone 2}) we also gain $\H_\mu(v_k)\leq C_3$ uniformly in $k\in \BN$.

    Let now $w_k\in \CE^1_\nmlz(X,\om)$ such that $(\om+\ddc w_k)^n/V=(\om_k+\ddc v_k)^n/V_k$ where $V:=\int_X \om^n, V_k:=\int_X \om_k^n$. Since $\H_\mu(w_k)=\H_\mu(v_k)\leq C_3$, the strong compactness of entropy lower level sets gives that $w_k\to w\in \CE^1_\nmlz(X,\om)$ strongly, up to considering a subsequence. 
    Moreover, letting $\delta_k\to 1$ as above such that $\frac{1}{\delta_k}\om\leq \om_k\leq \delta_k\om $ for any $k\in \BN$, Lemma \ref{lem:Openness in Kähler cone} yields $\I(w_k,\frac{v_k}{\delta_k})\longrightarrow 0$ as $k\to +\infty$. 
    This leads to $\frac{v_k}{\delta_k}\to w$ strongly as a consequence of \cite[Theorem 1.8, Proposition 2.3]{BBEGZ_2019}. 
    Similarly to the calculation made above for $\{\psi_k\}_k$, Lemma \ref{lem:Openness in Kähler cone} and Remark \ref{lem:Openness in Kähler cone} show that
    {
    \begin{equation*}
        |\E_k(v_k)-\E(w)|
        \longrightarrow 0,\quad
        |\E_{k,\Theta}(v_k)-\E_{\Theta}(w)|
        \longrightarrow 0
    \end{equation*}
    }%
    as $k\to +\infty$.
    Furthermore, we also obtain $(\om_k+\ddc v_k)^n/V_k=(\om+\ddc w_k)^n/V\to (\om+\ddc w)^n/V$, which implies $\liminf_{k\to +\infty}\H_\mu(v_k)\geq \H_\mu(w)$. Hence it follows
    {
    \begin{align*}
        A_0D-B_0 
        &= -A_0\E(w)-B_0
        \leq \M(w)=\wM(\ww)+\M(\psi)\\
        &\leq \liminf_{k\to +\infty}\big(\wM_k(\wv_k)+\M(\psi_k))=\liminf_{k\to +\infty}\M_k(v_k)\leq DA,
    \end{align*}
    }%
    from which we get the contradiction if $D=\frac{B_0+1}{A_0-A}$.
\end{proof}

\section{Smoothable case}
\label{sect_smoothing}
Let $(X_0, \om_0)$ be a normal compact K\"ahler variety with klt singularities. 
Suppose that $\pi: (\CX,\om) \to \BD$ is a $\BQ$-Gorenstein smoothing of $(X_0, \om_0)$. 
Namely, $\pi: \CX \to \BD$ satisfies Setting~\ref{sett:klt}, and the general fibre $X_t$, $t\neq 0$, is smooth.
Also, we have $\om$ a hermitian metric so that  $\om_{|X_0} = \om_0$, and $\om_t := \om_{|X_t}$ is K\"ahler.

The following corollary is a direct consequence of Theorem \ref{thm:Openness Coercivity} and Theorem \ref{thm:Coerc Existence}. 

\begin{cor}\label{cor:Existence cscK}
Assume that $\M_0$ is coercive. Then there exists $r>0$ such that $X_t$ admits a unique cscK metric for any $t\in \BD_r^*$.
\end{cor}

Recalling the definition of \emph{singular cscK metrics} for singular varieties (Definition \ref{defn:cscK}), in this section, we are going to prove the following main result.
\begin{thm}
\label{thm:Main Theorem}
    Let $\pi: (\CX,\om) \to \BD$ be a $\BQ$-Gorenstein smoothing of a compact K\"ahler variety $(X_0,\om_0)$ with klt singularities. 
    If $\M_0$ is coercive, then $(X_0,\om_0)$ admits a singular cscK metric which is also a minimizer of $\M_0$. 
    Moreover, the cscK potentials $(\vph_{cscK,t})_{t \in \BD}$ converges to $\vph_{cscK,0}$ strongly and smoothly in the family sense.
\end{thm}

It follows from Corollary \ref{cor:Existence cscK} that for any $t \neq 0$ sufficiently small, there exists $\varphi_t$ satisfying the coupled cscK-equations
\begin{equation}\label{eq_cscK_t}
    (\omega_t+ \ddc_t \varphi_t)^n = e^{F_t} \omega_t^n 
    \quad\text{and}\quad
    \Delta_{\varphi_t} F_t= -\bar s_t + \tr_{\vph_t} \Ric(\omega_t).
\end{equation}
We proceed by establishing uniform $L^\infty$-estimate for $\varphi_t$ and uniform $L^p$-estimate of Laplacian of $\varphi_t$ on smooth fibres. 
Then, we extract a limit of these cscK potentials to a cscK potential on the central fibre.

A key point in our approach is to work with new reference metrics with canonical densities. 
This allows us to control uniformly its Ricci curvature with respect to the reference form itself, which is not the case for the original metric $\om$.  

Fix a smooth $(1,1)$-form $\Ta \in c_1(-K_{\CX/\BD})$.
Let $h$ be a smooth hermitian metric on $m K_{\CX/\BD}$ corresponding to $\Ta$ and let $\mu_t$ be the adapted probability measure induced by $h_t = h_{|X_t}$.
For each $t$, one can find $\psi_t \in \PSH(X_t, \om_t) \cap L^\infty(X_t)$ which solves 
\[
    \frac{1}{V_t}(\om_t + \ddc_t \psi_t)^n = \mu_t
    \quad \text{and} \quad \sup_{X_t} \psi_t = 0.
\]
Denote by $\wom_t = \om_t + \ddc_t \psi_t$ and one has $\Ric(\wom_t) = \Ta_t$. Denote by $f_t$ the density of $\wom_t^n$ with respect to $\omega_t^n$, then we have $\Ric(\omega_t)=\Ric(\wom_t) +\ddc \log f_t$. Therefore, we rewrite the cscK-equations for the new reference metric $\wom_t$:
\begin{equation}\label{eq_cscK_tilde}
    (\wom_t + \ddc_t \phi_t)^n = e^{\wF_t}\wom_t^n, 
    \quad \Delta_{\wom_{\phi_t}}\wF_t = -\bar{s}_t + \tr_{\wom_{\phi_t}}(\Ric(\wom_t)),
    \quad\text{and}\quad \sup_{X_t} \phi_t = 0
\end{equation}
where $\wF_t = F_t - \log f_t$ and $\phi_t= \varphi_t- \psi_t$.

\subsection{$L^\infty$-estimates}
This section aims to establish a uniform $L^\infty$-estimate for cscK-potentials $(\varphi_t)_{t \neq 0}$.
Since $\psi_t$ is uniformly bounded when $t$ is close to $0$, it suffices to bound $\phi_t$. 

\begin{thm}\label{thm_L_inf}
Suppose that $\pi: \CX \to \BD$ is a $\BQ$-Gorenstein smoothing.
Assume that there exist constants $A \in (0,1)$ and $B > 0$ such that $\M_0(u) \geq A(-\E_0(u)) - B$ for all $u \in \CE^1_{\nmlz}(X_0,\om_0)$. Then there exists a constant $C$ such that for all t sufficiently close to $0$, we have
\begin{equation}
     \|\wF_t\|_{L^{\infty}(X_t)}+ \|\phi_t\|_{L^{\infty}(X_t)}\leqslant C.
\end{equation} 
As a consequence, we have
\begin{equation}
    {\rm osc}_{X_t} \varphi_t\leq C,
\end{equation}
for all t sufficiently close to $0$.
\end{thm}

We shall use a version of $L^\infty$-estimate for the cscK-equations (Theorem~\ref{thm_uniform_cscK}). 
A proof of estimates on cscK-equations is initiated by Chen--Cheng \cite{Chen_Cheng_2021_1} relying on an ABP estimate in local coordinates for the reference metric $\omega$ (see also \cite{Deruelle_DiNezza_2022, Guo_Phong_2022} for different methods). 
However, due to the delicate dependence on $\om$, extending Chen--Cheng's approach to degenerate families in our setting seems challenging. 
Instead, we shall follow the method of Guo--Phong \cite[Thm.~3]{Guo_Phong_2022} using auxiliary Monge--Amp\`ere equations (see also \cite{Guo-Phong-Tong2023}) which is more adaptable for degenerate setting.

\begin{thm}\label{thm_uniform_cscK}
Let $(X, \omega)$ be a compact K\"ahler manifold of complex dimension $n$ with $V=\int_X \omega^n$.  Suppose that  $(\varphi, F)$ is the solution to the coupled equations
\begin{equation}\label{eq_cscK}
    (\omega + \ddc \varphi)^n = e^{F}\omega^n, 
    \quad 
    \Delta_{\varphi} F = -\bar{s} + \tr_\varphi(\Ric(\omega)),
    \quad\text{and}\quad
    \sup_X \varphi = 0.
\end{equation}
We assume that there are positive constants $K_1, K_2, K_3$  such that 
 \begin{equation}\label{cond_ric}
     -K_1 \omega \leqslant  {\rm Ric} (\omega)\leqslant K_2 \omega,
 \end{equation}
    \begin{equation}\label{cond_ent}
        {\bf H}(\varphi)= \frac{1}{V}\int_X \log\lt(\frac{\omega_\varphi^n}{\omega^n}\rt) \omega_\varphi^n \leqslant K_3,   \end{equation}
and  there exists $\alpha>0$ and $K_4>0$ such that for all $\phi\in {\rm PSH}(X, \omega)$ \begin{equation}\label{cond_skoda}
       \int_{X}e^{-\alpha(\phi-\sup\phi)} \omega^n \leqslant K_4.
       \end{equation}
Then there is a constant $C>0$ depending only on $n,\bar s, V,\alpha, K_1, \cdots, K_4$ such that 
$$ 
    \|\varphi\|_{L^\infty} + \|F\|_{L^\infty} \leq C. 
$$
\end{thm}
We remark that the first condition in \cite[Thm.~3]{Guo_Phong_2022}, namely $\omega\leq C\omega_X$ for a fixed K\"ahler metric $\omega_X $ on $X$, is suitable with a degenerate family in the K\"ahler cone of a fixed K\"ahler manifold. 
In our situation, this condition is replaced by a Skoda--Zeriahi type integrability condition \eqref{cond_skoda}, which fits well for degenerate families.
We also keep track of the depending constants appearing in the proof.
We also remark that Chen--Cheng's estimates are exploited by Zheng \cite{Zheng_2018, Zheng_2022} for cscK cone metrics, but it differs from our setting. 

\begin{proof}[Proof of Theorem~\ref{thm_uniform_cscK}] 
The proof follows the approach in \cite[Thm 3]{Guo_Phong_2022} with certain simplifications. 
Let $\tau_k:\mathbb{R}\rightarrow \mathbb{R}_+$ be a sequence of positive smooth functions which decreases to the function $x\mapsto x\cdot {\bf 1}_{\mathbb{R}_+}(x).$  We solve the following auxiliary complex Monge-Amp\`ere equation {\small
\begin{equation}
V^{-1}(\omega+\ddc \phi_{k})^n=\frac{\tau_k(-\varphi + \lambda F)+1}{A_k} e^F\omega^n, \quad \sup_X\phi_k=0,
\end{equation} }
where {\small $$A_k=\int_X ( \tau_k (-\varphi +\lambda F)+1 )e^F \omega^n\rightarrow  \int_{\{-\varphi+\lambda F>0\}} (-\varphi +\lambda F)e^F \omega^n + V=A_\infty ,$$}
as $k\rightarrow \infty$. Young inequality  with  $\chi(s)=(s+1)\log (s+1)-s$ implies  {\small $$\int_X (-\varphi) e^F\omega^n \leq \int_X \chi(\alpha^{-1}e^F)\omega^n+   \int_X \chi^*(-\alpha\varphi)\omega^n$$ }
where $\alpha$ is the constant in \eqref{cond_skoda}. 
It  then follows from Remark \ref{rmk:L_chi_entropy},  \eqref{cond_skoda} and \eqref{cond_ent} that  $V \leq  A_\infty \leq C(K_3, K_4, V)$. Thus $ V \leq A_k\leq  C_1=C(K_3, K_4, V)$ for $k$ sufficiently large. Consider the function $$\Phi=-\epsilon(-\phi_k+\Lambda)^{\frac{n}{n+1}}-\varphi+\lambda F $$ with $\epsilon=\left(\frac{(n+1)(n+\lambda \bar s )}{n^2}\right)^{\frac{n}{n+1}}A_k^{\frac{1}{n+1}} $ and $\Lambda=\left( \frac{2n}{n+1} \epsilon\right)^{n+1} $.
Let $x_0$ be a maximal point of $\Phi$. 
At $x_0$ we have {\small
\begin{align*}
    0&\geq \Delta_{\omega_\varphi}\Phi \geq \frac{\epsilon n}{n+1}(-\phi_k+ \Lambda) ^{ -\frac{1}{n+1}}  \Delta_{\omega_\varphi} \phi_k -\Delta_{\omega_\varphi} \varphi  + \lambda \Delta_{\omega_\varphi} F\\
    &= \frac{\epsilon n}{n+1}(-\phi_k+ \Lambda) ^{ -\frac{1}{n+1}}  (\tr_{\omega_\varphi} \omega_{\phi_k}  -\tr_{\omega_\varphi}\omega) -\tr_{\omega_\varphi} (\omega_\varphi-\omega)  + \lambda (-\bar s+ \tr_{\omega_\varphi} {\rm Ric}(\omega)) \\
    &\geq \frac{n^2\epsilon}{n+1} (-\phi_k+\Lambda)^{-\frac{-1}{n+1}} \left( \frac{\tau_k(-\varphi+\lambda F)+1}{A}\right)^{1/n} -n- \lambda \bar s+  \left(1-\frac{n\epsilon}{n+1}\Lambda^{-\frac{1}{n+1}}  -\lambda K_1 \ \right)\tr_{\omega_\varphi}\omega\\
    &\geq \frac{n^2\epsilon}{n+1} (-\phi_k+\Lambda)^{-\frac{-1}{n+1}} \left( \frac{\tau_k(-\varphi+\lambda F)+1}{A}\right)^{1/n} -n- \lambda \bar s,
\end{align*} }
where we choose $\lambda = \frac{1}{(2+ nV)K_1}$ so that $n+\lambda \bar s>0$ and $\lambda K_1<1/2$. Therefore at $x_0$ we get {\small
\begin{equation*}
   -\varphi+\lambda F\leq  \left(\frac{(n+\lambda \bar s)(n+1)}{n^2\epsilon}\right)^n A_k(-\phi_k+\Lambda)^{n/n+1};
\end{equation*} } hence, $\Phi(x_0) \leq 0$ and $\Phi \leq 0$ on $X$. 
By the choice of $\epsilon, \Lambda$
and $V\leq  A_k\leq C(K_3, K_4, V)$, and Young inequality, we derive that for any $\delta>0$
\begin{equation}\label{ineq_F}
\lambda F\leq -\varphi+\lambda F\leq C(V, K_1,K_3,K_4)(-\phi_k+\Lambda)^{n/(n+1)}\leq -\delta \phi_k+ C_2 , 
\end{equation}
with $C_2=C(\delta,V,K_2, K_3, K_4).$
Therefore, combing with \eqref{cond_skoda}, the inequality \eqref{ineq_F} implies that for any $\beta >\lambda^{-1}$ and $\delta>0$ such that $\delta \beta<\alpha$, 
\begin{equation} \label{ineq_exp_F}
\int_X e^{\beta \lambda F}\omega^n \leq e^{C_2} \int_{X}e^{-\alpha \phi_k}\omega^n \leq C(\delta, K_1, K_3, K_4).
\end{equation}
In particular, this yields $\|e^F\|_{L^p(X, \omega^n)}\leq  C(\alpha, K_1, K_3, K_4)$ for some $p>1$. 
From a refined version of Ko{\l}odziej's $L^\infty$-estimate \cite{Kolodziej_1998} (see \cite[Thm.~A]{DGG2023} for the version we referred), we obtain 
$\|\varphi\|_{L^\infty}\leq C(n,V,\alpha, K_1,K_3,K_4).$ 
Combining this with \eqref{ineq_exp_F}, we infer that $$ \|(\tau_k(-\varphi +\lambda F)+1 )e^F \|_{L^{p'}(X, \omega^n)}\leq  C(n,\alpha, V, K_1, K_3, K_4)$$  for some $p'>1$ and for all $k>0$ sufficiently large. 
Again, Ko{\l}odziej's $L^\infty$-estimate gives a uniform bound $\|\phi_k\|_{L^\infty}\leq C(n,V,\alpha, K_1,K_3,K_4)$. Therefore, the inequality \eqref{ineq_F} shows a uniform upper bound for $F$.

A uniform lower bound for $F$ follows from the maximum principle with the test function $H=F+(K_2+1)\varphi.$ 
Indeed, we have
 \begin{align*}
 \Delta_{\omega_\varphi}H&= -\bar s + \tr_{\omega_\varphi} (\Ric(\omega)) +(K_2+1)n-(K_2+1)\tr_{\omega_\varphi} \omega\\
 &\leq (K_2+1)n-\bar s -\tr_{\omega_\varphi}\omega\leq (K_2+1)n-\bar s+  ne^{-F/n}.  
 \end{align*}
Therefore, at a minimum point $x_0$ of $H$,  $F(x_0)\geq -C(n,K_2)$; thus, we get $F\geq -C(n,K_2)+  (K_2+1)\|\varphi\|_{L^\infty}$. 
\end{proof}

\begin{proof}[Proof of Theorem \ref{thm_L_inf}]
It suffices to verify all conditions in Theorem \ref{thm_uniform_cscK} for the equation \eqref{eq_cscK_tilde}. 
We first remark that up to shrinking $\BD$, there is a constant uniform $c > 0$ such that
\begin{equation}\label{eq:strict_positive}
    \wom_t \geq c\om_t 
\end{equation}
for any $t \in \BD$. 
Indeed,
by \cite[Lem.~3.5]{GGZ_2023_strict}, up to shrinking $\BD$, one can find uniform $p>1$ and $C_1>0$ such that 
\begin{equation}\label{eq:Lp_estimate_change_base}
 \|f_t\|_{L^p(X_t, \om_t^n)} \leq C_1,   
\end{equation}
for any $t \in \BD$, where $f_t = \mu_t / \om_t^n$.
Then a refined version of Ko{\l}odziej's theorem \cite[Thm.~A]{DGG2023}, guarantees the existence of a uniform constant $M_1>0$ such that 
\begin{equation}\label{eq:uniform_L_infty_change_base}
    \|\psi_t\|_{L^\infty(X_t)} \leq M_1.
\end{equation} 

Obviously, for each $t \in \BD$, there is a uniform constant $C_2 > 0$ such that $\Ric(\wom_t) = \Ta_t \geq -C_2\om_t$. 
On the other hand, as $\om$ extends smoothly under local embedding $\CX \xhookrightarrow[\loc.]{} \BC^N$ and the bisectional curvature decreases when passing to holomorphic submanifolds, one can find a uniform constant $C_3 > 0$ such that $\Bisec(\om_t) < C_3$ on $X^{\reg}_t$. Now combining this with \eqref{eq:uniform_L_infty_change_base} and applying Chern--Lu inequality as  
Step 1 in the proof of Theorem \ref{prop:uniform_Lap_est}, we get \eqref{eq:strict_positive}.
One can find a smooth family of smooth maps $F_t: X_0^{\reg} \to X_t$ inducing a diffeomorphism onto their image and such that $F_0 = \Id_{X_0^\reg}$.
Following the same argument as in~\cite[p. 13]{GGZ_2023_strict}, one can check that $F_t^\ast \wom_t$ converges locally smoothly on $X_0^\reg$ when $t \to 0$.
Then \eqref{eq:strict_positive} implies $\wom_0 \geq c \om_0$.

Now, by \eqref{eq:strict_positive}, the condition \eqref{cond_ric} holds by our choice of reference metric $\wom_t$. 
The condition \eqref{cond_skoda} with respect to $\wom_t^n$ follows from Theorem~\ref{thm:SL_and_Skoda_in_family}, \eqref{eq:Lp_estimate_change_base}, and \eqref{eq:uniform_L_infty_change_base}. 

We now show that $\wH_{t}(\phi_t):=\H_{\mu_t}(\phi_t)$ is uniformly bounded from above. Since $\phi_t$ is a minimizer of the Mabuchi functional $\M_t$, we have $\M_{t}(\phi_t)\leq \M_t(0)=0$. Remark that we still have that for some $A,B>0$, $\M_0(u)\geq A(-\E_0(u))-B$ for all $u\in\CE^1_{\nmlz}(X_0,\om_0)$. By the uniform coercivity of Thereom \ref{thm:Openness Coercivity}, we get that $-\E_t(\phi_t)\leq D$ uniformly in $t$. As
$$
-K_1\wom_t\leq \Ric(\wom_t)\leq K_2\wom_t,
$$
this also implies $\lvert\E_{t,\Theta_t}(u)\rvert\leq C(-\E_t(u))$ for a uniform constant $C>0$. Moreover, as $\psi_t$ strongly converges to $\psi_0$ (Corollary \ref{cor:Strong continuity of psi}), $\E_t(\psi_t)$ and $\E_{t,\Theta_t}(\psi_t)$ are uniformly bounded. Thus by (\ref{eqn:Useful Formula}) and $\wM_t(\phi_t)\leq \wM(\psi_t)=0$ we gain
$$
\wH_t(\phi_t)\leq \bar{s}\big(\E_t(\psi_t)-\E_t(\phi_t)\big)+n\big(\E_{t,\Theta_t}(\psi_t)-\E_{t,\Theta_t}(\phi_t)\big)\leq D'
$$
for an uniform $D'>0$; hence we get the uniform bound for the entropy $\wH(\phi_t)$ as required. 
All in all, we obtain the $L^\infty$-estimate by Theorem \ref{thm_uniform_cscK}. 
\end{proof}

\subsection{Laplacian  and higher order estimates}
In this section, we prove higher-order estimates for the solutions of the cscK-equations away from the singular set. 

\begin{thm}\label{thm_C_2}
Up to shrinking $\BD$, for any $K$ a compact subset of $\mathcal{X} \setminus \CZ$ and $l\geq 1$, there is a uniform constant $C(l,K) > 0$ such that for any $t \neq 0$,
\[
    \|\varphi_t\|_{\CC^l(K \cap X_t)} \leq C(l,K).
\]
\end{thm}
 
We recall the cscK-equations for the reference metric $\wom_t = \omega_t + \ddc \psi_t$: 
\begin{align}\label{eq_cscK_C_2}
\left\{\begin{matrix}
(\wom_t + \ddc \phi_t)^n = e^{\wF_t}\wom_t^n = e^{\wF_t-\log f_t}\omega_t^n, \quad \sup_X \phi_t =0 \\ 
\Delta_{\wom_{\phi_t}}\wF_t = -\bar{s}_t + \tr_{\wom_{\phi_t}}(\Ric(\wom_t)), 
\end{matrix}\right.
\end{align}
where $f_t$ is the density of $\wom_t^n$ with respect to $\omega_t^n$, $\wF_t = F_t -\log f_t$ and $\phi_t= \varphi_t- \psi_t$. 
Recall that there exists some constant $C$ so that $-\ddc\log f_t \geq -C\omega_t$ (cf. Section \ref{sect_mild_sing}). 
By adding some constant, one can assume that $\sup_{X_t} \varphi_t = 0$.  Then it follows from Theorem \ref{thm_L_inf} that $\|\wF_t\|_{L^\infty}+ \|\varphi_t\|_{L^\infty}\leq C_0$ for some $C_0>0$ for all $t$ sufficiently close to $0$.  
Denote by $\homg_t:= \omega_t +\ddc \varphi_t = \wom_t + \ddc\phi_t$. 
\begin{prop}\label{C_2_int}
For any $p\geq 1$, we have
\[  
    \|\tr_{\homg_t }\omega_t \|_{L^{2p+2}(X_t, \om_t^n)}\leq C(n,p, A, B, C_0)
\] 
for all $t$ close to 0, where $A, A', B, C_0$ are constants satisfying
$$-A\omega_t \leq \Ric(\wom_t)\leq A\omega_t, \quad \Bisec(\omega_t)\leq B,$$ 
and
$$\|\wF_t\|_{L^\infty}+ \|\varphi_t\|_{L^\infty}\leq C_0.$$
\end{prop}

\begin{proof} 
We shall adapt the strategy of Chen--Cheng to deal with our case. 
We emphasize that we shall use Chern--Lu inequality for Laplacian instead of Aubin--Yau's one since the holomorphic bisectional curvatures of reference metrics are not uniformly bounded from below along the family. 

Consider $u= e^{-a \wF_t -b \varphi_t}\tr_{\homg_t }\omega_t$, for some constant $a, b>1$ independent of $t$ will be determined hereafter. 
For simplicity, we remove the subscript $t$ in the sequel. 
Then 
\[
    \Delta_{\homg} u 
    \geq u \Delta_{\homg} \log u
    = u \left\{ -\Delta_{\homg}(a\wF +b \varphi ) +\Delta_{\homg} \log \tr_{\homg} \om \right\}.
\]
As $\Ric(\wom) \leq A\omega$, we have
\begin{equation}\label{est_1}
\begin{split}
    \Delta_{\homg}(a\wF +b \varphi) 
    &= a(\tr_{\homg} \Ric(\wom)  -\bar s) + b(n-\tr_{\homg} {\omega})\\
    &\leqslant -a \bar s+ bn+( a A -b ) \tr_{\homg}\omega.
\end{split}
\end{equation}
Combining Chern--Lu inequality (cf. Proposition~\ref{prop:Chern-Lu}) and \eqref{est_1}, one can infer
\begin{align*}
    \Delta_{\homg} u 
    & \geqslant u \left\{\frac{\hg^{i\bar \ell} \hg^{k\bar j} \hR_{i\bar j} g_{k\bar \ell}}{\tr_{\homg}\omega} + (a\bar s-bn)+(b-Aa- 2B)\tr_{\homg}\omega
    \right\}\\
    &\geqslant 
    e^{-(a \wF+b\varphi)}  \left\{
    \iprod{\Ric(\homg)}{\om}_\homg + (a\bar s-bn)\tr_{\homg}\omega+ \frac{b}{2}(\tr_{\homg}\omega)^2
 \right\},
\end{align*}
where we choose $a,b>0$ such that $(b - A a- 2B)\geq b/2 > 1$. Set $G := -a \wF-b\varphi$. 
Since 
\begin{equation*}
    \frac{1}{2p+1} \Delta_{\homg}u^{2p+1}= u^{2p}\Delta_{\homg} u + 2p u^{2p-1}|\widehat{\nabla} u|_{\homg}^2,
\end{equation*} 
we have
\begin{equation}\label{int_zero} 
\begin{split}
    0=\int_X \frac{1}{2p+1} \Delta_{\homg}u^{2p+1} \homg^n &\geqslant 2p\int_X u^{2p-1} |\widehat{\nabla} u|_{\homg}^2 \homg^n+ \frac{b}{2}\int_X u^{2p} (\tr_{\homg}\omega)^2 e^{G} \homg^n\\ 
    &\quad + (a\bar s -bn )\int_X u^{2p+1} \homg^n + \int_X u^{2p} \iprod{\Ric(\homg)}{\om}_\homg e^G \homg^n.
\end{split}
\end{equation}
For the last term, we use the fact that $\Ric(\homg) = \Ric(\wom) - \ddc 
\wF$ with $-\Ric(\wom) \leq A \omega$ to deduce
\begin{align} \label{eq_est_ric}
    &-\int_X u^{2p} \iprod{\Ric(\homg)}{\om}_\homg e^G \homg^n 
    \leqslant A\int_X u^{2p} e^{G}(\tr_{\homg } \omega)^2\homg^2 + \underbrace{\int_X u^{2p}e^{G} \langle \ddc \wF, \om \rangle_\homg \homg^n}_{=: (\RN{1})}.
\end{align}
Note that for any two $(1,1)$-forms $\alpha, \beta$ and a K\"ahler metric $\omega$, we have
\begin{equation}\label{eq_tr_tr}
n(n-1)\alpha\wedge \beta\wedge \omega^{n-2}=[(\tr_\omega \alpha)(\tr_{\omega}\beta) -\langle \alpha,\beta \rangle_\omega]\omega^n,
\end{equation}
(see e.g. \cite[Lem.~4.7]{Szekelyhidi_book}) and in particular if $\alpha\geq 0$ and  $\beta\geq 0$, then
\begin{equation}\label{eq_ineq_tr_tr}
 n(n-1)\alpha\wedge \beta\wedge \omega^{n-2}\leq (\tr_\omega \alpha)(\tr_{\omega}\beta) \om^n.
\end{equation}
Applying \eqref{eq_tr_tr}, we get 
\begin{equation}\label{eq:int_lap_term1_est1}
{\small
\begin{split}
    (\RN{1}) 
    &= \int_X u^{2p} e^{G} \langle \ddc \wF, \omega \rangle_\homg \homg^n 
    = \int_X u^{2p}e^{G}  \Delta_{\homg} \wF (\tr_\homg \omega) \homg^n- n(n-1)\int_X  u^{2p}e^{G} \ddc \wF \wedge \omega\wedge \homg^{n-2}  \\
    &= \int_X u^{2p}e^{G} (-\bar s+ \tr_\homg\Ric(\wom))  (\tr_\homg \omega) \homg^n - n(n-1)\int_X  u^{2p}e^{G} \ddc \wF\wedge \omega\wedge \homg^{n-2} \\
    &\leq  -\bar s \int_X u^{2p}e^{G} (\tr_\homg\omega) \homg^n 
    + A \int_X u^{2p}e^{G}(\tr_\homg \omega)^2 \homg^n
    \underbrace{-n(n-1)\int_X  u^{2p}e^{G} \ddc \wF\wedge \omega\wedge \homg^{n-2}}_{=: (\RN{2})}
\end{split}
}%
\end{equation}
Considering now the term $(\RN{2})$, we have
\begin{equation}\label{eq:int_lap_term1_est2}
{\small
\begin{split}
    (\RN{2})
    &= -n(n-1) \int_X  u^{2p}e^{G}  \ddc \wF \wedge \omega\wedge \homg^{n-2} 
    = \frac{n(n-1)}{a}\int_X u^{2p}e^{G}   \ddc(G+b \vph) \wedge \omega \wedge \homg^{n-2}\\
    &= \frac{n(n-1)}{a}\left( \int_X  u^{2p}e^{G} \ddc G \wedge \omega \wedge \homg^{n-2}+ b \int_X  u^{2p}e^{G}   (\homg-\omega)\wedge \omega \wedge \homg^{n-2}  \right)\\
    &\leq \frac{n(n-1)}{a} \underbrace{\int_X  u^{2p}e^{G} \ddc G \wedge \omega \wedge \homg^{n-2}}_{=: (\RN{3})} + \frac{n(n-1)b}{a} \int_X  u^{2p}e^{G}  (\tr_{\homg}\omega)\homg^n.
\end{split}
}%
\end{equation}
Using Stokes' theorem, 
\begin{equation}\label{eq:int_lap_term1_est3}
\begin{split}
    (\RN{3}) 
    &= \int_X u^{2p} e^{G} \ddc G \wedge \omega \wedge \homg^{n-2} \\
    &= - \int_X u^{2p}e^G \dd G\wedge \dc G\wedge \omega\wedge \homg^{n-2} -2p \int_X e^G u^{2p-1} \dd u\wedge \dc G \wedge\omega\wedge \homg^{n-2}\\
    &\leq -\frac{1}{2} \int_X u^{2p}e^G \dd G\wedge \dc G\wedge\omega\wedge  \homg^{n-2}  +  2p^2  \int _X u^{2p-2} e^{G} \dd u\wedge \dc u\wedge \omega\wedge \homg^n\\
    &\leq \frac{2p^2}{n(n-1)}  \int _X u^{2p-2} e^{G} |\widehat{\nabla} u|^2_\homg (\tr_\homg \omega) \homg^n,
\end{split}
\end{equation}
where we used \eqref{eq_ineq_tr_tr} in the last inequality and Cauchy--Schwarz inequality in the third line as follows
\begin{align*}
    &-2pu^{2p-1} \dd u\wedge \dc G\wedge \omega\wedge\homg^{n-2}  
    = -2pu^{2p} \dd \log u\wedge \dc G\wedge \omega\wedge\homg^{n-2}\\
    &\qquad \leqslant \frac{1}{2} u^{2p}  \dd G\wedge \dc G\wedge \omega\wedge\homg^{n-2} +  2p^2 u^{2p}  \dd \log u\wedge \dc \log u\wedge \omega\wedge\homg^{n-2} \\
    &\qquad = \frac{1}{2} u^{2p}  \dd G\wedge \dc G\wedge \omega\wedge\homg^{n-2} +  2p^2 u^{2p-2}  \dd u\wedge \dc u\wedge \omega\wedge\homg^{n-2}. 
\end{align*}
Therefore, with \eqref{eq:int_lap_term1_est1}, \eqref{eq:int_lap_term1_est2}, and \eqref{eq:int_lap_term1_est3}, we derive that 
\begin{equation}\label{eq:int_lap_term1}
\begin{split}
    (\RN{1}) 
    &= \int_Xu^{2p}e^{G}    \langle \ddc \wF, \omega \rangle_\homg \homg^n \\
    &\leq \lt(-\bar s+ \frac{n(n-1)b}{a}\rt) \int_X u^{2p} e^G (\tr_{\homg}\omega) \homg^n +A \int_X u^{2p}e^{G}(\tr_\homg \omega)^2 \homg^n + \frac{2p^2}{a} \int _X u^{2p-1} |\widehat{\nabla} u|^2_\homg  \homg^n. 
\end{split}
\end{equation}
Combining \eqref{int_zero}, \eqref{eq_est_ric}, \eqref{eq:int_lap_term1} and using $u = e^G \tr_{\homg}\omega$, we obtain
\begin{align*}
    0\geqslant 
    & \left(\frac{b}{2}- 2A\right)  \int_X u^{2p} e^{G}(\tr_{\homg } \omega)^2\homg^n \\
    &+\left(2p- \frac{2p^2}{a} \right)\int_X u^{2p-1} |\widehat{\nabla} u |^2_{\homg} \homg^n+ \left(a\bar s-bn+ \bar s- \frac{n(n+1)b}{a}\right)\int_Xu^{2p+1} \homg^n.
\end{align*}
Taking $a = p$ and $b \gg p$ large enough such that 
\[
   b - Aa - 2B \geq b/2 \geq 1, 
   \quad b/2 - 2A \geq 1,  
   \quad\text{and}\quad
   (a+1)\bar{s} - bn \lt(1 + \frac{n+1}{a}\rt) \leq -1,
\] 
we have 
\begin{equation*}
    0 \geqslant 
    C_1 \int_X u^{2p+1}(\tr_{\homg } \omega) \homg^n - C_2\int_Xu^{2p+1} \homg^n,
\end{equation*}
where $C_1, C_2$ only depend on $p, A, B, C_0, \bar{s}$.
Hence, using the fact that $\|\wF\|_{L^\infty}+ \|\varphi\|_{L^\infty}\leq C_0$ and $\homg^n= e^{\wF}\wom^n $, we infer that  $\|G\|_{L^\infty}\leq C(a, b, C_0)$ and so
\begin{equation*}
0\geqslant C_3   \int_X( \tr_{\homg}\omega)^{2p+2} \wom^n -C_4 \int_X( \tr_{\homg}\omega)^{2p+1} \wom^n. 
\end{equation*}
where $C_3, C_4$ only depend on $p, A,B, C_0$.
Then H\"older's inequality implies that 
\[
    \|\tr_{\homg} \omega \|_{L^{2p+2}(X, \wom^n)}\leq C(n,p, A,  B, C_0).
\]
Finally, from \eqref{eq:strict_positive}, as one has a uniform constant $c>0$ such that $\wom_t \geq c \om_t$ for all $t \in \BD$, the above estimate implies  
\[
    \|\tr_{\homg_t} \om_t\|_{L^{2p+2}(X_t,\om_t)} \leq C'(n,p,A,B,C_0,c)
\]
and this completes the proof.
\end{proof}

Now, we use the local estimate of Chen--Cheng to get the $\CC^l$-estimates for $\varphi_t$ away from $\CZ$.

\begin{proof}[Proof of Theorem \ref{thm_C_2}]
Let $v_t$ be the local potential of $\omega_t$ in a neighborhood of $K$. 
Denote by $w_t:=v_t+\varphi_t$, and $H_t:= F_t+ \log\det (\omega_{t,i\bar j})$. 
The cscK-equation \eqref{eq_cscK_t} becomes
\begin{equation*}
    \det(w_{i\bar j})=e^{H} 
        \quad \text{and}\quad
        \Delta_{w} H= -\bar s.
\end{equation*}
By Theorem~\ref{thm_L_inf} and Proposition~\ref{C_2_int}, for any fixed $p>3n(n-1)$,  there exists a constant $C_1>0$ such that for all $t$ sufficiently close to 0, we have
\[
    \|w_t\|_{L^\infty(K)} + \| \Delta w_t\|_{L^p(K)} + \|\sum_{j=1}^n\frac{1}{(w_{t}){}_{j\bar j}}\|_{L^p(K)} \leq C_1.
\]  
Then it follows from the local estimate due to Chen--Cheng (Prop. 6.1 arXiv version of \cite{Chen_Cheng_2021_1}) that  
$\|w\|_{\CC^2(K)}\leq C(p,K, C_1)$ and $\|\nabla H\|_{\CC^0(K)}\leq C(p,K, C_1)$. 
Then we use Evans--Krylov's theorem to get the estimates for all orders $\|w\|_{\CC^l(K)}\leq C(p,l,K, C_1)$ as desired.
\end{proof}

\subsection{Constructing a singular cscK metric from the convergence}

From the uniform $L^\infty$ and local $\CC^l$ estimates above, the Arzelà--Ascoli theorem implies that there exists a sequence $(\varphi_{t_k})_{t_k}$ converging smoothly in families to $\varphi_0 \in \PSH(X_0,\om_0) \cap L^\infty(X_0)$ as $t_k \to 0$, and $\om_{0,\varphi_0}$ is a singular cscK metric on $X_0$ in the sense of Definition \ref{defn:cscK}.

\begin{proof}[Proof of Theorem \ref{thm:Main Theorem}]
    We already proved that $(X_0,\om_0)$ admits a cscK potential $\vph_0$, and we also know that any sequence of cscK potentials $\vph_{t_k}$ in $(X_{t_k},\om_{t_k})$ strongly subconverges to $\vph_0$ as $t_k\to 0$. 
    It remains to show that $\vph_0$ is a minimizer for $\M_0$.
    
    Let $u\in \CE^1(X_0,\om_0)$. 
    We need to show that $\M_0(\vph_0) \leq \M_0(u)$. 
    Without loss of generality, one may assume $\H_0(u)<+\infty$. 
    Then by Lemma \ref{lem:BDL17}, we construct a sequence $u_j\in \CE^1(X_0,\om_0)\cap L^\infty(X_0)$ such that $\M_0(u_j)\to \M_0(u)$. Moreover, $u_j$ is given by Lemma~\ref{lem:regularization_fini_entropy}, and Lemma \ref{lem:cptness_appox_E_conv} gives a sequence $u_{j,k}\in \CE^1(X_k,\om_k)$ such that $u_{j,k}$ converges strongly in family to $u_j$ and such that $\H_k(u_{j,k})\to \H_0(u_j)$ as $k\to +\infty$. 
    Indeed, the last assertion follows from \eqref{eq:fcn_beta_parameter}. In particular $\M_k(u_{j,k})\to \M_0(u_j)$ as $k\to +\infty$ and we gain
    $$
    \M_0(\vph_0)\leq \liminf_{k\to+\infty} \M_k(\vph_k)\leq \liminf_{k\to+\infty}\M_k(u_{j,k})=\M_0(u_j)
    $$
    where the first inequality is given by Proposition \ref{prop:Lower Semicontinuity Mabuchi families} and where we used that $\vph_k$ is a minimizer for $\M_k$. 
    Letting $j\to +\infty$, we obtain $\M_0(\vph_0)\leq \M_0(u)$, which concludes the proof.
\end{proof}

\subsection{Remark on examples}
In this section, we provide a way to build some examples of singular varieties that admit singular cscK metrics with our results.
We first review some general results on the deformation of K\"ahler spaces due to Bingener \cite{Bingener_1983}. 
Denote by $A_X$ the sheaf of real analytic functions on $X$. 
Similar to Section~\ref{ssec:Proposition B}, we have a short exact sequence of sheaves
\[
    0 \to \PH_X \to A_X \to A_X/\PH_X \to 0.
\]
Recall that a smooth K\"ahler metric is canonically attached as an element in $H^0(X,\CC^\infty_X/\PH_X)$. 
Then a real analytic K\"ahler metric is defined as a smooth K\"ahler metric belonging to $H^0(X, A_X/\PH_X)$.
The following is a result extracted from the proofs of \cite[Thm.~4.7 \& Cor~4.8]{Bingener_1983}:

\begin{lem}\label{lem:Bingener_cor4.8}
On a compact K\"ahler variety $X$, for any class $\af$ in the K\"ahler cone $\CK_X \subset H^1(X,\PH_X)$, there exists a real analytic K\"ahler metric inside $\af$.  
\end{lem}

\begin{proof}
As in \cite[bottom of p.~522]{Bingener_1983}, we have the following commutative diagram
\[
\begin{tikzcd}
    0 \ar[r]& H^0(X,\PH_X) \ar[r]\ar[d, "\Id"]& H^0(X,A_X) \ar[r] \ar[d, "f_1"]& H^0(X, A_X/\PH_X) \ar[r] \ar[d, "f_2"]& H^1(X,\PH_X) \ar[r]\ar[d, "\Id"]& 0\\
    0 \ar[r]& H^0(X,\PH_X) \ar[r]& H^0(X,\CC^\infty_X) \ar[r]& H^0(X,\CC^\infty_X/\PH_X) \ar[r]& H^1(X,\PH_X) \ar[r]& 0
\end{tikzcd}
\]
where each row is exact and $f_1, f_2$ are natural inclusions.
Indeed, from \cite{Acquistapace_Broglia_Tognoli_1979}, $X$ can be embedded real analytically into $\BR^N$ for a sufficiently large $N$; then the vanishing of $H^1(X, A_X)$ comes from a real analytic version of Cartan's theorem B \cite[Thm.~3]{Cartan_1957} (cf. \cite[top of p.~523]{Bingener_1983}). 
Let $\om \in \af$ be a smooth K\"ahler metric. 
From the first row, there is a real analytic representative $\gm \in \af$ and thus, $[\om - \gm] = 0 \in H^1(X,\PH_X)$.
Therefore, from the second row, $\om = \gm + \ddc h$ for some $h \in \CC^\infty(X)$. 
Since $A_X(X)$ is dense in $\CC^\infty(X)$, there exists a sequence of real analytic functions $(h_m)_m$ converges smoothly to $h$.
Then, by the positivity of $\om$, one can obtain a real analytic K\"ahler metric $\om' = \gm + \ddc h_m$ for some $m$ sufficiently large, and it shows a real analytic K\"ahler metric in $\af$.
\end{proof}

The following theorem is a special case of \cite[Thm.~6.3]{Bingener_1983}: 

\begin{thm}\label{thm:Bingener_thm6.3}
Under Setting~\ref{sett:klt}, let $\om_0$ be a real analytic K\"ahler metric on $X_0$.
Up to shrink $\BD$, there exists a (real analytic) hermitian metric $\om$ on $\CX$ such that $\om_{|X_0} = \om_0$ and $\om_t := \om_{|X_t}$ is K\"ahler for all $t \in \BD$.
\end{thm}

The cohomological condition "$f: H^2(X_0,\BR) \to H^2(X_0,\CO_{X_0})$ is surjective" in \cite[Thm.~6.3]{Bingener_1983} is valid for klt singularities where $f$ is the morphism induced by the short exact sequence of sheaves $0 \to \underline{\BR} \to \CO_X \xrightarrow[]{\Im(\cdot)} \PH_X \to 0.$ 
Indeed, in \cite[Rmk.~3.2~(2)]{Graf_Kirschner_2020}, the surjectivity of $f$ holds for varieties with rational singularities and, in particular, klt singularities are rational. 

\smallskip
We now provide a combination of Theorem~\ref{bigthm:openness_classes} and Theroem~\ref{bigthm:smoothable_variety} which will help to provide some examples:

\begin{cor}\label{cor:fano_cscK}
Suppose that $\pi: \CX \to \BD$ is a $\BQ$-Gorenstein smoothing of $X_0$ a K-stable Fano variety.
Then there exist an open subset $U \subset \CK_X$ containing $c_1(-K_{X_0})$ such that each class $\af \in U$ containing a singular cscK metric.
\end{cor}

\begin{proof}
Let $\om_0$ be a K\"ahler metric in $c_1(-K_{X_0})$. 
By Lemma~\ref{lem:Bingener_cor4.8} and Theorem~\ref{thm:Bingener_thm6.3}, up to shrinking $\BD$, there exists a real analytic hermitian metric $\om'$ such that $\om'_{|X_0} \in [\om_0]$ and $\om_t = \om'_{|X_t}$ is K\"ahler on $X_t$ for each $t \in \BD$.
Since $X_0$ is K-stable, by \cite{Li_Tian_Wang_2022, Liu_Xu_Zhuang_2022}, there is a unique singular K\"ahler--Einstein metric on $X_0$.
From \cite[Prop.~4.11]{BBEGZ_2019}, \cite[Thm.~5.5]{Dinezza_Guedj_2018}, $\M_{\om'_0}$ is coercive. 
Theorem~\ref{bigthm:openness_classes} shows that there exits an open subset $U$ in $\CK_{X_0}$ containing $c_1(-K_{X_0})$ such that $\M_{\gm_0}$ is coercive for any $\af \in U$ where $\gm_0$ is a smooth K\"ahler metric in $\af$.
Again, by Lemma~\ref{lem:Bingener_cor4.8} and Theorem~\ref{thm:Bingener_thm6.3}, up to shrinking $\BD$, there exists a real analytic hermitian metric $\gm'$ such that $\gm'_{|X_0} \in [\gm_0] = \af$ and $\gm'_t = \gm'_{|X_t}$ is K\"ahler on $X_t$ for each $t \in \BD$. 
Since $\M_{\gm'_0}$ is also coercive, Theorem~\ref{bigthm:smoothable_variety} shows the existence of a singular cscK metric inside $[\gm'_0] = \af$ for all $\af \in U$. 
\end{proof}

Finally, we extract an example from \cite{Odaka_Spotti_Sun_2016}. 
Consider $X_0$ a cubic surfaces in $\BP^3$ with $A_1$-singularities.
Note that $X_0$ has a unique K\"ahler--Einstein metric (cf. \cite[Example~1.16]{Cheltsov_Kosta_2014}, see also \cite[middle of p.~165]{Odaka_Spotti_Sun_2016} for other examples). 
It follows from \cite{Kollar_Shepherd_1988, Hacking_Prokhorov_2010} that a normal complex surface with a T-singularity (i.e. either Du Val (A-D-E type), or quotient of Du Val singularity $A_{dn-1}$ by $\BZ_n$) admits a $\BQ$-Gorenstein smoothing, so does $X_0$.  
Since $X_0$ has a Picard number greater than $1$, by Corollary~\ref{cor:fano_cscK} we have a singular cscK metric in classes near $c_1(-K_{X_0})$ in the K\"ahler cone and most of them are not scaling of singular K\"ahler--Einstein metrics. Corollary~\ref{cor:fano_cscK} can provide more examples of singular cscK metrics also in higher dimensions: it is enough to pick a K-stable Fano variety of Picard number greater than $1$ that admits a $\BQ$-Gorenstein smoothing.

\bibliographystyle{smfalpha_new}
\bibliography{biblio}
\end{document}